\documentclass[11pt]{amsart}

\input{LouisTimoDefs.tex}

 \RequirePackage{datetime} 

\allowdisplaybreaks[1]
\begin{document}
 
\title[Corner growth model]{Joint distribution of Busemann functions in the exactly solvable corner growth model}

\author[W.-T.~Fan]{Wai-Tong (Louis) Fan}
\address{Wai-Tong (Louis) Fan\\ Indiana University\\  Mathematics Department\\ Rawles Hall, 831 E. Third St.
	\\  Bloomington, IN 47405\\ USA.}
\email{waifan@iu.edu}
\urladdr{https://sites.google.com/site/louisfanmath/}
\thanks{Part of this research was done while W.-T.\ Fan was a Van Vleck visiting assistant professor at UW-Madison. W.-T.\ Fan was partially supported by National Science Foundation grant DMS-1804492.}

\author[T.~Sepp\"al\"ainen]{Timo Sepp\"al\"ainen}
\address{Timo Sepp\"al\"ainen\\ University of Wisconsin-Madison\\  Mathematics Department\\ Van Vleck Hall\\ 480 Lincoln Dr.\\   Madison WI 53706-1388\\ USA.}
\email{seppalai@math.wisc.edu}
\urladdr{http://www.math.wisc.edu/~seppalai}
\thanks{T.\ Sepp\"al\"ainen was partially supported by  National Science Foundation grants  DMS-1602486 and DMS-1854619 and by  the Wisconsin Alumni Research Foundation.} 

\keywords{Busemann function, Catalan number, Catalan triangle, cocycle, competition interface, corner growth model,  
directed percolation, geodesic, last-passage percolation, M/M/1 queue, multiclass fixed point, rhobar distance, TASEP}
\subjclass[2000]{60K35, 65K37} 
\date{\today} 

\begin{abstract}
The 1+1 dimensional corner growth model with exponential weights is a centrally important  exactly solvable model in the Kardar-Parisi-Zhang class of statistical mechanical models.  While significant progress has been made on the fluctuations of the growing random shape, understanding of the optimal paths, or geodesics,  is less developed.  The Busemann function is a useful analytical tool for studying geodesics.  This paper describes the joint distribution of the Busemann functions, simultaneously in all directions of growth.   As applications of this description  we derive a marked point process representation for the Busemann function across a single lattice edge  and calculate some marginal distributions of Busemann functions and semi-infinite geodesics.  
\end{abstract}
\maketitle

\maketitle
\tableofcontents

\section{Introduction}

\subsection*{The corner growth model in the Kardar-Parisi-Zhang class}
The planar corner growth model (CGM)   is a directed last-passage percolation (LPP) model on the planar integer lattice $\Z^2$ whose paths are allowed to take nearest-neighbor steps $\evec_1$ and $\evec_2$.   In the exactly solvable case the  random weights attached to the vertices of $\Z^2$   are i.i.d.\ exponentially or geometrically distributed random variables.

The exact solvability of the exponential and geometric CGM 
has been fundamental to the 20-year progress in the study of the 
1+1 dimensional  Kardar-Parisi-Zhang (KPZ) universality class.  After the initial breakthrough by Baik, Deift and Johansson \cite{baik-deif-joha-99} on planar LPP on Poisson points, the Tracy-Widom limit of the geometric and exponential CGM followed in Johansson's article \cite{joha}.  This work relied on techniques that today would be called {\it integrable probability}, a subject that applies ideas from representation theory and integrable systems  to study stochastic models.  A large literature has followed.  Recent reviews appear in  \cite{corw-18-bull,corw-16-rev,corw-12-rev}. 

A different  line of work was initiated in \cite{bala-cato-sepp} that gave a probabilistic proof of the KPZ exponents of the exponential CGM,   following the seminal work  \cite{cato-groe-06} on the planar Poisson LPP. The proof utilized  the tractable stationary version of the CGM and developed estimates  by  coupling perturbed versions of the CGM process.    This opening  led to the first proofs of   KPZ exponents  for the asymmetric simple exclusion process (ASEP) \cite{bala-sepp-aom} and the  KPZ equation \cite{bala-quas-sepp}, to the discovery of the first exactly solvable positive-temperature lattice  polymer model  \cite{sepp-12-aop-corr}, to a  proof of KPZ exponents for a class of zero-range processes outside known exactly solvable models \cite{bala-komj-sepp-12},  and most recently to Doob transforms and martingales  in random walk in random environment (RWRE)  that manifest KPZ  behavior \cite{bala-rass-sepp-18-arxiv}.  The estimates from  \cite{bala-cato-sepp} have also been applied to coalescence times of geodesics \cite{pime-16} and to the local behavior of Airy processes \cite{pime-17-arxiv}. 

\subsection*{Joint distribution of Busemann functions} 
The present article places the stationary CGM into a larger context   by describing the natural coupling of all the stationary CGMs.   This coupling arises from the  joint distribution of the Busemann functions in all directions of  growth.  

 Let $G_{x,y}$ denote the last-passage value between points $x$ and $y$ on the lattice $\Z^2$ (precise definition follows in \eqref{m:G807} in Section 2).   The Busemann function $\Bus^\rho_{x,y}$ is the limit of increments $G_{v_n,y}-G_{v_n,x}$  as $v_n$ is taken to infinity in the direction parametrized by $\rho$.   In a given direction this limit exists almost surely.   These limits  are extended to a  process $B^{\rcbullet}$ by taking limits in the parameter $\rho$.  
 Finite-dimensional  distributions  of $B^\rcbullet$ are    identified   as the unique invariant distributions of  multiclass LPP processes.  These distributions are conveniently described in terms of mappings that represent FIFO (first-in-first-out) queues.     Key points of the development  are (i)  an intertwining between two types of multiclass processes, called the {\it multiline process} and the {\it coupled process}, and (ii)  a triangular array representation of the intertwining mapping.  
 
 Properties of the joint Busemann process $B^\rcbullet$ discovered here (in particular,  Theorem \ref{B-th5}) are applied in a forthcoming work \cite{janj-rass-sepp-future} to describe the overall structure of the geodesics of the exponential corner growth model.   The joint distribution is a necessary tool for a full picture because in a fixed direction semi-infinite geodesics are almost surely unique and coalesce (these facts are reviewed in Section \ref{s:geod} below) but   there are exceptional random directions of non-uniqueness.  Through the joint distribution one  captures the jumps of the Busemann function as the direction varies.  These   correspond to jumps in coalescence points and non-uniqueness of geodesics.  
  A separate line of   work \cite{fan-sepp-future} develops the joint distribution of Busemann functions for the positive-temperature log-gamma polymer model.

%

\subsection*{An analogue of the Busemann function on the Airy sheet}  
An interesting similarity appears between our paper and recent work on the universal objects that arise from LPP.
 Basu, Ganguly and Hammond   \cite{basu-gang-hamm-arxiv} study  an analogue of the Busemann function in the Brownian last-passage model.  Instead of the lattice scale and all spatial directions, they look at a difference of last-passage values  on the scale $n^{2/3}$ into a fixed macroscopic direction, 
 where universal objects such as Airy processes arise.  Translated to the CGM, their object of interest is the weak limit $Z(z)$  of the scaled difference 
$n^{-1/3} [ G_{(n^{2/3},0),(n+zn^{2/3},n)} - G_{(-n^{2/3},0),(n+zn^{2/3},n)}   + 4n^{2/3} ] $ 
 that they call the {\it difference weight profile}. 
 In terms of the Airy sheet $\{W(x,y): x,y\in\R\}$ constructed recently by Dauvergne, Ortmann and Vir\'ag \cite{dauv-ortm-vira-arxiv},  the limit $Z(z)= W(1,z)-W(-1,z)$. 
 
  The limit $Z(\cdot)$ is a continuous process, while the Busemann process we construct is a jump process.  But like the Busemann process,  the limit  $Z(\cdot)$ is constant in a neighborhood of each point, except for a small  set of exceptional points.   In both settings  this constancy  reflects the same underlying phenomenon, namely   the coalescence of geodesics.  In our lattice setting  Theorem \ref{B-th5} gives a precise description of these exceptional directions  in terms of an inhomogeneous Poisson process.

\subsection*{Past work}    We mention related past work on queues,  particle systems,  and the CGM.


\smallskip 

{\it Queueing fixed points.}  We formulate a queueing operator as a mapping of  bi-infinite sequences of interarrival times and service times into a bi-infinite sequence of interdeparture times (details in Section \ref{s:queue}).  When the service times are i.i.d.\ exponential (memoryless, or $\cdot/M/1$ queue),  it is classical that i.i.d.\ exponential times are preserved by the mapping from the interarrival process to the interdeparture process, subject to the stability condition that the mean interarrival time exceed the mean  service time.
Anantharam \cite{anan+corr} 
proved the uniqueness of this fixed point  and Chang \cite{MR1303943} gave a shorter argument.  (\cite{MR1325044} cites also an unpublished manuscript of Liggett and Shiga.) 
Convergence to the fixed point was proved in  \cite{MR1325044}. These results were partially extended to  general $\cdot/G/1$ queues in \cite{mair-prab-03,MR2016618}. 

 
We look at LPP processes with multiple classes of input, but this is {\it not} the same as a multiclass queue that serves  customers in different priority classes. 
In queueing terms,  the present paper describes the unique invariant distribution in a situation where a single memoryless queueing operator transforms a vector of interarrival processes into a vector of interdeparture processes.   It is fairly evident  a priori that this operation cannot preserve an  independent collection  of  interarrival processes because they are correlated  after passing through the same queueing operator.  (For example, this operation preserves monotonicity.)   It turns out that the queueing mappings themselves provide a way to describe the structure of the invariant distribution.

\smallskip 

  {\it Multiclass measures for particle systems.} 
 In a series of remarkable papers \cite{ferr-mart-06, ferr-mart-07, ferr-mart-09} P.~A.~Ferrari and J.~B.~Martin developed queueing descriptions of the 
stationary distributions of the multiclass totally  asymmetric simple exclusion process (TASEP) and the Aldous-Diaconis-Hammersley process. 
The intertwining that establishes  our Theorem \ref{th:eta-ex} became possible after the discovery of a way to apply the ideas of Ferrari and Martin  to the CGM.   We use the terms multiline process and coupled process to highlight the analogy with their work.


\smallskip 

  {\it Busemann functions and semi-infinite geodesics.}
Existence and properties of Busemann functions and semi-infinite geodesics are reviewed in Sections  \ref{s:cgm}--\ref{s:geod}. 
Two strategies exist for proving the existence   of Busemann functions for the exponential CGM.      

(i)  Proofs by   Ferrari and Pimentel  \cite{ferr-pime-05} and Coupier  \cite{coup-11} relied  on  C.~Newman's approach to geodesics  \cite{howa-newm-01,lice-newm-96, newm-icm-95}. This strategy  is feasible   because the   exact solvability  shows that the shape function \eqref{gpp} satisfies the required  curvature hypotheses.    
 
 (ii) A direct argument from the stationary growth  model to  the Busemann limit    was introduced in \cite{geor-rass-sepp-yilm-15} for the log-gamma polymer, and applied to the exponential CGM in the lecture notes \cite{sepp-cgm-18}.  An application of this strategy  to the CGM with general i.i.d.\ weights appears in \cite{geor-rass-sepp-17-geod,geor-rass-sepp-17-buse} where the role of the regularity of the shape function becomes explicit.  
 
  A   sampling of other significant work on Busemann functions and geodesics can be found  in \cite{Cat-Pim-13, Fer-Mar-Pim-09,  Cat-Pim-12, Bak-Cat-Kha-14, Hof-05, Hof-08}. 
 

\subsection*{Organization of the paper} 

Section \ref{s:prel}  collects preliminaries on the CGM and queues.  The main results for Busemann functions and semi-infinite geodesics are stated in Section \ref{s:Bus}.  
 Section \ref{s:D} proves a key lemma for the queueing operator.   Section \ref{s:m-class} introduces   the multiline process,  the coupled process, and the multiclass LPP process, and then  states and proves results on their invariant distributions.   
 The key intertwining between the multiline process and the coupled process  appears in equation \eqref{int-twi} in the proof of Theorem \ref{th:eta-ex} in Section \ref{s:inv-cpl}.   Section \ref{s:Bpf} proves the results of Section \ref{s:Bus}.  For the proof of  Theorem \ref{B-th5}, Section \ref{s:array} introduces   a triangular array  representation  for the intertwining mapping.   Auxiliary matters on queues and exponential distributions are relegated to Appendices \ref{a:queue} and \ref{a:exp}.

 \subsection*{Notation and conventions}  
Points $x=(x_1,x_2),y=(y_1,y_2)\in\R^2$ are ordered coordinatewise: $x\le y$ iff  $x_1\le y_1$  and $x_2\le y_2$.    The $\ell^1$ norm is $\abs{x}_1=\abs{x_1}+\abs{x_2}$.   Subscripts indicate restricted subsets of the reals and integers: for example 
$\Z_{>0}=\{1,2,3,\dotsc\}$.    Boldface  notation for vectors: $\evec_1=(1,0)$, $\evec_2=(0,1)$,   and members of the simplex $[\evec_2,\evec_1]=\{t\evec_1+(1-t)\evec_2: 0\le t\le 1\}$ are denoted by $\uvec$. 

For $n\in\Z_{>0}$, $[n]=\{1,2,\dotsc,n\}$, with the convention that $[n]=\varnothing$ for $n\in\Z_{\le0}$.  A finite  integer interval is denoted by $\lzb m,n\rzb=\{m,m+1,\dotsc,n\}$,  and  $\lzb m,\infty\lzb=\{m,m+1, m+2, \dotsc\}$.  
  
  For $0<\alpha<\infty$, $X\sim$ Exp$(\alpha)$ means that random variable $X$ has exponential distribution with rate $\alpha$, in other words $P(X>t)=e^{-\alpha t}$ for $t>0$ and $E(X)=\alpha^{-1}$.   In the discussion   we parametrize exponential variables with their mean.   For $0<\rho<\infty$,  $\nu^\rho$ is the probability distribution on  the space $\R_{\ge0}^\Z$ of  bi-infinite sequences under which the coordinates are i.i.d.\ exponential variables with common mean $\rho$.  Higher-dimensional product measures are denoted by $\nu^{(\rho_1, \rho_2,\dotsc,\rho_n)}= \nu^{\rho_1}\otimes\nu^{\rho_2}\otimes\dotsm\otimes\nu^{\rho_n}$. 
  
  For $0\le p\le 1$, $X\sim$ Ber$(p)$ means that random variable $X$ has Bernoulli distribution with parameter $p$,  in other words $P(X=1)=p=1-P(X=0)$. 

In general, $E^\mu$ represents expectation under a measure $\mu$.

\section{Preliminaries}  \label{s:prel} 


Section \ref{s:cgm} below introduces the main objects of discussion:  the planar corner growth model (CGM), which is a special case of  last-passage percolation (LPP), and Busemann functions.  Section \ref{s:geod} explains the significance of Busemann functions in the description of directed semi-infinite geodesics and the asymptotic direction of the competition interface.  The  somewhat technical Section \ref{s:queue} defines  FIFO (first-in-first-out) queueing mappings that are used in Section \ref{s:Bus} to describe the joint distribution of the Busemann functions.  To be sure, the distribution of the Busemann functions could be described by plain mathematical formulas without their queueing content. But the queueing context gives the mathematics meaning that can help comprehend the results.  

\subsection{Busemann functions in the corner growth model} 
\label{s:cgm}

The  setting for  the exponential CGM  is the  following.    $\OSP$ is  a probability space with generic sample point $\w$.   A  group of measure-preserving measurable bijections   $\{\theta_x\}_{x\in\Z^2}$ acts on $\OSP$.   Measure preservation means that $\P(\theta_xA)=\P(A)$ for all sets $A\in\kS$ and $x\in\Z^2$.    $\Yw=(\Yw_x)_{x\in\Z^2}$ is a random field of   independent and identically distributed  {Exp}(1) random weights defined on $\Omega$ that satisfies $\Yw_x(\theta_y\w)=\Yw_{x+y}(\w)$ for  $x,y\in\Z^2$ and $\w\in\Omega$.  

The canonical choice for the sample  space is the product space $\Omega=\R_{\ge0}^{\Z^2}$ with its Borel $\sigma$-algebra $\kS$,    generic sample point $\w=(\w_x)_{x\in\Z^2}$,  translations $(\theta_x\w)_y=\w_{x+y}$, and coordinate random variables $\Yw_x(\w)=\w_x$.  Then  $\P$ is the  i.i.d.\ product measure  on $\Omega$ under which each $\Yw_x$ is an Exp(1) random variable.   

\begin{figure}[t]  
	\begin{center}
		\begin{minipage}{0.45\textwidth}
			\begin{tikzpicture}[>=latex, scale=0.7]  
			
			\definecolor{sussexg}{gray}{0.7}
			\definecolor{sussexp}{gray}{0.4} 

			\draw[->] (0,0)--(0,-5);
			\draw[->] (0,0)--(-7,0);
			
			\draw[line width = 3.2 pt, color=sussexp](0,0)--(0,-2)--(-1,-2)--(-1,-3)--(-3,-3)--(-4,-3)--(-4,-4)--(-6,-4);
			
			\foreach \x in { 0,...,-6}{
				\foreach \y in {0,...,-4}{
					\fill[color=white] (\x,\y)circle(1.8mm); 
					
					\draw[ fill=sussexg!65](\x,\y)circle(1.2mm);
				}
			}	
			
			\foreach \x in {0,...,-6}{\draw(\x, 0.3)node[above]{\small{$\x$}};}
			\foreach \x in {0,...,-4}{\draw(0.3, \x)node[right]{\small{$\x$}};}	
			
			\end{tikzpicture}
		\end{minipage}
		\begin{minipage}{0.28\textwidth} 
			
			\begin{tikzpicture}[>=latex,scale=0.7]
			\draw[->,line width=1.5pt] (0,0)--(-3,0);
			\draw[->,line width=1.5pt] (0,0)--(0,-3);
			\draw[->,line width=1.5pt] (0,0)--(-2,-1);
			\draw[dashed] (-3,0)--(0,-3);
			\draw(-3,0.48)node{\large$-\evec_1$};
			\draw(0.8,-2.8)node{\large$-\evec_2$};
			\draw(-2.7,-1.3)node{\large $\uvec(\rho)$};
			\draw(-4,-0.2)node{\large$\rho=1$};
			\draw(0,-3.6)node{\large$\rho=\infty$};
			\end{tikzpicture}
		\end{minipage}		
	\end{center}
	\caption{ \small Left: The thickset line segments define an element of $\Pi_{(-6,-4),(0,0)}$. Right:  As the parameter  $\rho$ increases from 1 to $\infty$,  vector $\uvec(\rho)$ of \eqref{u-rho} sweeps the directions from $-\evec_1$ to $-\evec_2$  in the third quadrant.}\label{fig:cgm} 
\end{figure}
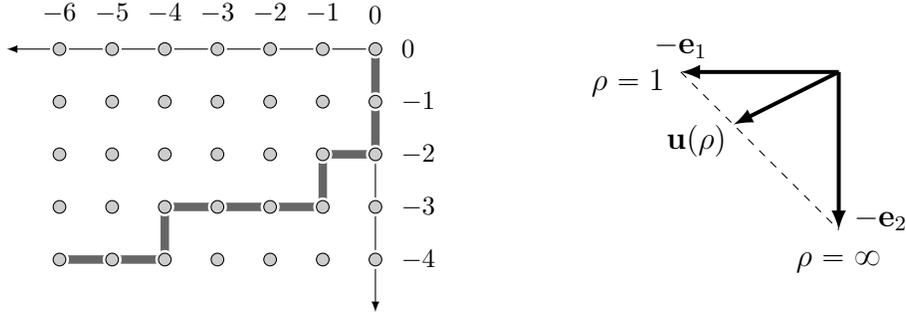

%
%
%
%
%
%
%
%
%
%

   For $u\le v$ on $\Z^2$ (coordinatewise ordering)  let   $\Pi_{u,v}$ denote  the set of up-right paths $x_{\cbullet}=(x_i)_{i=0}^{\abs{v-u}_1}$ from $x_0=u$  to $x_{\abs{v-u}_1}=v$  with steps $x_i-x_{i-1}\in\{\evec_1,\evec_2\}$.  (The left diagram of Figure \ref{fig:cgm} illustrates.)       Define the last-passage percolation (LPP) process 
   \be\label{m:G807} 
G_{u,v} =\max_{x_\bbullet\in\Pi_{u,v}} \sum_{i=0}^{\abs{v-u}_1} \Yw_{x_i}  
\quad\text{for $u\le v$ on $\Z^2$.} 
\ee
For $v\in u+\Z_{>0}^2$ we have the inductive equation 
\be\label{G-ind18} G_{u,v}=G_{u,v-\evec_1}\vee G_{u,v-\evec_2}+\Yw_v  . \ee  

 The convention of this paper is that growth proceeds in the {\it southwest} direction (into the {\it third quadrant}  of the plane).  
 Thus the well-known shape theorem (Theorem 5.1 in \cite{mart-04}, Theorem 3.5 in \cite{sepp-cgm-18}) of the CGM takes the following  form.   With probability one,  
\be\label{sh-thm} 
\lim_{r\to\infty} \; \sup_{x\in(\Z_{\le0})^2: \, \abs{x}_1\ge r} \frac{\abs{G_{x,0}- \gpp(x)}}{\abs{x}_1}  =0  
\ee
with the concave, continuous and one-homogeneous shape function (known since \cite{rost})  
\be\label{gpp}  \gpp(x)=\bigl( \sqrt{\abs{x_1}} +\sqrt{\abs{x_2}}\;\bigr)^2 
\quad\text{for} \ \ x=(x_1,x_2)\in\R_{\le 0}^2.  
\ee 

Busemann functions are limits of differences $G_{v,x}-G_{v,y}$  of last-passage values from two fixed points $x$ and $y$  to a common point  $v$ that is taken to infinity in a particular direction.   These limits are described by relating the direction $\uvec$ that $v$ takes to a real parameter $\rho$ that specifies the distribution of the limits:   a bijective  mapping  
  between directions  $\uvec=(u_1,u_2)\in\;] \!-\evec_1,-\evec_2[$ in the open third quadrant of the plane and parameters $\rho\in(1,\infty)$ is defined by the equations 
 \be\label{u-rho}    \uvec=\uvec(\rho)=-\,\biggl(  \frac1{1+(\rho-1)^2}\,,\,  \frac{(\rho-1)^2}{1+(\rho-1)^2} \biggr) 
 \ \Longleftrightarrow \ 
 \rho= \rho(\uvec)= \frac{\sqrt{-u_1}+ \sqrt{-u_2}}{\sqrt{-u_1}}. 
  \ee
(See the right diagram of Figure \ref{fig:cgm} for an illustration.)

  The existence and properties of Busemann functions are summarized in the following theorem.    By definition,  a {\it down-right lattice path}  $\{y_k\}$ satisfies  $y_k-y_{k-1}\in\{\evec_1, -\evec_2\}$ for all $k$. 
  
\begin{theorem} \label{B-th-zero} On the probability space $\OSP$ there exists a cadlag  process $\Bus^\rho=(\Bus^\rho_{x,y})_{x,y\in\Z^2}$  with state space $\R^{\Z^2\times\Z^2}$,    indexed by ${\rho\in(1,\infty)}$, with the following properties.  

\smallskip 

{\rm (i) Path properties.}   There is a single event $\Omega_0$ such that $\P(\Omega_0)=1$ and the following properties hold for all $\w\in\Omega_0$, for all $\lambda, \rho\in(1,\infty)$ and $x,y,z\in\Z^2$.  
\be\label{B-mono}  \text{If}\quad \lambda<\rho\quad\text{then} \quad \Bus^\lambda_{x,x+\evec_1}\le \Bus^\rho_{x,x+\evec_1}
\quad\text{and}\quad 
 \Bus^\lambda_{x,x+\evec_2}\ge \Bus^\rho_{x,x+\evec_2}.   
 \ee
  \be\label{B-add}    \Bus^\rho_{x,y}+\Bus^\rho_{y,z}=\Bus^\rho_{x,z}.  \ee
\be\label{B-reco}   \Yw_x =  \Bus^\rho_{x-\evec_1,x}\wedge\Bus^\rho_{x-\evec_2,x}.   
\ee
Cadlag property:  the path $\rho\mapsto \Bus^\rho_{x,y}$ is right continuous and has left limits.  

\smallskip 

{\rm (ii) Distributional properties.}  
Each process $\Bus^\rho$ is stationary under lattice shifts.  
The marginal distributions of nearest-neighbor increments are 
\be\label{B-distr}    \Bus^\rho_{x-\evec_1,x}\sim {\rm Exp}(\rho^{-1}) \quad\text{and}\quad  
 \Bus^\rho_{x-\evec_2,x} \sim {\rm Exp}(1-\rho^{-1}) . \ee
  Along any down-right path $\{y_k\}_{k\in\Z}$ on $\Z^2$, for fixed $\rho\in(1,\infty)$ the increments $\{\Bus^\rho_{y_k,y_{k+1}}\}_{k\in\Z}$ are independent.  

\smallskip 

{\rm (iii) Limits.}  Fix $\rho\in(1,\infty)$ and let $\uvec=\uvec(\rho)$ be the vector determined by \eqref{u-rho}.  Then there exists an event $\Omega_0^{(\rho)}$ such that $\P(\Omega_0^{(\rho)})=1$ and the following  holds:  
for any sequence $\{u_n\}$ in $\Z^2$  such that  $\abs{u_n}_1\to\infty$ and $u_n/n\to \uvec$  and for any  $\w\in\Omega_0^{(\rho)}$,  
\be\label{B-Glim}
\Bus^\rho_{x,y}=\lim_{n\to\infty} [G_{u_n, y}-G_{u_n,x}]  . 
\ee
Continuity from the left   at a fixed $\rho\in(1,\infty)$ holds with probability one:    
 $\ddd\lim_{\lambda\nearrow\rho} \Bus^\lambda_{x,y}=\Bus^\rho_{x,y} $ almost surely.  
\end{theorem}  

The theorem above is proved as Theorem 4.2 in lecture notes  \cite{sepp-cgm-18}. 
The central point of the theorem is the limit \eqref{B-Glim}, on account of which we  call $\Bus^\rho$ the {\it Busemann function} in direction $\uvec$.    We record some observations. 

Additivity \eqref{B-add} implies  that $\Bus^\rho_{x,x}=0$ and $\Bus^\rho_{x,y}=-\Bus^\rho_{y,x}$.  
 The {\it weights recovery} property \eqref{B-reco} can be seen   from  \eqref{G-ind18} and  limits \eqref{B-Glim}: 
\begin{align*}
\Bus^\rho_{x-\evec_1,x}\wedge\Bus^\rho_{x-\evec_2,x}&= \lim_{n\to\infty} [G_{u_n, x}-G_{u_n,x-\evec_1}]  \wedge [G_{u_n, x}-G_{u_n,x-\evec_2}] \\
&=  \lim_{n\to\infty} [G_{u_n, x}-G_{u_n,x-\evec_1}\vee G_{u_n,x-\evec_2}] =\Yw_x. 
 \end{align*}


\begin{lemma}\label{B-reco-lm}  With probability one,  $\forall x\in\Z^2$ there exists   random $\rho^*(x)\in(1,\infty)$  such that
\be\label{B-reco3}\begin{aligned}  
\Bus^\rho_{x-\evec_1,x}&=\Yw_x<\Bus^\rho_{x-\evec_2,x} \quad \text{for} \quad \rho\in(1,\rho^*(x))\\
\text{and}\quad \Bus^\rho_{x-\evec_2,x}&=\Yw_x<\Bus^\rho_{x-\evec_1,x} \quad \text{for} \quad\rho\in(\rho^*(x),\infty).  
\end{aligned}\ee 
The distribution function of $\rho^*(x)$ is $\P\{ \rho^*(x)\le \lambda\}=1-\lambda^{-1}$ for $1\le\lambda<\infty$. 
\end{lemma}

\begin{proof} 
 Monotonicity \eqref{B-mono} and the exponential rates \eqref{B-distr} force $\Bus^\rho_{x-\evec_2,x}\nearrow\infty$ almost surely as $\rho\searrow 1$ and $\Bus^\rho_{x-\evec_1,x}\nearrow\infty$ almost surely as $\rho\nearrow\infty$.    Edges $\{x-\evec_1,x\}$ and $\{x-\evec_2,x\}$ are part of a down-right path, and hence  $\Bus^\rho_{x-\evec_1,x}$ and $\Bus^\rho_{x-\evec_2,x}$ are independent  exponential random variables for each fixed $\rho$.  Consequently, with probability one,  they are  distinct for each rational $\rho>1$.  By monotonicity again there is a unique real $\rho^*(x)\in(1,\infty)$  such that for rational $\lambda\in(1,\infty)$,
\be\label{B-reco6}\begin{aligned}  
\lambda<\rho^*(x)&\text{ implies }    \Bus^\lambda_{x-\evec_1,x}<\Bus^\lambda_{x-\evec_2,x}, \\
\text{and}\quad \lambda>\rho^*(x)&\text{ implies }    \Bus^\lambda_{x-\evec_1,x}>\Bus^\lambda_{x-\evec_2,x}. 
\end{aligned}\ee 
By monotonicity the same holds for real $\lambda$.  
\eqref{B-reco3} follows from weights recovery  \eqref{B-reco}.   The distribution function comes from \eqref{B-reco3}, independence of $\Bus^\rho_{x-\evec_1,x}$ and $\Bus^\rho_{x-\evec_2,x}$, and \eqref{B-distr}. 
\end{proof}

In particular, for a fixed $x$ the processes 
$\{\Bus^\rho_{x-\evec_1,x}\}_{1<\rho<\infty}$ and $\{\Bus^\rho_{x-\evec_2,x}\}_{1<\rho<\infty}$ 
are not independent of each other, even though for a fixed $\rho$ the random variables $\Bus^\rho_{x-\evec_1,x}$ and $\Bus^\rho_{x-\evec_2,x}$ are independent.    
Vector $\uvec(\rho^*(x))$ is the asymptotic direction of the competition interface emanating from $x$ (see Remark \ref{r:cif} further below).

The  process $\Bus=\{\Bus^\rho_{x,y}\}$  is a Borel function of the weight configuration $Y$.  Limits \eqref{B-Glim} define $\Bus^\rho$ as a function of $Y$ for a countable dense set of $\rho$ in $(1,\infty)$.  The remaining $\rho$-values $\Bus^\rho_{x,y}$ can then be defined as right limits.  Shifts $\theta_u$ act on the weights by $(\theta_u\Yw)_x=\Yw_{x+u}$. 
 The limits \eqref{B-Glim} give stationarity and ergodicity of $B$ as stated in this lemma.  
 
 \begin{lemma}\label{B-lm-erg}   Fix $\rho_1,\dotsc,\rho_n\in(1,\infty)$ and $y_1,\dotsc,y_n\in\Z^2$.   Let   $A_x=(\Bus^{\rho_1}_{x,x+y_1},\dotsc, \Bus^{\rho_n}_{x,x+y_n})$.   
    Let $0\ne u\in\Z^2$.  Then   
 the  $\R^n$-valued  process  $A=\{A_x\}_{x\in\Z^2}$  
    is stationary and  ergodic under the  shift  $\theta_u$.
 \end{lemma} 
 
 \begin{proof}  Since the i.i.d.\ process $Y$   is stationary and  ergodic under every shift, it suffices to show that $A_x=A_0\circ\theta_x$ as functions of $Y$.   Let $\uvec^i\in\,]-\evec_1,-\evec_2[$  be associated to $\rho_i$ via \eqref{u-rho} and fix sequences $\{u^1_m\},\dotsc,\{u^n_m\}$ in $\Z^2$  such that, as $m\to\infty$,   $\abs{u^i_m}_1\to\infty$ and $u^i_m/m\to \uvec^i$  for each $i\in[n]$.   Then almost surely, 
 \begin{align*}
 A_x&=(\Bus^{\rho_1}_{x,x+y_1},\dotsc, \Bus^{\rho_n}_{x, x+y_n})
=  \lim_{m\to\infty}  \bigl( [G_{u^1_m, x+y_1}-G_{u^1_m,x}], \dotsc,  [G_{u^n_m, x+y_n}-G_{u^n_m,x}]\bigr) \\
&=  \lim_{m\to\infty}  \bigl( [G_{u^1_m-x, y_1}-G_{u^1_m-x,0}], \dotsc,  [G_{u^n_m-x,y_n}-G_{u^n_m-x,0}]\bigr) \circ\theta_x \\
&=(\Bus^{\rho_1}_{0, y_1},\dotsc, \Bus^{\rho_n}_{0, y_n})  \circ\theta_x  =  A_0 \circ\theta_x. 
\qedhere  \end{align*}
\end{proof} 

\medskip

\subsection{Semi-infinite geodesics in the corner growth model} \label{s:geod} 
Let  $x_\bbullet=\{x_k\}$ be a   finite or infinite south-west directed   nearest-neighbor path  on $\Z^2$  ($x_{k+1}\in\{x_k-\evec_1, x_k-\evec_2\}$).  Then  $x_\bbullet$ is a {\it geodesic} if it gives a maximizing path between any two of its points:  for any $k<\ell$ in the index set of  $x_\bbullet$, 
\[  G_{x_\ell, x_k}=\sum_{i=k}^\ell \Yw_{x_i} . \]

Given $\rho\in(1,\infty)$,  define from each $x\in\Z^2$ the  semi-infinite, south-west directed  path $\bgeodrx=\{\bgeodrx_k\}_{k\in\Z_{\ge0}}$ that starts at $x= \bgeodrx_0$ and chooses a step from $\{-\evec_1,-\evec_2\}$ by following the minimal increment of $\Bus^\rho$:  for $k\ge 0$, 
\be\label{bg6} 
\bgeodrx_{k+1}=\begin{cases}   \bgeodrx_{k} - \evec_1, &\text{if } \ \Bus^\rho_{\bgeodrx_{k}-\evec_1,\,\bgeodrx_{k} } <  \Bus^\rho_{\bgeodrx_{k}-\evec_2,\,\bgeodrx_{k} } 
\\[5pt]  
 \bgeodrx_{k} - \evec_2, &\text{if } \   \Bus^\rho_{\bgeodrx_{k}-\evec_2,\,\bgeodrx_{k}} \le   \Bus^\rho_{\bgeodrx_{k}-\evec_1,\,\bgeodrx_{k} } . 
\end{cases} 
  \ee
The tie-breaking rule in favor of $-\evec_2$  is chosen simply to make $\bgeodrx$ a cadlag function of $\rho$.     For a given $\rho$ 
 equality on the right-hand side  happens with probability zero.    Pictorially, to each point $z$  attach the arrow that points from $z$ to $\bgeodrz_1$.   For each $x$  the path  $\bgeodrx$ is constructed by starting at $x$ and  following the arrows.   

The additivity \eqref{B-add} and weights recovery \eqref{B-reco} imply that $\bgeodrx$ is a (semi-infinite) geodesic:   let $\ell>k\ge 0$ and  suppose  $\{y_i\}_{i=k}^\ell$ is a south-west directed path from $y_k=\bgeodrx_k$ to $y_\ell=\bgeodrx_\ell$. Then 
\begin{align*}
\sum_{i=k}^\ell \Yw_{y_i} &\le \sum_{i=k}^{\ell-1} \Bus^\rho_{y_{i+1},y_i} +      \Yw_{y_\ell} = \Bus^\rho_{y_{\ell},y_k} +      \Yw_{y_\ell}  =  \Bus^\rho_{\bgeodrx_\ell, \bgeodrx_k}  +      \Yw_{\bgeodrx_\ell}   =  \sum_{i=k}^{\ell-1} \Bus^\rho_{\bgeodrx_{i+1},\bgeodrx_i} +      \Yw_{\bgeodrx_\ell} 
= \sum_{i=k}^\ell \Yw_{\bgeodrx_i} . 
\end{align*}
Thus 
 \be\label{GB7} 
G_{\bgeodrx_\ell, \bgeodrx_k}  =\sum_{i=k}^\ell \Yw_{\bgeodrx_i} 
=  \Bus^\rho_{\bgeodrx_\ell, \bgeodrx_k} +  \Yw_{\bgeodrx_\ell} . 
\ee
 We call   $\bgeodrx$ a {\it Busemann geodesic}. 


We state the key properties of semi-infinite geodesics in the next theorem.

\begin{theorem} \label{geod-thm-1}  Fix $\rho\in(1,\infty)$ and let $\uvec=\uvec(\rho)$ be the direction associated to $\rho$ by \eqref{u-rho}.   The following properties hold with probability one. 

{\rm (i) Directedness.}  $\forall x\in\Z^2$  $\ddd\lim_{k\to\infty} \bgeodrx_k/k = \uvec$.  

{\rm (ii) Uniqueness.}  Let $x_\bbullet=\{x_k\}_{k\in\Z_{\ge0}}$ be any semi-infinite geodesic that satisfies   $x_k/k\to\uvec$ as $k\to\infty$.  Then $x_\bbullet=\bgeodr{x_0}$.  

{\rm (iii) Coalescence.}   $\forall x,y\in\Z^2$  the paths  $\bgeodrx$ and $\bgeodr{y}$ coalesce: 
$\exists z=z^\rho(x,y)\in\Z^2$ such that $\bgeodrx\cap\bgeodr{y}=\bgeodr{z}$.   

\end{theorem} 

It is clear from the construction \eqref{bg6} that once $\bgeodrx$ and $\bgeodr{y}$    come together, they stay together.  
We call $z^{\rho(\uvec)}(x,y)$ the {\it coalescence point} of the unique  $\uvec$-directed semi-infinite geodesics from $x$ and $y$.   The Busemann function satisfies 
\be\label{B-G15}   \Bus^\rho_{x,y}=G_{z^\rho(x,y), y}- G_{z^\rho(x,y), x}\quad\text{a.s.} 
\ee
It is important to note that parts (ii) and (iii)  of Theorem \ref{geod-thm-1} are true with probability one only  for a given $\uvec$ and not simultaneously for all directions.  

  Theorem \ref{geod-thm-1}(i) follows from an ergodic theorem for Busemann functions and the shape theorem \ref{sh-thm} (see for example Theorem 4.3 in \cite{geor-rass-sepp-17-geod}). 
   Theorem \ref{geod-thm-1}(ii)-(iii) were established for the exponential CGM in  \cite{coup-11,ferr-pime-05}.   Article \cite{sepp-arxiv-18} gives an alternative derivation of  Theorem \ref{geod-thm-1} based on the properties of the stationary exponential CGM. Versions of  Theorem \ref{geod-thm-1}   for the CGM with general weights appear  in \cite{geor-rass-sepp-17-geod}. 

\begin{remark}[Competition interface]   \label{r:cif}     The {\it geodesic tree} emanating from $x$ consists of all the geodesics between $x$ and points $y\in x+\Z^2_{\le0}$ south and west of $x$.    The semi-infinite geodesics $\bgeodrx$ are infinite rays in this tree.    Every geodesic to $x$ comes through either $x-\evec_1$ or $x-\evec_2$.  This dichotomy splits the tree into two subtrees.  Between the two subtrees lies a unique path $\{\varphi^x_n\}_{n\in\Z_{\ge0}}$  on the dual lattice $(\tfrac12,\tfrac12)+\Z^2$ that starts at $\varphi^x_0=x- (\tfrac12,\tfrac12)$. $\varphi^x_n$ is a.s.\ uniquely defined as the point in $x- (\tfrac12,\tfrac12)+\Z^2_{\le0}$ that satisfies   $\abs{x-\varphi^x_n}_1=n+1$ and 
\[   G_{\varphi^x_n+(-\frac12, \frac12), \, x-\evec_1}-  G_{\varphi^x_n+(-\frac12, \frac12), \, x-\evec_2}>0>G_{\varphi^x_n+(\frac12, -\frac12), \, x-\evec_1}-  G_{\varphi^x_n+(\frac12, -\frac12), \, x-\evec_2}. \] 
(Use the convention $G_{x,y}=-\infty$ if $x\le y$ fails.)   The competition interface has a random asymptotic direction:  
\be\label{cif-lln}   \lim_{n\to\infty}  \frac{\varphi^x_n}{n}=\uvec(\rho^*(x)) \quad \text{ almost surely} \ee 
where the limit is described in \eqref{u-rho} and Lemma \ref{B-reco-lm}.  
This was first proved in \cite{ferr-pime-05}. The limit came from the study of geodesics with Newman's approach.  Identification of the limit  came via a mapping of $\varphi^x$  to a second class particle in the rarefaction fan of  TASEP whose limit had been identified in \cite{ferr-kipn-95}.   An alternative proof that relies on the stationary LPP processes was given in \cite{geor-rass-sepp-17-geod}. 
\qedrm\end{remark}  

 \medskip   


\subsection{Queues} 
\label{s:queue}

We begin with a standard  formulation of  a  queue that obeys  FIFO (first-in-first-out) discipline.  This treatment   goes back to classic works of Lindley \cite{lind-52} and Loynes \cite{loyn-62}.  Modern  references that connect queues with LPP include   \cite{glyn-whit, bacc-boro-mair,drai-mair-ocon}.    

The inputs are   two bi-infinite sequences: the {\it arrival process}   $I=(I_k)_{k\in\Z}$ and the {\it service process} $\w=(\w_j)_{j\in\Z}$ in $\R_{\ge0}^\Z$. They are assumed to  satisfy 
\begin{equation}\label{Iw}
\lim_{m\to-\infty}  \sum_{i=m}^0 (\w_i-I_{i+1})=-\infty.
\end{equation}
The interpretation is that $I_j$ is the time between the arrivals of customers $j-1$ and $j$ and    $\w_j$ is the service time of customer $j$.  From these inputs   three outputs  $\wt I=(\wt I_k)_{k\in\Z}$, $J=(J_k)_{k\in\Z}$ and $\wt\w=(\wt\w_k)_{k\in\Z}$,   also elements of $\R_{\ge0}^\Z$,  are constructed  as follows. 

Let $G=(G_k)_{k\in\Z}$ be any function on $\Z$ that satisfies $I_k=G_k-G_{k-1}$.   Define the sequence  $\wt G=(\wt G_\ell)_{\ell\in\Z}$ by 
\be\label{m:800}
\wt G_\ell=\sup_{k:\,k\le \ell}  \Bigl\{  G_k+\sum_{i=k}^\ell \w_i\Bigr\},  \quad \ell\in\Z. 
 \ee
 Under assumption \eqref{Iw} the supremum in \eqref{m:800} is assumed at some finite $k$. 
 The {\it interdeparture time} between   customers $\ell-1$ and $\ell$ is defined by 
\be\label{m:801}  \wt I_\ell =  \wt G_\ell - \wt G_{\ell-1}  \ee 
and the  sequence  $\wt I=(\wt I_k)_{k\in\Z}$  is the {\it departure process}.  
          The {\it sojourn time} $J_k$ of customer $k$ is defined by 
  \be\label{m:J} J_k=\wt G_k -  G_k, \quad k\in\Z.     \ee
   The third output 
\be\label{m:R}   \wt \w_k=I_k\wedge J_{k-1},   
\quad k\in\Z,  \ee  
is the amount of time customer $k-1$ spends as the last customer in the queue.

  $\wt I$, $J$ and $\wt\w$ are well-defined nonnegative real sequences,  and they do not  depend  on the choice of the function $G$ as long as $G$  has  increments  $I_k=G_k-G_{k-1}$.  
The  three mappings are denoted by  
 \be\label{m:DSR}   \wt I=\Dop(I,\w),  \quad J=\Sop(I,\w), \quad\text{and} \quad \wt \w=\Rop(I,\w). \ee
  The queueing story is good  for imbuing the mathematics with meaning, but it is not necessary for the sequel.  
 
     From
 \be\label{m:806} \wt G_k=\w_k+  G_k\vee \wt G_{k-1} \ee 
follow the useful  iterative equations
  \be\label{m:IJ5}  
 \wt I_k=\w_k+(I_k-J_{k-1})^+
 \quad\text{and}\quad 
 J_k=\w_k+(J_{k-1}-I_k)^+ .  
  \ee
The difference of the two equations above gives a ``conservation law'' 
\be\label{m:cons}  I_k+J_k=J_{k-1}+\wt I_k. \ee

We  extend the queueing operator $\Dop$  to  mappings   $\Dop^{(n)}: (\R_{\ge0}^{\Z})^n \to \R_{\ge0}^{\Z}$  of multiple sequences into a single  sequence.  Let $\zeta,\zeta^1,\zeta^2,\dotsc$ denote elements of $\R_{\ge0}^{\Z}$.  Then, as long as  the actions below are well-defined, let    
	\be\label{m:D67}  \begin{aligned}
		\Dop^{(1)}(\zeta)&=\Dop(\zeta,0)=\zeta, \\  
		\Dop^{(2)}(\zeta^1, \zeta^2) &= \Dop\bigl(\Dop^{(1)}(\zeta^1), \zeta^2\bigr)= \Dop(\zeta^1, \zeta^2),\\   
		\Dop^{(3)}(\zeta^1, \zeta^2, \zeta^3) &= \Dop\bigl(\Dop^{(2)}(\zeta^1,\zeta^{2}), \zeta^3\bigr) = \Dop\bigl(\Dop(\zeta^1,\zeta^{2}), \zeta^3\bigr),  \\
		\text{and in general}\quad   \Dop^{(n)}(\zeta^1, \zeta^2,\dotsc,  \zeta^n) &= \Dop\bigl(\Dop^{(n-1)}(\zeta^1,\dotsc, \zeta^{n-1}), \zeta^n\bigr)  \quad \text{ for $n\ge 2$.}   
	\end{aligned} \ee
In queueing terms,  $ \Dop^{(n)}(\zeta^1, \zeta^2,\dotsc,  \zeta^n)$ is the departure  process that results from feeding arrival process $\zeta^1$ through a series of $n-1$  service stations.  For  $i=2,3,\dotsc,n$,  $\zeta^i$ is the service process at station $i$.    Departures from station $i$ are the arrivals at station $i+1$.  The final output is the departure  process from the last station whose service process is $\zeta^n$.  

We record some inequalities which are to be understood  coordinatewise:   
for example,  $I'\ge I$ means  that $I'_k\ge I_k$ $\forall k\in\Z$.  

\begin{lemma} 	\label{Dop-lm1}
	Assuming that the mappings below are well-defined,  we have the following inequalities. 
	\be\label{m:D6}   \Dop(I,\w) \ge \w.    
	\ee
	\be\label{m:D8}   \text{ If    $I'\ge I$  then  $\Dop(I',\w)\ge \Dop(I,\w)$. }\ee
		\be\label{m:D6.4}  \text{For $n\ge 2$, }\   \Dop^{(n)}(\zeta^1, \zeta^2, \zeta^3,\dotsc,  \zeta^n) \ge   \Dop^{(n-1)}(\zeta^2, \zeta^3,\dotsc,  \zeta^n) .  
	\ee
\end{lemma} 

\begin{proof}
  The first part of \eqref{m:IJ5}  implies \eqref{m:D6}.    For  \eqref{m:D8} observe that  
 \begin{align*}
J_k=  \sup_{j:\,j\le k}  \Bigl\{  G_j - G_k +\sum_{i=j}^k \w_i\Bigr\}  
\ge   \sup_{j:\,j\le k}  \Bigl\{  G'_j - G'_k +\sum_{i=j}^k \w_i\Bigr\} =J'_k.  
 \end{align*}
Now \eqref{m:IJ5} gives $\wt I'_k\ge \wt I_k$.  
 
 Inequality \eqref{m:D6.4} comes by induction on $n$. The case $n=2$ is \eqref{m:D6}. Then, by induction and \eqref{m:D8},  
\begin{align*}
 \Dop^{(n)}(\zeta^1, \dotsc,  \zeta^n) &= \Dop\bigl(\Dop^{(n-1)}(\zeta^1,\dotsc, \zeta^{n-1}), \zeta^n\bigr)  \ge  \Dop\bigl(\Dop^{(n-2)}(\zeta^2,\dotsc, \zeta^{n-1}), \zeta^n\bigr) \\
 &=\Dop^{(n-1)}(\zeta^2, \dotsc,  \zeta^n). 
\qedhere \end{align*}
\end{proof}

We record the most basic fact about M/M/1 queues.  The following notation will be used in the sequel.  Let  $\lambda=(\lambda_1,\dotsc,\lambda_n)\in(0,\infty)^n$  be an $n$-tuple of positive reals.   Let    $\zeta=(\zeta^1, \dotsc,\zeta^n)\in(\R_{\ge0}^\Z)^n$ with $\zeta^i=(\zeta^i_k)_{k\in\Z}$  denote  an $n$-tuple  of nonnegative bi-infinite random sequences.  Then $\zeta$  has distribution $\nu^\lambda$ if all the coordinates $\zeta^i_k$ are mutually independent with marginal distributions $\zeta^i_k\sim$ Exp$(\lambda_i^{-1})$. In other words, $\zeta^i$ is a sequence of i.i.d.\ mean $\lambda_i$ exponential variables, and the sequences are independent.  

\begin{lemma}\label{Dop-lm4} Let $n\ge 2$ and let $\lambda=(\lambda_1,\dotsc,\lambda_n)$ satisfy   $\lambda_1>\dotsm>\lambda_n>0$.  Let $\zeta$ have distribution $\nu^\lambda$.   Then $\Dop^{(n)}(\zeta^1,\dotsc,  \zeta^n)$ has distribution $\nu^{\lambda_1}$, in other words,  $\Dop^{(n)}(\zeta^1,\dotsc,  \zeta^n)$ is a sequence of i.i.d.\ mean $\lambda_1$ exponential random variables.  
\end{lemma} 

\begin{proof} The case $n=2$ is in Lemma \ref{m:Lem-I}.  The general case follows by induction on $n$.  
\end{proof}

\section{Joint distribution of the Busemann functions}\label{s:Bus}

This section contains the main results on 
  the joint distribution of the Busemann process $\Bus^{\rcbullet}=\{\Bus^\rho: 1<\rho<\infty\}$ defined in Theorem \ref{B-th-zero}.   Proofs are  in Section \ref{s:Bpf}.    The  distribution of the $n$-tuple  $\{(\Bus^{\rho_1}_{x-\evec_1,x} , \dotsc, \Bus^{\rho_n}_{x-\evec_1,x})\}_{x\cdot\evec_2=t}$ on a given lattice level $t\in\Z$ comes through a mapping of a product of  exponential distributions.  This mapping is developed next. 
  
 \subsection{Coupled exponential distributions}    Fix $n\in\Z_{>0}$ for the moment and 
define the following two  spaces of $n$-tuples of nonnegative real sequences. 
      The sequences themselves are denoted by $I^i=(I^i_k)_{k\in\Z}$ and $\eta^i=(\eta^i_k)_{k\in\Z}$ for $i\in[n]$. 
 \be\label{m:aYYth} \begin{aligned}   
 \caY_{n}=\Bigl\{  & I=(I^1, I^2, \dotsc, I^n)\in (\R_{\ge0}^\Z)^n : \\
 & \,\forall\, i\in\lzb 2,n\rzb , \   \lim_{m\to-\infty} \frac1{\abs m}\sum_{k=-m}^0 I^{i}_k\;>\,\lim_{m\to-\infty} \frac1{\abs m}\sum_{k=-m}^0 I^{i-1}_k> 0 \; \Bigr\} . 
\end{aligned}
 \ee
 \be\label{m:XXth}\begin{aligned}    \cX_{n}=\Bigl\{\eta=(\eta^1, \eta^2, \dotsc, \eta^n)\in (\R_{\ge0}^\Z)^n &:  \eta^{i}\ge\eta^{i-1} \; \,\forall\, i\in\lzb 2,n\rzb, \ \text{ and}  \\
&\qquad\qquad \varliminf_{m\to-\infty} \frac1{\abs m}\sum_{k=-m}^0 \eta^1_k> 0\;\Bigr\}.  
\end{aligned} \ee
  The existence of the Ces\`aro limits as $m\to-\infty$  is part of the definitions.  
$\caY_{n}$ and $\cX_{n}$ are Borel subsets of $(\R_{\ge0}^\Z)^n$ and thereby separable metric spaces in the product topology. We endow them with their Borel $\sigma$-algebras.  



Define a mapping  $\Daop^{(n)}:  \caY_{n}\to  \cX_{n}$ in terms of the multiqueue mappings $\Dop^{(k)}$ of \eqref{m:D67} as follows: for  $I=(I^1, I^2, \dotsc, I^n)\in \caY_{n}$,  the image   $\eta=(\eta^1, \eta^2, \dotsc, \eta^n)=\Daop^{(n)}(I)$ is defined by 
	\be\label{m:eta5}   \eta^i=\Dop^{(i)}(I^i, I^{i-1},  \dotsc, I^1) \quad \text{for} \ \  i=1,\dotsc, n.
	\ee
In particular, the first sequence is simply copied over: $\eta^1=I^1$.  Then $\eta^{2}  =\Dop(I^{2}, I^1)$,  $ \eta^3=\Dop^{(3)}(I^3,I^2,I^1)=\Dop(\Dop(I^3,I^2),I^1)$, and so on.  Iterated application of Lemma \ref{lm-D14} from  Appendix \ref{a:queue} together with the assumption $I \in \caY_{n}$ ensures that the mappings $\Dop^{(i)}(I^i, I^{i-1},  \dotsc, I^1)$ are well-defined.   Furthermore,   $\eta \in \cX_{n}$ follows from inequalities \eqref{m:D6} and \eqref{m:D6.4}.   
Lemma \ref{lm-D14}  implies also that  $\Daop^{(n)}$ maps  $\caY_{n}$ into itself.  We do not need this feature in the sequel, which is why we did not define   $ \cX_{n}$ as a subspace of $\caY_{n}$.  

Recall that
\be\label{nu5} \begin{aligned} 
&\text{for    $\rho=(\rho_1,\dotsc,\rho_n)\in (0,\infty)^n$,  $I=(I^1, I^2, \dotsc, I^n)$ has distribution   $\nu^\rho$ if}\\ 
&\text{all coordinates $I^i_k$ are independent and $I^i_k\sim$ Exp($\rho_i^{-1}$) for each $k\in\Z$ and  $i\in[n]$.}
\end{aligned} \ee
If $\rho$ satisfies   $0<\rho_1<\rho_2<\dotsm<\rho_n$  then  $\nu^\rho$   is supported on $\caY_{n}$.  For these $\rho$   define  the probability measure $\mu^\rho$ on $\cX_{n}$ as the image of $\nu^\rho$ under $\Daop^{(n)}$: 
\be\label{mu5} \begin{aligned} 
\text{for $\rho=(\rho_1,\rho_2,\dotsc, \rho_n)$ such that   $0<\rho_1<\rho_2<\dotsm<\rho_n$,  define   $\mu^\rho=\nu^\rho\circ(\Daop^{(n)})^{-1}$.}    
\end{aligned} \ee
   By Lemma \ref{Dop-lm4},   if $\eta$ has distribution $\mu^\rho$ with $0<\rho_1<\rho_2<\dotsm<\rho_n$, then  for each $i\in[n]$,  $\eta^i=(\eta^i_k)_{k\in\Z}$ is a sequence of i.i.d.\ mean $\rho_i$ exponential variables.    The mapping $\Daop^{(n)}$  couples  the variables $\eta^i_k$   together so that $ \eta^{i-1}_k\le\eta^{i}_k$ for all $i\in\lzb 2,n\rzb$ and  $k\in\Z$.
   
  Translations $\{\theta_\ell\}_{\ell\in\Z}$ act on  $n$-tuples  of sequences  by $(\theta_\ell\eta)^i_k=\eta^i_{k+\ell}$ for $i\in[n]$ and $k,\ell\in\Z$.  A {\it translation-ergodic} probability measure $Q$ on $\cX_{n}$ is invariant under $\{\theta_\ell\}$ and satisfies $Q(A)\in\{0,1\}$ for any Borel set $A\subset\cX_{n}$ that is invariant under $\{\theta_\ell\}$ (and similarly for any other sequence space).    
   
   \begin{theorem}\label{murho-th1}    The probability measures $\mu^\rho$ are translation-ergodic and  have the following  properties. 
   
  {\rm (Continuity.)}  The probability measure $\mu^\rho$ is weakly continuous as a function of $\rho$ on the set of vectors that satisfy $0<\rho_1<\rho_2<\dotsm<\rho_n$. 
  
 {\rm (Consistency.)}   If $(\eta^1, \dotsc, \eta^n)\sim\mu^{(\rho_1,\dotsc,\rho_n)}$, then  $(\eta^1, \dotsc,  \eta^{j-1}, \eta^{j+1},\dotsc, \eta^n)\sim\mu^{(\rho_1,\dotsc,  \rho_{j-1}, \rho_{j+1},\dotsc,  \rho_n)}$ for all $j\in[n]$. 
   
 \end{theorem}
 
 Continuity of $\rho\mapsto\mu^\rho$  is proved in Section \ref{s:Bpf}.  Translation-covariance of the queueing mappings ($\Dop(\theta_\ell I, \theta_\ell\w)=\theta_\ell\Dop(I,\w)$) implies that $\mu^\rho$ inherits the translation-ergodicity of $\nu^\rho$.     We  omit the proof of consistency.  Consistency  will be an  indirect consequence of the uniqueness of $\mu^\rho$ as the translation-ergodic invariant distribution of the so-called coupled process (Theorem \ref{m:thm-eta}). 
   
 

\subsection{Distribution of Busemann functions} 
Return to the Busemann functions $\Bus^{\rcbullet}$ defined in Theorem \ref{B-th-zero}.   For each level $t\in\Z$ define the   level-$t$ sequence  of weights  $\wb\Yw_t=(\Yw_{(k,t)})_{k\in\Z}$ and for a given $\rho\in(1,\infty)$, sequences of $\evec_1$ and $\evec_2$ Busemann variables at level $t$: 
\[  \Busv{\rho}{\evec_1}{t} =(\Bus^\rho_{(k-1,t),(k,t)})_{k\in\Z}
\quad\text{and}\quad 
\Busv{\rho}{\evec_2}{t} =(\Bus^\rho_{(k,t-1),(k,t)})_{k\in\Z}. \]

The next main result characterizes uniquely the distribution of the joint process $(\Yw, \Bus^{\rcbullet})$ of weights and Busemann functions.

\begin{theorem}\label{B-th1}  Let $Y=(Y_x)_{x\in\Z^2}$ be i.i.d.\ {\rm Exp(1)} variables  as in Section \ref{s:cgm}.   Let $1<\rho_1<\dotsm<\rho_n$.  Then at  each level  $t\in\Z$,  the $(n+1)$-tuple of sequences   $(\wb\Yw_t, \Busv{\rho_1}{\evec_1}{t}, \dotsc, \Busv{\rho_n}{\evec_1}{t})$ has distribution $\mu^{(1, \rho_1,\dotsc,\rho_n)}$. 

\end{theorem}

Once the   process $ \{\Busv{\rho}{\evec_1}{t-1}\}_{1<\rho<\infty}$ 
on a single level $t-1$ is given, the variables $\Bus^\rho_{x-\evec_i,x}$  at higher  levels $t, t+1, t+2, \dotsc$  can be deduced by drawing independent weights  $\wb\Yw_{t}, \wb\Yw_{t+1}, \dotsc$ and by applying queueing mappings.  By stationarity the full distribution will then have  been determined.  The next lemma describes the  single step of computing  the  $\evec_1$ and $\evec_2$ Busemann increments on  level $t$    from the process 
 $ \{\Busv{\rho}{\evec_1}{t-1}\}_{1<\rho<\infty}$   and independent level-$t$ weights $\wb\Yw_t$.  The mappings $\Dop$ and $\Sop$ were  specified in \eqref{m:DSR}.

\begin{lemma} \label{B-lm3}   There exists an event of full probability on which  
\[ \Busv{\rho}{\evec_1}{t}=\Dop\bigl(\,\Busv{\rho}{\evec_1}{t-1}, \wb\Yw_t\bigr)\quad\text{and}\quad  \Busv{\rho}{\evec_2}{t}=\Sop\bigl(\,\Busv{\rho}{\evec_1}{t-1}, \wb\Yw_t\bigr)\quad
\text{for all $\rho\in(1,\infty)$ and  $t\in\Z$.} 
 \] 
\end{lemma}

The remainder of this section describes some distributional properties of $\Bus^\rcbullet$ restricted to horizontal edges and lines on $\Z^2$.  
The corresponding statements for vertical edges and lines are obtained by replacing $\rho$ with $\rho/(\rho-1)$.  This is due to  the distributional equality $\{\Bus^\rho_{x-\evec_2,x}\}_{x\in\Z^2}\deq  \{\Bus^{\rho/(\rho-1)}_{Rx-\evec_1,Rx}\}_{x\in\Z^2}$ where $R(x_1,x_2)=(x_2,x_1)$. This follows  from \eqref{u-rho}  and  the limits  \eqref{B-Glim},   by reflecting the lattice across the diagonal.  

\subsection{Marginal distribution on a single edge} 
Lemma \ref{B-reco-lm} implies that for a fixed horizontal edge $(x-\evec_1,x)$ we can extend $\{\Bus^\rho_{x-\evec_1,x}:1< \rho<\infty\}$ to a cadlag process  $\Bus^\rcbullet_{x-\evec_1,x}=\{\Bus^\rho_{x-\evec_1,x}:1\le \rho<\infty\}$ by setting $\Bus^1_{x-\evec_1,x}=\Yw_x$.   We  describe  the distribution of this process in terms of  a marked point process.  
Figure \ref{f:B} illustrates a sample path of this process. 

 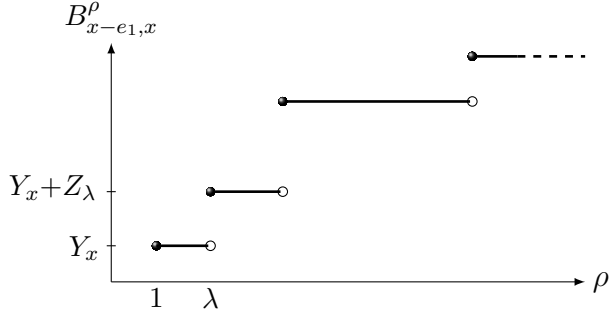
\begin{figure}[t] 
	\begin{tikzpicture}[>=latex,scale=0.6]
	\draw[->] (0,0)--(10.5,0);
	\draw[->] (0,0)--(0,5.3);
	\draw(1,-0.35)node{$1$};
	\draw(10.85,0)node{\large$\rho$};
	\draw(0, 5.82)node{$B^{\rho}_{x-e_1,x}$};
	\shade[ball color=black](1,0.8)circle(1mm);
	\draw[line width=1pt](1,0.8)--(2.15,0.8);
	\draw[line width=1pt](2.2,2)--(3.75,2);
	\draw[line width=1pt](3.8,4)--(7.95,4);
	\draw[line width=1pt](8,5)--(9,5);
	\draw[dashed, line width=1pt](9,5)--(10.5,5);
	\shade[ball color=black](2.2,2)circle(1mm);
	\shade[ball color=black](3.8,4)circle(1mm);
	\shade[ball color=black](8,5)circle(1mm);
	\draw(-0.6,0.75)node{$Y_x$};
	\draw(-1.3,2.09)node{$Y_x$+$Z_{\lambda}$};
	\draw [-] (-0.1,0.8) -- (0.1,0.8);
	\draw [-] (-0.1,2) -- (0.1,2);
	\draw (2.2,0.8) circle (3pt);
	\draw (3.8,2) circle (3pt);
	\draw (8,4) circle (3pt);
	\draw(2.2,-0.35)node{$\lambda$};
	\draw [decoration={brace,amplitude=10pt},xshift=-4pt,yshift=0pt];
	\end{tikzpicture}
	\caption{\small A sample path of the pure jump process $\{B^{\rho}_{x-e_1,x}\}_{\rho\,\in\, [1,\infty)}$, with initial value  $B^1_{x-e_1,x}=Y_x$. The jump times are a Poisson point process on $(1,\infty)$ with intensity $s^{-1}ds$. Given that there is a jump at $\lambda$, the jump size  is an independent Exp$(\lambda^{-1})$ variable $Z_{\lambda}$.}
	\label{f:B} 
\end{figure}

 Let $N$ be the simple point process on the interval  $\{s:1\le s<\infty\}$  
  that has  a  point at  $s=1$ with probability 1, and on the open interval  $(1,\infty)$  $N$  is  a Poisson point process  with parameter measure $s^{-1}\,ds$.    (We use $N$ to denote both the random discrete set of locations and the resulting random point measure.)     To each point   $t\in[1,\infty)$ of $N$   attach an independent Exp($t^{-1}$) distributed weight $Z_t$.   
Define the nondecreasing cadlag process  $X(\cdot)=\{X(\rho): \rho\in[1,\infty)\}$ by  
\be\label{X1} X(\rho)=\sum_{s\in N\cap[1,\rho]} Z_s. \ee 

 \begin{theorem} \label{B-th5}   Fix $x\in\Z^2$.  The nondecreasing cadlag processes $\Bus^{\,\!\rcbullet}_{x-\evec_1,x}$ and $X(\abullet)$  indexed by  $[1,\infty)$ 
 are equal in distribution.  
 \end{theorem}
 
 A qualitative consequence of Theorem \ref{B-th5} is that for any given $\lambda\in(1,\infty)\setminus N$,  $\rho\mapsto\Bus^{\rho}_{x-\evec_1,x}$ is constant in an  interval around $\lambda$.  From identity \eqref{B-G15}, it is evident that this is due to the fact that the coalescence point function  $\rho\mapsto z^{\rho}(x-\evec_1,x)$ is constant in an interval.   This is an analogue of the local constancy of the difference weight profile in Theorem 1.1 of \cite{basu-gang-hamm-arxiv}. 
 Full implications of Theorem \ref{B-th5}  for the coalescence structure of the geodesics  of the CGM will be explored in the subsequent paper \cite{janj-rass-sepp-future}. 

 Theorem \ref{B-th5} is proved by establishing  that 
  $\Bus^{\,\!\rcbullet}_{x-\evec_1,x}$ has independent increments and by deducing the distribution of an increment.  Independent increments means that   for $1=\rho_0<\rho_1<\dotsm<\rho_n$, the random variables $\Yw_x=\Bus^{\rho_0}_{x-\evec_1,x},  \Bus^{\rho_1}_{x-\evec_1,x}-\Bus^{\rho_0}_{x-\evec_1,x}, \dotsc, \Bus^{\rho_{n}}_{x-\evec_1,x}-\Bus^{\rho_{n-1}}_{x-\evec_1,x}$ are independent.    For $1\le\lambda<\rho<\infty$,  the distribution of the increment is 
 \be\label{B-incr13} \begin{aligned} 
 \P\{ \Bus^\rho_{x-\evec_1,x}-\Bus^\lambda_{x-\evec_1,x} =0\}&=\P\{ N(\lambda,\rho]=0\}= \tfrac\lambda\rho\\
  \P\{ \Bus^\rho_{x-\evec_1,x}-\Bus^\lambda_{x-\evec_1,x} >s\} &=  \bigl(1-\tfrac\lambda\rho\bigr) e^{-s/\rho} 
  \quad\text{for } \ s>0. 
\end{aligned}  \ee

 For the process $\Bus^\rcbullet_{x-\evec_2,x}$ on a vertical edge, the result of  Theorem \ref{B-th5} is that 
\be\label{B-flip}  
\{\Bus^\rho_{x-\evec_2,x}:1< \rho<\infty\}\deq \{X\bigl((\tfrac{\rho}{\rho-1})+\bigr): 1< \rho<\infty\} . 
\ee


 
 \subsection{Marginal distribution on a level of the lattice} 
 A striking and useful property of the Busemann process $\{\Bus^\rho_{(k-1,t),(k,t)}\}_{k\in\Z}$ along a horizontal line in $\Z^2$  for a {\it fixed} value $\rho\in(1,\infty)$ is that the variables $\{\Bus^\rho_{(k-1,t),(k,t)}\}_{k\in\Z}$  are i.i.d.\ (part (ii) of Theorem \ref{B-th-zero}).  For example, article \cite{bala-cato-sepp} used this feature heavily to deduce the KPZ fluctuation exponents of the corner growth model. 
The next theorem shows that this property breaks down totally already for the joint process 
  $\{(\Bus^\lambda_{(k-1,t),(k,t)}, \Bus^\rho_{(k-1,t),(k,t)})\}_{k\in\Z}$ for two parameter values $\lambda<\rho$.  Namely, this pair process is not even a Markov chain and not reversible.  However, if we restrict attention to the differences $\Bus^\rho_{(k-1,t),(k,t)}-\Bus^\lambda_{(k-1,t),(k,t)}$, we can recover the reversibility.  The differences are of interest because they indicate a jump in the coalescence point $z^\bbullet((k-1,t),(k,t))$ in \eqref{B-G15} as a function of the direction. 
  
  For the statement of the theorem below,  the negative part of a real number is $x^-=(-x)\vee 0$.   The Markov chain $X_k$ in part (a) below has a queueing interpretation as the difference between the sojourn time of customer $k-1$ and the waiting time till the arrival of customer $k$. The details are in the proof in Lemma \ref{X-lm1}. 
 
 
\begin{theorem}\label{B-th8}    Let $1\le\lambda<\rho<\infty$.  

{\rm(a)}  The sequence of differences $\{ \Bus^\rho_{(k-1,t),(k,t)}-\Bus^\lambda_{(k-1,t),(k,t)}\}_{k\in\Z}$ is not a Markov chain, but  there  exists a stationary reversible Markov chain $\{X_k\}_{k\in\Z}$ such that this distributional equality of processes  holds: 
\[  \bigl\{ \Bus^\rho_{(k-1,t),(k,t)}-\Bus^\lambda_{(k-1,t),(k,t)}\bigr\}_{k\in\Z}\deq\{X^-_k\}_{k\in\Z}. \]
  In particular, the process of differences is reversible:  
\[   \bigl\{ \Bus^\rho_{(k-1,t),(k,t)}-\Bus^\lambda_{(k-1,t),(k,t)}\bigr\}_{k\in\Z}\deq
 \bigl\{\Bus^\rho_{(-k-1,t),(-k,t)}-\Bus^\lambda_{(-k-1,t),(-k,t)}\bigr\}_{k\in\Z}. \]

{\rm(b)} The sequence of pairs $\{(\Bus^\lambda_{(k-1,t),(k,t)}, \Bus^\rho_{(k-1,t),(k,t)})\}_{k\in\Z}$ is not a Markov chain.   The joint distribution of   two successive pairs  
\[   \bigl( (\Bus^\lambda_{(k-1,t),(k,t)}, \Bus^\rho_{(k-1,t),(k,t)}) \,, \, 
(\Bus^\lambda_{(k,t),(k+1,t)}, \Bus^\rho_{(k,t),(k+1,t)}) \bigr)  \]
is {\rm not} the same as the joint distribution of its transpose 
\[   \bigl( (\Bus^\lambda_{(k,t),(k+1,t)}, \Bus^\rho_{(k,t),(k+1,t)})   \,, \,  (\Bus^\lambda_{(k-1,t),(k,t)}, \Bus^\rho_{(k-1,t),(k,t)})
 \bigr).  \]
 In particular,  the process of pairs $\{(\Bus^\lambda_{(k-1,t),(k,t)}, \Bus^\rho_{(k-1,t),(k,t)})\}_{k\in\Z}$ is {\rm not}  reversible. 

\end{theorem}

\subsection{The initial segment of the Busemann geodesic}
As the last application of Theorem \ref{B-th1} we calculate the probability distribution of the length of the initial  horizontal run of a semi-infinite geodesic.  

Let $\arrowrx=\bgeodrx_1-x$ be the first step of the  $\Bus^\rho$ Busemann geodesic \eqref{bg6} started at $x$.    $\{\arrowrx\}_{x\in\Z^2}$ is a  random configuration with values in $\{-\evec_1,-\evec_2\}$.       
 By  weight recovery \eqref{B-reco},  $\arrowr{x}=-\evec_1$  iff $\Bus^\rho_{x-\evec_1, x}-\Yw_x=0$.   Hence by Theorem \ref{B-th8}(a) with $\lambda=1$, reversibility holds along a line: 
 $\{\arrowr{k\evec_i}\}_{k\in\Z}\deq\{\arrowr{-k\evec_i}\}_{k\in\Z}$. 
 
The first part of the theorem below   gives a queueing characterization for the process  $\{\arrowr{k\evec_1}\}_{k\in\Z}$.   
To that end, for the queueing mapping $\wt I=\Dop(I,\w)$ of \eqref{m:DSR}   define the indicator variables 
\be\label{wait-k}  \eta_k=\ind_{\wt I_k=\w_k}=\ind\{\text{customer $k$ has to wait before entering service}\} .\ee

Let 
\[ \counta_x=\inf\{k\in\Z_{\ge0}:   \arrowr{x-k\evec_1}=-\evec_2  \} 
  \]
denote the number of consecutive $-\evec_1$ steps that $\bgeodrx$ takes  from a  deterministic starting point $x$. 
Part (b)  of the theorem gives  the distribution of  $\counta_x$.
The  {\it Catalan triangle}   $\{C(n,k):\,0\leq k\leq n\}$ is given by 
\be\label{Cat-tr} C(n,k)=\frac{(n+k)!(n-k+1)}{k!(n+1)!}. \ee
Information about $C(n,k)$ is given above  Lemma \ref{lm:poi8} 
in Appendix \ref{a:exp}.  

\begin{theorem} \label{B-th10} $ $   Let $1<\rho<\infty$.  

{\rm (a)}   Let the service and arrival processes satisfy  $(\w,I)\sim\nu^{(1,\rho)}$ and define  $\eta_k$ by \eqref{wait-k}.  Then  we have the distributional equality 
$\{ \ind\{\arrowr{k\evec_1}=-\evec_1\}\}_{k\in\Z}\deq\{\eta_k\}_{k\in\Z}$. 

{\rm (b)}    Let $x\in\Z^2$.  Then $\P\{\counta_x=0\}= 1-\rho^{-1}$  and for  $n\in\Z_{>0}$, 
\be\label{distr4}   \P\{\counta_x=n\}= (1-\rho^{-1}) 
\sum_{k=0}^{n-1} C(n-1,k) \frac{\rho^k}{(\rho+1)^{n+k}}. \ee

\end{theorem}

The   distribution in \eqref{distr4} is  proper, that is,  $\sum_{n\in\Z_{\ge0}} \P\{\counta_x=n\}= 1$.    This follows for example  from Theorem  \ref{geod-thm-1}(i) according to which  the Busemann geodesic has direction strictly off the axes. 

\begin{remark}\label{B-rm10}  If we take $(\w,I)\sim\nu^{(\lambda,\rho)}$   for $1<\lambda<\rho$ in  Theorem \ref{B-th10} and define  $\eta_k$ again by \eqref{wait-k}, we get the distributional equality  $\{ \ind\{\Bus^\rho_{(k-1,t),(k,t)}=\Bus^\lambda_{(k-1,t),(k,t)}\}\}_{k\in\Z}\deq\{\eta_k\}_{k\in\Z}$.   The calculation that produced part (b)  gives the distribution $\P\{\counta^{\lambda,\rho}_x=0\}=\frac{\rho-\lambda}{\rho}$ and 
\[  \P\{\counta^{\lambda,\rho}_x=n\}=\frac{\rho-\lambda}{\rho} \sum_{k=0}^{n-1} C(n-1,k) \frac{\rho^k\lambda^n}{(\lambda+\rho)^{n+k}} \quad \text{ for }  n\in\Z_{>0},  \]
for the random variable 
\[  \counta^{\lambda,\rho}_x=\inf\{k\in\Z_{\ge0}:  \Bus^\rho_{x-(k+1)\evec_1, x-k\evec_1}>\Bus^\lambda_{x-(k+1)\evec_1, x-k\evec_1}\}. 
 \]
Note that $\Bus^\rho_{x-\evec_1, x}=\Bus^\lambda_{x-\evec_1, x}$  tells us that $\bgeodrx_1=\bgeod{\lambda}x_1$ but not which step is chosen.  
\qedrm\end{remark}

\section{Properties of queueing mappings} \label{s:D} 

This section proves a property of the queueing mapping $\Dop$ (Lemma \ref{m:D-lm4} below) on which the intertwining property that comes in Section \ref{s:inv-cpl} rests.  To prove  Lemma \ref{m:D-lm4} we develop  a duality   in the queueing setting of Section \ref{s:queue}: namely, an LPP process defined in terms of weights $(I,\w)$ can be equivalently described in terms of weights $(\wt I, \wt\w)$ defined by \eqref{m:DSR}.   Routine facts about the queueing mappings are collected in Appendix \ref{a:queue}.

\medskip

Fix an origin  $m\in\Z$.  Assume given nonnegative real weights 
\be\label{IJ99}  
J_m, \quad (I_i)_{i\ge m+1}, \quad\text{and}\quad (\w_i)_{i\ge m+1}. 
\ee 
  From these define iteratively for $k=m+1,m+2,\dotsc$
\be\label{IJ100}  
\wt I_k=\w_k+(I_k-J_{k-1})^+, \quad 
J_k=\w_k+(J_{k-1}-I_k)^+, \quad\text{and}\quad \wt\w_k=I_k\wedge J_{k-1}.  
\ee

There is a duality or reversibility of sorts  here.    For a fixed $k$, equations \eqref{IJ100} are equivalent to 
\be\label{IJ101}  
 I_k= \wt \w_k + (\wt I_k-J_{k})^+, \quad 
J_{k-1}=  \wt \w_k + (J_{k}-\wt I_k)^+, \quad\text{and}\quad 
\w_k=\wt I_k\wedge J_{k}. 
\ee

We turn  this reversibility into a lemma as follows. 
  Restrict the given $J$, $I$ and $\w$ weights in \eqref{IJ99}   to the interval $\lzb m, n\rzb$.  Then on the interval $\lzb-n,-m\rzb$ define the  given weights  $J'_{-n}$, $(I'_i)_{-n+1\le i\le -m}$ and $(\w'_i)_{-n+1\le i\le -m}$ as
\be\label{dual7}
I'_i=\wt I_{-i+1},\quad J'_{-n}=J_n, \quad\text{and}\quad  \w'_i=\wt\w_{-i+1}.  
\ee
Now   apply \eqref{IJ100} to these given weights to compute $(\wt I'_k, J'_k, \wt\w'_k)$ for $k\in\lzb-n+1,-m\rzb$. 
First assume by induction that $J'_{k-1}=J_{-k+1}$.  The base case $k-1=-n$   is covered by the definition in \eqref{dual7}.  Then   
\begin{align*}
J'_k&=\w'_k+(J'_{k-1}-I'_k)^+= \wt\w_{-k+1}+ (J_{-k+1}-\wt I_{-k+1})^+  \\
&=  I_{-k+1}\wedge J_{-k}+  (J_{-k}-I_{-k+1})^+ = J_{-k}. 
\end{align*} 
The third equality above used the definition of $\wt\w$ in \eqref{IJ100} and the conservation law 
\be\label{IJ104}  I_k+J_k=J_{k-1}+\wt I_k \ee
 that follows from \eqref{IJ100}.   Thus $J'_k=J_{-k}$ for all $k\in\lzb-n,-m\rzb$. 
Next 
\begin{align*}
\wt I'_k&=\w'_k+(I'_k-J'_{k-1})^+= \wt\w_{-k+1}+ (\wt I_{-k+1}-J_{-k+1})^+\\
&= I_{-k+1}\wedge J_{-k}+  (I_{-k+1}-J_{-k})^+ =  I_{-k+1}.  
\end{align*}
Finally 
\[  \wt\w'_k=I'_k\wedge J'_{k-1}=\wt I_{-k+1}\wedge J_{-k+1} =\w_{-k+1}  \]
as follows again from  \eqref{IJ100}.   We summarize this finding as follows.  

\begin{lemma}\label{lm-dual12}  Fix $m<n$.  Assume given $J_m$, $(I_i)_{m+1\le i\le n}$ and $(\w_i)_{m+1\le i\le n}$.    Compute $(\wt I_k, J_k, \wt\w_k)_{m+1\le k\le n}$ from \eqref{IJ100}.   Then define $J'_{-n}$, $(I'_i)_{-n+1\le i\le -m}$ and $(\w'_i)_{-n+1\le i\le -m}$ by \eqref{dual7} and apply  \eqref{IJ100}   to compute $(\wt I'_k, J'_k, \wt\w'_k)_{-n+1\le k\le-m}$.    The conclusion is that $(\wt I'_k, J'_k, \wt\w'_k)=(I_{-k+1}, J_{-k}, \w_{-k+1})$ for 
$k\in\lzb-n+1,-m\rzb$. 
\end{lemma} 

 
\begin{figure}[t] 
	\begin{tikzpicture}[>=latex,scale=1.0]
	\draw[step=1cm,gray, thin] (0,0) grid (5,1);
	\draw(0,-0.28)node{\scriptsize$(m,0)$};
	\draw(0,1.25)node{\scriptsize$(m,1)$};
	\draw(5,-0.28)node{\scriptsize$(n,0)$};
	\draw(5,1.25)node{\scriptsize$(n,1)$};
	\draw(1.9,-0.28)node{\scriptsize$(i$-$1,0)$};
	\draw(3.1,-0.28)node{\scriptsize$(i,0)$};
	\draw(3,1.35)node{\large$\omega_i$};
	\draw(2.5,0.30)node{\large$I_i$};
	\draw(-0.37,0.5)node{\large$J_m$};
	\draw[line width=3.0pt](0,0)--(0,1);
	\draw[line width=3.0pt](2,0)--(3,0);
	\shade[ball color=black](3,1)circle(1mm);
	\end{tikzpicture}
	\qquad\qquad
	\begin{tikzpicture}[>=latex,scale=1.0]
	\draw[step=1cm,gray, thin] (0,0) grid (5,1);
	\draw(0,-0.28)node{\scriptsize$(m,0)$};
	\draw(0,1.25)node{\scriptsize$(m,1)$};
	\draw(5,-0.28)node{\scriptsize$(n,0)$};
	\draw(5,1.25)node{\scriptsize$(n,1)$};
	\draw(1.8,1.25)node{\scriptsize$(i$-$1,1)$};
	\draw(3.2,1.25)node{\scriptsize$(i,1)$};
	\draw(2,-0.33)node{\large$\widetilde{\omega}_i$};
	\draw(2.55,1.38)node{\large$\widetilde{I}_i$};
	\draw(5.37,0.5)node{\large$J_n$};
	\draw[line width=3.0pt](5,0)--(5,1);
	\draw[line width=3.0pt](2,1)--(3,1);
	\shade[ball color=black](2,0)circle(1mm);
	\end{tikzpicture}
		\caption{ \small  Illustration of the weights $(I,J,\w)$  on the left and weights $(\wt I,J,\wt\w)$  on the right. Pairs $(k,a)\in\lzb m,n\rzb\,\times\lzb0,1\rzb$ mark vertices of the two-level strip.}\label{fig:IJ4} 
\end{figure}

 Next we use the  weights   given in  \eqref{IJ99}  to construct a last-passage process on the two-level strip $\lzb m,\infty\lzb\,\times\lzb0,1\rzb$ in $\Z^2$.  In this construction $I_i$  serves as a weight on the horizontal edge $((i-1,0),(i,0))$  on the lower $0$-level,   $J_m$ is a weight on the vertical edge $((m,0), (m,1))$,  and $\w_i$  is a   weight  at vertex $(i,1)$ on the upper $1$-level.  (The left diagram of Figure \ref{fig:IJ4} illustrates.) 
The  last-passage values $H_{(m,0),(n,a)}$ are defined for   $(n,a)\in\lzb m,\infty\lzb\,\times\lzb0,1\rzb$ as follows: 
\be\label{H}\begin{aligned} 
H_{(m,0),(m,0)}&=0 \quad\text{and}\quad H_{(m,0),(n,0)}=\sum_{i=m+1}^n I_i \quad \text{for } n>m,   \\
H_{(m,0),(m,1)}&=J_m\,, \\
H_{(m,0),(n,1)}&=\Bigl\{ J_m+\sum_{i=m+1}^n \w_i\Bigr\} \vee \max_{m+1\le j\le n}\Bigl\{ \,\sum_{i=m+1}^j I_i +\sum_{i=j}^n  \w_i\Bigr\}, \quad n>m.  
\end{aligned} \ee
If the given  weights  \eqref{IJ99} come from the queueing setting of Section \ref{s:queue}, then   $H_{(m,0),(n,1)}=\wt G_n-G_m$.  But this connection is not needed for the present.  

The next lemma gives alternative formulas for $H$ in terms of the weights calculated  in \eqref{IJ100}. Pictorially, imagine $\wt I_i$ as a weight on the edge $((i-1,1),(i,1))$ and $\wt\w_i$ as a weight on the vertex $(i-1,0)$.  (The right diagram of Figure \ref{fig:IJ4} illustrates.)  In \eqref{H4} below,  a sum expression of the form $a_j+\dotsm+a_{j-1}$ is interpreted as zero.   \eqref{H7} makes sense also for $\ell=n$ in which case the right-hand side simplifies to $J_n$. 

\begin{lemma} \label{lm-H3}   Let $m\le n$.  Then  
\be\label{H4} 
H_{(m,0),(n,1)}= I_{m+1}+\dotsm+I_k+J_k+\wt I_{k+1}+\dotsm+\wt I_n
\quad \text{for each $k\in\lzb m,n\rzb$.}
\ee
For each $\ell\in\lzb m,n-1\rzb$, 
\be\label{H7} 
H_{(m,0),(n,1)}-H_{(m,0),(\ell,0)}=  
\max_{\ell+1\le j\le n}\Bigl\{ \,\sum_{i=\ell+1}^j \wt\w_i +\sum_{i=j}^n  \wt I_i\Bigr\}\vee
\Bigl\{ \,\sum_{i=\ell+1}^n \wt\w_i +J_n\Bigr\}. 
\ee
\end{lemma} 

Some observations before the proof.   By the two  top lines of 
\eqref{H},   equivalent to  \eqref{H4} are the increment formulas (for all $n>m$)  
\be\label{H10}   \wt I_n=H_{(m,0),(n,1)}-H_{(m,0),(n-1,1)}
\quad\text{and}\quad 
J_n=H_{(m,0),(n,1)}-H_{(m,0),(n,0)}. 
\ee
Taking $\ell=m$ in \eqref{H7} gives this dual representation for $H$: 
 \be\label{H12} 
H_{(m,0),(n,1)}=  
\max_{m+1\le j\le n}\Bigl\{ \,\sum_{i=m+1}^j \wt\w_i +\sum_{i=j}^n  \wt I_i\Bigr\}\vee
\Bigl\{ \,\sum_{i=m+1}^n \wt\w_i +J_n\Bigr\}. 
\ee

\begin{proof}[Proof of Lemma \ref{lm-H3}]  Let  $m<n$ and develop the definition \eqref{H}. As in \eqref{G-ind18}, 
 \be\label{H20} \begin{aligned}
H_{(m,0),(n,1)}&=\Bigl\{ J_m+\sum_{i=m+1}^{n-1} \w_i\Bigr\} \vee \max_{m+1\le j\le n-1}\Bigl\{ \,\sum_{i=m+1}^j I_i +\sum_{i=j}^{n-1}  \w_i\Bigr\} \vee \Bigl\{ \,\sum_{i=m+1}^n I_i  \Bigr\}  + \w_n \\
&= H_{(m,0),(n-1,1)} \vee H_{(m,0),(n,0)} +\w_n. 
\end{aligned}\ee
Set temporarily 
\begin{align*}  
 A_n=H_{(m,0),(n,1)}-H_{(m,0),(n-1,1)}
\quad\text{and}\quad 
B_n=H_{(m,0),(n,1)}-H_{(m,0),(n,0)} .   
\end{align*} 
 Then    \eqref{H20} gives the iterative equations 
\[   A_n=\w_n+(I_n-B_{n-1})^+
\quad\text{and}\quad 
B_n=\w_n+(B_{n-1}-I_n)^+ . \]  
Definition   \eqref{H} gives  $B_m=J_m$.  This starts an  induction.  Apply the equations above together with  \eqref{IJ100}  to obtain  $A_n=\wt I_n$ and $B_n=J_n$ for all $n\ge m+1$.    This establishes \eqref{H10}.   \eqref{H4} follows.  
 

We prove \eqref{H7} by induction as $\ell$ decreases.  The base case $\ell=n$ comes from the just proved $B_n=J_n$.  Assume \eqref{H7} for $\ell+1$. Then for $\ell$ the right-hand side of \eqref{H7} equals 
\begin{align*}
&\wt\w_{\ell+1}+  \Bigl\{ \, \sum_{i=\ell+1}^n  \wt I_i\Bigr\}\vee\max_{\ell+2\le j\le n}\Bigl\{ \,\sum_{i=\ell+2}^j \wt\w_i +\sum_{i=j}^n  \wt I_i\Bigr\}\vee
\Bigl\{ \,\sum_{i=\ell+2}^n \wt\w_i +J_n\Bigr\}\\
&\quad =H_{(m,0),(\ell+1,0)}\wedge H_{(m,0),(\ell,1)} - H_{(m,0),(\ell,0)}\\[4pt] 
&\qquad\qquad 
+  \bigl\{ H_{(m,0),(n,1)} - H_{(m,0),(\ell,1)}\bigr\} \vee \bigl\{  H_{(m,0),(n,1)}-H_{(m,0),(\ell+1,0)}\bigr\}  \\[4pt] 
&\quad=H_{(m,0),(n,1)}-H_{(m,0),(\ell,0)}. 
\end{align*}
In the first equality we used $\wt\w_{\ell+1}=I_{\ell+1}\wedge J_\ell$, \eqref{H10} and the induction assumption.  
\end{proof}

The last line of \eqref{H} and formula \eqref{H12} give dual representations of the quantity $H_{(m,0),(n,1)}$. The next lemma shows that equality persists if we drop the terms that involve $J$ from both formulas.  This statement is the crucial ingredient of Lemma \ref{m:D-lm4} below.

 \begin{lemma} \label{m:T-lm}  Let  $m\le n$ in $\Z$.  Assume given nonnegative weights $J_{m-1}$, $(I_i)_{m\le i\le n}$ and $(\w_i)_{m\le i\le n}$.    Compute $(\wt I_k, J_k, \wt\w_k)_{m\le k\le n}$ from \eqref{IJ100}.    Define  
 \be\label{m:T1}
 T_{m,n}=\max_{m\le j\le n} \Bigl\{ \,\sum_{i=m}^j I_i + \sum_{i=j}^n \w_i\Bigr\} 
 \quad\text{and}\quad 
  \wt T_{m,n}=\max_{m\le j\le n} \Bigl\{ \, \sum_{i=m}^{j} \wt\w_i +  \sum_{i=j}^n \wt I_i \Bigr\} .
 \ee
 Then $T_{m,n}=\wt T_{m,n}$.  
 \end{lemma} 
  
 \begin{proof}   
 The case $m=n$ is the identity  $I_n+\w_n=\wt\w_n+\wt I_n$ 
 that follows from    \eqref{IJ100}.  
 
 Let $n\ge m+1$ and assume by induction that $\wt T_{m,n-1}\le T_{m,n-1}$.  Develop the definitions.  
 \be\label{T3}  \begin{aligned}
 T_{m,n}=\max_{m\le j\le n-1} \Bigl\{ \,\sum_{i=m}^j I_i + \sum_{i=j}^{n-1} \w_i\Bigr\} \vee
 \Bigl\{ \,\sum_{i=m}^n I_i  \Bigr\}  + \w_n 
 = T_{m,n-1} \vee
 \Bigl\{ \,\sum_{i=m}^n I_i  \Bigr\}  + \w_n . 
 \end{aligned}\ee
Similarly 
 \be\label{T5}  \begin{aligned}
\wt T_{m,n}   & = \wt T_{m,n-1} \vee
 \Bigl\{ \,\sum_{i=m}^n \wt\w_i  \Bigr\}  + \wt I_n 
  =   \wt T_{m,n-1} \vee
 \Bigl\{ \,\sum_{i=m}^n (I_i\wedge J_{i-1})    \Bigr\}  + \w_n +(I_n-J_{n-1})^+ \\
 &\le  T_{m,n-1} \vee
 \Bigl\{ \,\sum_{i=m}^n I_i    \Bigr\}  + \w_n +(I_n-J_{n-1})^+.  
 \end{aligned}\ee
The induction assumption was used in the last step.
 
 \medskip 
 
 {\bf Case 1.}  $I_n\le J_{n-1}$.  This assumption kills the last term of   \eqref{T5} and gives   
 \begin{align*}
 \wt T_{m,n} 
 \le  T_{m,n-1} \vee
 \Bigl\{ \,\sum_{i=m}^n I_i    \Bigr\}  + \w_n  =  T_{m,n} . 
 \end{align*}

 \medskip 
 
 {\bf Case 2.}  $I_n>J_{n-1}$.   For this case induction is not needed.  We use the last-passage process $H_{(m-1,0), (\,\rdbullet\,, \,\rdbullet\,)}$. 
 Conservation law \eqref{IJ104} and \eqref{H10}  imply  
 \[  I_n>J_{n-1}   \ \Longleftrightarrow \   \wt I_n>J_n \ \Longleftrightarrow \  
 H_{(m-1,0),(n-1,1)} < H_{(m-1,0),(n,0)} .  \]  
 Then by \eqref{H20}  
 \[   H_{(m-1,0),(n,1)} =   H_{(m-1,0),(n,0)}+\w_n =  \sum_{i=m}^n I_i   +\w_n \le T_{m,n}. \]
On the other hand, by definition \eqref{H} 
 \[   H_{(m-1,0),(n,1)} =  \Bigl\{ J_{m-1}+\sum_{i=m}^n \w_i\Bigr\} \vee  T_{m,n}.  \]
 Hence $H_{(m-1,0),(n,1)} =T_{m,n}$.   
 By the dual formula \eqref{H12} 
 \[ H_{(m-1,0),(n,1)}=  
\wt T_{m,n} \vee
\Bigl\{ \,\sum_{i=m}^n \wt\w_i +J_n\Bigr\} \ge \wt T_{m,n}.  \]
We conclude that in Case 2, $\wt T_{m,n}\le T_{m,n}$.  

\medskip 

We have shown that $\wt T_{m,n}\le T_{m,n}$.   This suffices for the proof by the duality in Lemma \ref{lm-dual12} because the roles of $T_{m,n}$ and $\wt T_{m,n}$ can be switched around.  
 \end{proof}  
 
 The next lemma is the key property of the  queueing mapping $\Dop$ that underlies our results.  Its proof relies on  Lemma \ref{m:T-lm}.    Lemma \ref{m:T-lm} applies to the queueing setting  described in Section \ref{s:queue} because equations \eqref{m:R} and \eqref{m:IJ5} ensure  that the assumptions of  Lemma \ref{m:T-lm} are satisfied.

 \begin{lemma} \label{m:D-lm4}   Assume given three sequences  $I^2, I^1, \w^1\in\R_{\ge0}^\Z$  such that the queueing operations below 
 are well-defined.  
  Let $\w^2=\Rop(I^1, \w^1)$ as defined in \eqref{m:R}.    
Then  we have the identity 
 \be\label{m:D120}   \Dop\bigl( \Dop(I^2, \w^2), \Dop(I^1, \w^1)\bigr) = \Dop\bigl( \Dop(I^2, I^1), \w^1\bigr). 
 \ee 
 \end{lemma}

\begin{proof}   Choose $G^1$ and $G^2$ so that $I^t_k=G^t_k-G^t_{k-1}$ for $t=1,2$.   Let 
\[  H_j=\sup_{\ell:\,\ell\le j}  \Bigl\{  G^2_\ell+\sum_{i=\ell}^j I^1_i\Bigr\}    \]  
and then 
\be\label{m:D121}  \wt H_k=\sup_{j:\,j\le k}  \Bigl\{  H_j+\sum_{i=j}^k \w^1_i\Bigr\} 
=  \sup_{\ell:\,\ell\le k}  \Bigl\{  G^2_\ell+ \max_{j: \ell\le j\le k}  \Bigl[  \,\sum_{i=\ell}^j I^1_i +\sum_{i=j}^k \w^1_i\Bigr]\,\Bigr\} .   \ee  
The sequence $(\wt H_k-\wt H_{k-1})_{k\in\Z}$ is the output $\Dop\bigl( \Dop(I^2, I^1), \w^1\bigr)$. 

For the left-hand side of \eqref{m:D120} define first for $\Dop(I^t, \w^t)$ the sequence
\[   \wt G^t_j=\sup_{\ell:\,\ell\le j}  \Bigl\{  G^t_\ell+\sum_{i=\ell}^j \w^t_i\Bigr\}, \quad t\in\{1,2\}.    \]
Set $\wt I^1_k=\wt G^1_k-\wt G^1_{k-1}$.  The  output $\Dop\bigl( \Dop(I^2, \w^2), \Dop(I^1, \w^1)\bigr) $ is given by the increments of the sequence 
\be\label{m:D122} 
\wh H_k= \sup_{j:\,j\le k}  \Bigl\{  \wt G^2_j+\sum_{i=j}^k \wt I^1_i\Bigr\}   
=   \sup_{\ell:\,\ell\le k}  \Bigl\{  G^2_\ell + \max_{j: \ell\le j\le k}  \Bigl[  \, \sum_{i=\ell}^j \w^2_i +\sum_{i=j}^k \wt I^1_i \Bigr]\,\Bigr\}   . 
\ee 
The rightmost members of lines \eqref{m:D121} and \eqref{m:D122} are equal  because the innermost maxima over the quantities in  square brackets $[\dotsm]$  agree, by  Lemma \ref{m:T-lm}.    We have shown that $\wt H=\wh H$ and thereby proved the lemma.  
\end{proof}

We extend Lemma \ref{m:D-lm4}  inductively.  

 \begin{lemma} \label{m:D-lm6}  Let $n\ge 2$ and  assume given  $n+1$  sequences  $ I^1, I^2, \dotsc, I^n, \w^1 \in\R_{\ge0}^\Z$  such that  all the queueing operations below  are well-defined.   Define iteratively  
   \be\label{m:R3}    \w^j=\Rop(I^{j-1}, \w^{j-1}) \quad\text{ for $j=2,\dotsc,n$. } \ee
Then  we have these identities for $1\le k\le n-1$: 
  \be\label{m:D124}  \begin{aligned}
  &  \Dop^{(n+1)}(I^n, I^{n-1},\dotsc, I^1 ,\w^1) \\
    &\qquad\qquad
    =\Dop^{(k+1)}\bigl( \Dop^{(n-k+1)}[ I^n, \dotsc, I^{k+1}  , \w^{k+1}],  \Dop(I^{k}, \w^{k}), \dotsc, \Dop(I^1, \w^1)\bigr).  
      \end{aligned}   \ee 
 \end{lemma}

\begin{proof}   The case $n=2$ is Lemma \ref{m:D-lm4}.  

Let $n\ge 3$ and assume that  the claim of the lemma holds when $n$ is replaced by $n-1$, for $1\le k\le n-2$.  We prove the claim for $n$.  

First  the case $k=1$, beginning with the right-hand side of \eqref{m:D124}:
\begin{align*}
& \Dop^{(2)}\bigl( \Dop^{(n)}(I^n,\dotsc, I^{2}, \w^{2}), \Dop(I^1, \w^1)\bigr)\\
&=\Dop\bigl( \Dop\bigl[  \Dop^{(n-1)}(I^n, \dotsc, I^{2} ) , \w^{2}\bigr],  \Dop(I^1, \w^1)\bigr)  \\
&=\Dop\bigl( \Dop\bigl[  \Dop^{(n-1)}(I^n, \dotsc, I^{2} ) , I^{1}\bigr],    \w^1  \bigr)  \\
&= \Dop\bigl( \Dop^{(n)}(I^n, \dotsc, I^1), \w^1\bigr)= \Dop^{(n+1)}(I^n,\dotsc, I^1 ,\w^1\bigr) . 
\end{align*} 
The first and last two equalities above are from  definition \eqref{m:D67} of $\Dop^{(n)}$, and the middle equality is Lemma \ref{m:D-lm4}.  

Now let $2\le k\le n-1$.  
   The first equality below is definition \eqref{m:D67} for  $\Dop^{(k+1)}$.    The second equality is the induction assumption.  
\begin{align*}
&\Dop^{(k+1)}\bigl( \Dop^{(n-k+1)}( I^n, \dotsc, I^{k+1}  , \w^{k+1}),  \Dop(I^{k}, \w^{k}), \dotsc, \Dop(I^1, \w^1)\bigr)\\
&=D\bigl(\Dop^{(k)}\bigl[ \Dop^{(n-k+1)}( I^n, \dotsc, I^{k+1}  , \w^{k+1}),  \Dop(I^{k}, \w^{k}), \dotsc,  \Dop(I^2, \w^2)\bigr] ,  \Dop(I^1, \w^1)\bigr)\\
&=D\bigl(    \Dop^{(n)}[ I^n, \dotsc, I^{2} ,  \w^{2}  ] , \Dop(I^1, \w^1)\bigr).  
\end{align*}
The last line above is the same as the first  line of the previous display.  The   calculation is   completed  as was done there.  
\end{proof}

In particular, for $k=n-1$ \eqref{m:D124}    gives 
 \be\label{m:D126} \begin{aligned}
 \Dop^{(n+1)}(I^n,\dotsc, I^1 ,\w^1) 
    =\Dop^{(n)}\bigl( \Dop(I^{n}  , \w^{n}),    \dotsc, \Dop(I^1, \w^1)\bigr)
   \end{aligned}  \ee 
 and  for $k=1$ 
  \be\label{m:D127}  \Dop^{(n+1)}(I^n, \dotsc, I^1 ,\w^1) 
    =\Dop\bigl( \Dop^{(n)}\bigl[ I^n, \dotsc, I^{2}  , \w^{2}\bigr],   \Dop(I^1, \w^1)\bigr).   \ee

\section{Multiclass processes} \label{s:m-class}

The distribution $\mu^{(1,\rho_1,\dotsc,\rho_n)}$  of the   $(n+1)$-tuple    $(\wb\Yw_{\!t}\,, \Busv{\rho_1}{\evec_1}{t}, \dotsc, \Busv{\rho_n}{\evec_1}{t})$  given in Theorem \ref{B-th1}   is deduced through studying two  multiclass   LPP processes. 
  Fix a positive integer $n$, the number of levels or classes.  
  We define two discrete-time Markov processes    on $n$-tuples of sequences, the {\it multiline process}  and the {\it coupled process}.   
Their state space is 
\be\label{m:YYth} 
 \cA_{n}=\Bigl\{I=(I^1, I^2, \dotsc, I^n)\in (\R_{\ge0}^\Z)^n :  \,\forall\, i\in[n] , \   \lim_{m\to-\infty} \frac1{\abs m}\sum_{k=-m}^0 I^i_k> 1\; \Bigr\} . 
 \ee
  At each step their evolution is driven by an independent sequence of  i.i.d.\ exponential weights: so   assume that 
  \be\label{m:exp}
\text{$\w=(\w_k)_{k\in\Z}$ is a sequence of i.i.d.\ variables $\w_k\sim$ Exp($1$). }  
 \ee

\subsection{Multiline process}
At time $t\in\Z_{\ge0}$  the state of the   {\it multiline process} is denoted by   $I(t)=(I^1(t),\dotsc, I^n(t))\in\cA_{n}$.
 The one-step evolution from time $t$ to $t+1$ is defined as follows in terms of the mappings \eqref{m:DSR}. 
    Given the time $t$ configuration $I(t)=I=(I^1,I^2,\dotsc, I^n)$ in the space $\cA_{n}$  and independent  driving weights  $\w$, define the time $t+1$  configuration $I(t+1)=\bar I=(\bar I^1,\bar I^2,\dotsc, \bar I^n)$ iteratively as follows: 
\be\label{m:I4}  \begin{aligned}
	&\text{first  set $\w^1=\w$ and  then $\bar I^1  = \Dop(I^1, \w^1)$; } \\
	&\text{then for each  $i=2, 3, \dotsc, n$: }\\
	&\qquad\qquad\text{ first $\w^i=\Rop(I^{i-1}, \w^{i-1})$ and then  $\bar I^i  = \Dop(I^i, \w^i)$. }
\end{aligned}\ee 
  Thus the driving sequence $\w$ acts on the first  line $I^1$ directly,  and is then  transformed at each stage before it is passed to the next line.  Lemma \ref{lm-D14} guarantees that, for almost every $\w$ from \eqref{m:exp},   the Ces\`aro limit  
${\ddd\lim_{m\to-\infty}} {\abs m}^{-1}  \sum_{k=m}^0\w^i_k =1$ holds for each $i\in[n]$ and    the new state $\bar I$ lies in $\cA_{n}$. 


\begin{theorem} \label{m:thm-I} Assume \eqref{m:exp}.   Then for each $\rho=(\rho_1,\dotsc,\rho_n)\in (1,\infty)^n$,  the product measure  $\nu^\rho$ defined in \eqref{nu5}  is invariant for the multiline process $(I(t))_{t\in \Z_{\ge0}}$.  
\end{theorem}

Theorem \ref{m:thm-I}  follows from Lemma \ref{m:Lem-I} in Appendix \ref{a:exp}:  induction on $k$ shows that $\bar I^1,\dotsc, \bar I^k$, $\w^{k+1}$, $I^{k+1},\dotsc, I^n$ are independent with $\bar I^i\sim\nu^{\rho_i}$, $\w^{k+1}\sim\nu^{1}$, $I^j\sim\nu^{\rho_j}$.   We do not have proof that $\nu^\rho$ is the unique translation-ergodic stationary distribution with mean vector $\rho$, but have no reason to doubt this either. 


\subsection{Coupled process}
At time $t\in\Z_{\ge0}$  the state of the   {\it coupled process}  is denoted by $\eta(t)=(\eta^1(t),   \dotsc, \eta^n(t))\in\cA_{n}$  where again  $\eta^i(t)=(\eta^i_k(t))_{k\in\Z}$.    
The evolution is simple:  the queueing operator $\Dop$  acts on each sequence $\eta^i$ with service times $\w$:  
\be\label{m:eta4}  \eta(t+1)=\bigl( \Dop(\eta^1(t), \w) ,  \Dop(\eta^2(t), \w) ,  \dotsc, \Dop(\eta^n(t), \w)   \bigr) .  \ee 

 We call $\eta(t)$  the coupled process because it  lives also on the smaller state space   $\cX_{n}\cap\cA_{n}$  (recall \eqref{m:XXth})  where the  sequences $\eta^i$   are coupled  so that $\eta^{i-1}\le \eta^i$. 
This is the case  relevant for the Busemann  processes because the latter are monotone (recall \eqref{B-mono}).   
   Inequality \eqref{m:D8} and Lemma \ref{lm-D13}  ensure that 
     the Markovian evolution $\eta(\cdot)$  is well-defined on $\cX_{n}\cap\cA_{n}$.  However,  since the mapping \eqref{m:eta4} is well-defined for more general states,  we consider it on the larger state space $\cA_{n}$ of \eqref{m:YYth}.  

  To state an invariance and uniqueness    theorem  for  all parameter vectors $\rho\in(1,\infty)^n$ we extend $\mu^\rho$ of  \eqref{mu5},   by ordering $\rho$ and  by requiring that $\eta^i=\eta^{i+1}$ if $\rho_i=\rho_{i+1}$.   This is necessary because   the mapping $\Daop^{(n)}$ in \eqref{mu5} cannot be applied if some  $\rho_i=\rho_{i+1}$.  For if $I$ and $\w$ are both i.i.d.\ Exp$(\rho^{-1})$ sequences, then $\wt G$ in \eqref{m:800} is identically infinite because it equals  a random constant plus the supremum of a symmetric random walk.   
  

\begin{definition}\label{d:murho}  Let $\rho=(\rho_1,\rho_2,\dotsc, \rho_n)\in(0,\infty)^n$. The probability measure  $\mu^\rho$ on the space $(\R_{\ge0}^\Z)^n$  is defined as follows. 
 
\smallskip 

{\rm (i)}  If $0<\rho_1<\rho_2<\dotsm<\rho_n$ then apply \eqref{mu5}.  

\smallskip 

{\rm (ii)}  If $0<\rho_1\le\rho_2\le\dotsm\le\rho_n$,  there exist $m\in[n]$, 
 a vector  $\sigma=(\sigma_1,\dotsc,\sigma_m)$ such that  $0<\sigma_1<\dotsm<\sigma_m$,  and indices  $1=i_1<i_2<\dotsm<i_m<i_{m+1}=n+1$ such that  $\rho_{i_\ell}=\dotsm=\rho_{i_{\ell+1}-1}=\sigma_\ell$ for $\ell=1,\dotsc,m$.   Let $I\sim\nu^\sigma$, $\zeta=\Daop^{(m)}(I)$, and then define $\eta=(\eta^1,\dotsc,\eta^n)\in\cX_{n}$ by   $\eta^{i_\ell}=\dotsm=\eta^{i_{\ell+1}-1}=\zeta^\ell$ for $\ell=1,\dotsc,m$.   Define  $\mu^\rho$ to be the distribution of $\eta$.  

\smallskip 

{\rm (iii)}    For general $\rho=(\rho_1,\dotsc, \rho_n)\in(0,\infty)^n$, choose   a permutation $\pi$ such  that   $\pi\rho=(\rho_{\pi(1)},\dotsc, \rho_{\pi(n)})$ satisfies  $\rho_{\pi(1)}\le\rho_{\pi(2)}\le\dotsm\le\rho_{\pi(n)}$.  Let $\pi$ act on   weight configurations $\eta=(\eta^1,\dotsc,\eta^n) $  by $\pi\eta=(\eta^{\pi(1)},\dotsc,\eta^{\pi(n)})$.  Define $\mu^\rho=\mu^{\pi\rho}\circ\pi^{-1}$, or more explicitly 
\[   E^{\mu^\rho}[f]=    E^{\mu^{\pi\rho}}[f(\pi\eta)]= E^{\mu^{\pi\rho}\circ\pi^{-1}}[f] \] 
for bounded Borel functions $f$  on $(\R_{\ge0}^\Z)^n$, where the measure $\mu^{\pi\rho}$ is the one defined in step {\rm(ii)}.  

\end{definition} 

 If there is more than one ordering permutation in step (iii),   there are identical  sequences whose ordering among themselves is immaterial.  If $\rho\in (1,\infty)^n$ then $\mu^\rho$ is supported on the space $\cA_{n}$ of \eqref{m:YYth}.   The next existence and uniqueness theorem is proved in Section \ref{s:inv-cpl}. 


 
 \begin{theorem} \label{m:thm-eta}  Assume \eqref{m:exp}. 
 
{\rm (i)  Invariance.}  Let    $\rho=(\rho_1,\rho_2,\dotsc, \rho_n)\in(1,\infty)^n$.   Then  the probability measure $\mu^\rho$  of Definition \ref{d:murho}  is invariant for the  Markov chain  $(\eta(t))_{t\in \Z_{\ge0}}$ defined by \eqref{m:eta4}.  
   
  
 {\rm (ii) Uniqueness.}  
  Let $\wt\mu$ be a  translation-ergodic probability measure on $\cA_{n}$  under which  coordinates $\eta^i_k$ have  finite means $\rho_i=E^{\wt\mu}[\eta^i_k]>1$.  If $\wt\mu$ is invariant for the   process $\eta(t)$, then $\wt\mu=\mu^\rho$ for $\rho=(\rho_1,\dotsc,\rho_n)\in(1,\infty)^n$.  
 
\end{theorem}

\subsection{Stationary multi-class LPP on the upper half-plane}
We reformulate the coupled process as a   multiclass CGM 
  on the upper half-plane.  
    Fix the number  $n$ of classes.   Assume given i.i.d.\ Exp(1)  random  weights $\{\w_x\}_{x\in\Z\times\Z_{>0}}$, and  an initial configuration $\eta(0)=(\eta^1(0),\dotsc,\eta^n(0))\in\cA_{n}$ independent of $\w$.    Define a vector of LPP processes  $G_x=(G^1_x,\dotsc, G^n_x)$   for $x\in\Z\times\Z_{\ge0}$ as follows.   First choose initial functions $\{G^i_{(k,0)}\}_{k\in\Z}$  with the property $\eta^i_k(0)=G^i_{(k,0)}-G^i_{(k-1,0)}$.  Then for $(k,t)\in\Z\times\Z_{>0}$ define 
\begin{equation}\label{Def:G^i}
G^i_{(k,t)}=\sup_{j:\,j\leq k}\{G^i_{(j,0)} + G_{(j,1),(k,t)}\},
\end{equation}
  where $G_{x,y}$ is the usual LPP process of   \eqref{m:G807} with weights $\Yw_x(\w)=\w_x$.    Then lastly define the process $\eta(t)=(\eta^1(t),\dotsc,\eta^n(t))$  for $t\in\Z_{>0}$ as the increments:   
  \be\label{G-eta}  \eta^i_k(t)=G^i_{(k,t)}-G^i_{(k-1,t)}  \quad \text{ for $i\in[n]$ and $k\in\Z$. }
 \ee
    
 \begin{theorem} \label{m:thm-eta-2}      Let    $\rho=(\rho_1,\rho_2,\dotsc, \rho_n)\in(1,\infty)^n$.    Then $\mu^\rho$ of Definition \ref{d:murho} is an   invariant measure of the increment process $\eta(\cdot)$  defined above by \eqref{G-eta} in the multiclass exponential corner growth model.  Measure $\mu^\rho$ is the unique invariant measure for $\eta(\cdot)$ among  translation-ergodic probability measures on $\cA_{n}$ with means given by $\rho$. 
\end{theorem}

This follows from   Theorem \ref{m:thm-eta} simply by noting that   \eqref{G-eta}    can be reformulated  inductively   as 
    \be\label{m:eta4.8}  \eta(t)=\bigl( \Dop(\eta^1(t-1), \wb\w_t) ,  \Dop(\eta^2(t-1), \wb\w_t) ,  \dotsc, \Dop(\eta^n(t-1), \wb\w_t)   \bigr) , 
  \quad t\in\Z_{>0},   
     \ee 
 where $\wb\w_t=\{\w_{(k,t)}\}_{k\in\Z}$ is the sequence of weights  on level $t$.

%
 

\subsection{Invariant distribution for the coupled  process}\label{s:inv-cpl}

This section proves Theorem \ref{m:thm-eta}.    We separate  the invariance of $\mu^\rho$  and the uniqueness in Theorems \ref{th:eta-ex} and \ref{th:eta-un} below.    Their combination establishes Theorem \ref{m:thm-eta}.   
   The proof of the next theorem shows how the invariance  of $\mu^\rho$ for $\eta(t)$ follows from the invariance of $\nu^\rho$ for $I(t)$ and the fact  that the mapping $\Daop^{(n)}$ intertwines the evolutions of $I(t)$ and $\eta(t)$.

\begin{theorem}  \label{th:eta-ex} Let 
   $\rho=(\rho_1,\rho_2,\dotsc, \rho_n)\in(1,\infty)^n$.     Then $\mu^\rho$  of Definition \ref{d:murho}  is an invariant distribution  for the $(\R_{\ge0}^\Z)^n$-valued  Markov chain $\eta(t)$ defined by \eqref{m:eta4}.  
\end{theorem} 

\begin{proof}
The general claim follows from  the case $1<\rho_1<\rho_2<\dotsm<\rho_n$ because  permuting the $\{\eta^i\}$ or setting $\eta^i=\eta^{j}$ produces the exact same change in the image of   the mapping in \eqref{m:eta4}.  

So assume $1<\rho_1<\rho_2<\dotsm<\rho_n$.  
Given a driving sequence $\w$, 
denote by  $\Saop^\w$ and $\Taop^\w$  the  mappings  on the state spaces that encode a single    temporal evolution step  of the processes $I(\cdot)$ and $\eta(\cdot)$.   In other words,  the mapping  from time $t$ to $t+1$ defined by \eqref{m:I4} for the multiline process  is encoded as    $I(t+1)=\Saop^\w( I(t))$.    For   the coupled process the step in \eqref{m:eta4} is encoded as   $\eta(t+1)=\Taop^\w( \eta(t))$.   Let   $\Daop=\Daop^{(n)}$ denote the mapping \eqref{m:eta5} that constructs the coupled configuration from the multiline configuration.   Let $\Daop_k$, $\Saop^\w_k$ and $\Taop^\w_k$ denote the $k$th $\R_{\ge0}^\Z$-valued  coordinates of  the images of these  mappings. 

 Let $I\sim\nu^\rho$ be a multiline configuration  with product exponential distribution $\nu^\rho$.   
We need to show that if $\eta$ has the distribution $\mu^\rho$  of $\Daop(I)$, then so does $ \Taop^\w( \eta)$ when $\w$ is an independent sequence of i.i.d.\ Exp($1$) weights.   For the argument we can assume that $\eta=\Daop(I)$.  
  As before let $\w^1=\w$ and   iteratively  
    $\w^j=\Rop(I^{j-1}, \w^{j-1})$ for $j=2, 3,\dotsc,n$.    The fourth  equality below is \eqref{m:D126}.  The other equalities are consequences of definitions.   
\begin{align*}
 \Taop^\w_k(\eta)&=\Dop(\eta^k,\w) =\Dop\bigl(   \Dop^{(k)}(I^k, \dotsc, I^1), \w^1 \bigr) 
  = \Dop^{(k+1)}(I^k, \dotsc, I^1, \w^1)\\
&\overset{\eqref{m:D126}}= \;  \Dop^{(k)}\bigl( \Dop(I^k, \w^k),   \Dop(I^{k-1}, \w^{k-1}), \dotsc,  \Dop(I^1, \w^1)\bigr)  \\
&= \Dop^{(k)} \bigl(\Saop^\w_k(I) ,\Saop^\w_{k-1}(I) ,\dotsc, \Saop^\w_1(I)\bigr)   = \Daop_k(\Saop^\w(I)).  
\end{align*} 
Since the above works for all coordinates $k\in[n]$,  we have $\Taop^\w(\eta) = \Daop(\Saop^\w(I))$.   Since $\eta=\Daop(I)$, we have verified the intertwining 
\be\label{int-twi} \Taop^\w(\Daop(I))= \Daop(\Saop^\w(I)).  \ee
 By Theorem \ref{m:thm-I},  
$\Saop^\w(I)\overset{d}=I\sim\nu^\rho$.  Consequently  $\Taop^\w(\eta)\overset{d}=\Daop(I)\sim\mu^\rho$. 
\end{proof}

  \begin{theorem}  \label{th:eta-un}  Assume \eqref{m:exp}. 
  Let $\wt\mu$ be a  translation-ergodic probability measure on $\cX_{n}$  under which each coordinate $\eta^i_k$ has a finite mean.  If $\wt\mu$ is invariant for the coupled process $\eta(t)$, then $\wt\mu=\mu^\rho$ for the mean vector $\rho$ of $\wt\mu$.  
\end{theorem}

We prove Theorem \ref{th:eta-un} following Chang  \cite{MR1303943}, by showing that the evolution contracts the $\wb\rho$  distance between stationary and ergodic sequences. Let $\eta=(\eta_k)_{k\in\Z}$ and $\xi=(\xi_k)_{k\in\Z}$ be stationary processes taking values  in $\R_{\ge0}^n$.  Their $\wb\rho$ distance is defined by 
\begin{equation}\label{rhobar_s}
\wb\rho(\eta,\xi)=\inf_{(X,Y)\in \mathcal{M}}\mE[\,\abs{X_0-Y_0}_{1}],
\end{equation}
where $\mathcal{M}$ is the set of jointly defined  stationary  sequences $(X,Y)=(X_k,Y_k)_{k\in\Z}$ such that $X\overset{d}{=} \eta$ and $Y\overset{d}{=} \xi$, $\mE$ is the  expectation on the probability space on which the coupling $(X,Y)$ is defined,
and $|\cdot |_{1}$ is the $\ell^1$ distance on $\R_{\ge0}^n$.
 
From \cite[Theorem 9.2]{MR2840299}    we know   that (i) $\wb\rho$ induces a  metric on 
the space of translation-invariant distributions  and  (ii) if $\eta$ and $\xi$ are both ergodic, there exists a   jointly stationary and ergodic pair  $(X,Y)$ at which  the infimum in \eqref{rhobar_s} is attained. 

The following  is a straight-forward generalization of Theorem 2.4 of \cite{MR1303943} to $\R_{\ge0}^n$-valued   stationary and ergodic 
sequences  $\eta=(\eta^1, \dots, \eta^n)$ and $\xi=(\xi^1, \dots, \xi^n)$  where  $\eta^i=(\eta^i_k)_{k\in\Z}$ and $\xi^i=(\xi^i_k)_{k\in\Z}$ are random elements of $\R_{\ge0}^\Z$.  Let 
\[    \wt\eta= (\wt\eta^1, \dots, \wt\eta^n)=\bigl( \Dop(\eta^1,\w), \dotsc, \Dop(\eta^n,\w)\bigr) \]
and similarly  $\wt\xi=(\wt\xi^1, \dots, \wt\xi^n)$ denote the outcome of applying the queueing map $\Dop(\cdot, \w)$ to each sequence-valued coordinate.  

\begin{proposition}\label{P:Contract}
 Let  $\w$ satisfy \eqref{m:exp}.  Let the $\R_{\ge0}^n$-valued   stationary and ergodic processes 
 $\eta$ and $\xi$ be  
independent of $\w$ and have finite means that satisfy $\E[\eta^i_k]=\E[\xi^i_k]=\lambda_i>1$ for $i\in[n]$ and $k\in\Z$.  
 Then
\begin{equation}\label{contract}
\wb\rho(\wt\eta,\wt\xi)\leq \wb\rho(\eta,\xi).
\end{equation} 
 If  $\eta$ and $\xi$  have different distributions the inequality in \eqref{contract} is strict. 
\end{proposition}

Before the proof we complete the proof of Theorem \ref{th:eta-un}.  Let $\rho=E^{\wt\mu}[\eta_0]$ be the mean vector of $\wt\mu$.   Let $\eta\sim\mu^\rho$ and $\xi\sim\wt\mu$. By the known invariance of $\mu^\rho$ and the assumed invariance of $\wt\mu$,  $\wt\eta\deq\eta$ and $\wt\xi\deq\xi$.  Hence $\wb\rho(\wt\eta,\wt\xi)=\wb\rho(\eta,\xi)$.  The last statement of Proposition  \ref{P:Contract} forces $\wt\mu=\mu^\rho$.  


\begin{proof}[Proof of Proposition \ref{P:Contract}] 

Let $(X,Y)=((X^1,\dotsc,X^n), (Y^1,\dotsc,Y^n))$ be an arbitrary $\R_{\ge0}^{2n}$-valued    jointly stationary and ergodic process  with marginals $X\overset{d}{=} \eta$ and $Y\overset{d}{=} \xi$,   independent of the weights $\w$, with $(X,Y,\w)$ coupled together under a probability measure $\mP$ with expectation $\mE$.  As above, write   $\wt X^i=(\wt X^i_k)_{k\in\Z}=\Dop(X^i,\w)$ and $\wt Y^i=(\wt Y^i_k)_{k\in\Z}=\Dop( Y^i,\w)$ for the action of the queueing operator on the individual sequences $X^i=(X^i_k)_{k\in\Z}$ and   $Y^i=(Y^i_k)_{k\in\Z}$.     Inequality \eqref{contract} follows from showing  
\begin{equation}\label{contract_2}
\mE[|\wt X_0-\wt Y_0|_{1}] \leq \mE[|X_0- Y_0|_{1}]. 
\end{equation}

Define the process $Z$ by $ Z^i_k=X^i_k \vee  Y^i_k$.  Then
\begin{equation}\label{etaxi}
|X_0- Y_0|_{1}=\sum_{i=1}^n|X^i_0- Y^i_0|=\sum_{i=1}^n (2 Z^i_0-X^i_0- Y^i_0).
\end{equation}
Let   $\wt Z^i =\Dop( Z^i,\w)$. Then $\wt Z^i\geq \wt X^i\vee \wt Y^i$ by monotonicity \eqref{m:D8}. Hence
\be\label{wtetaxi}\begin{aligned}
|\wt X_0-\wt Y_0|_{1}&=\sum_{i=1}^n|\wt X^i_0-\wt Y^i_0|=\sum_{i=1}^n \bigl(2(\wt X^i_0\vee \wt Y^i_0)-\wt X^i_0-\wt Y^i_0\bigr)  
\leq \sum_{i=1}^n \bigl(2\wt Z^i_0-\wt X^i_0-\wt Y^i_0\bigr).  
\end{aligned}\ee

The triple $(X,Y,\w)$ is jointly stationary and ergodic because $\w$ is an i.i.d.\ process independent of the ergodic process $(X,Y)$.   Consequently, as translation-respecting mappings of ergodic processes,  both  $(X,Y,Z,\w)$ and   $(\wt X, \wt Y,\wt Z)$ are jointly stationary and ergodic.    The queueing stability condition $\mE (X^i_0)>\mE(\w_0)$ implies  $\mE (\wt X^i_0)=\mE (X^i_0)$, and   by the same token    $\mE (\wt Y_0^i)=\mE ( Y_0^i)$ and $\mE (\wt Z_0^i)=\mE ( Z_0^i)$. This goes back to Loynes \cite{loyn-62} and follows also from Lemma \ref{lm-D14} in Appendix \ref{a:queue}.  
 Taking expectations on both sides of \eqref{etaxi} and \eqref{wtetaxi} gives \eqref{contract_2}.

\medskip

For the strict inequality  assume that $\eta$ and $\zeta$ are not equal in distribution and let $(X,Y)$  be a jointly ergodic  pair that gives the minimum in \eqref{rhobar_s}.   To deduce  the strict inequality
\be \label{strict}
\sum_{i=1}^n\mE[\wt X^i_0\vee \wt Y^i_0] < \sum_{i=1}^n\mE(\wt Z^i_0). 
\ee
 we can tap directly into the  proof of  part (ii)   of  Theorem 2.4 in \cite{MR1303943}, once we show that  $X^i$ and $Y^i$ must cross for some $i\in[n]$.   $X^i$ and $Y^i$  {\it cross} if with probability one there exist $k,\ell\in\Z$ such that $X^i_k>Y^i_k$ and $X^i_\ell<Y^i_\ell$. 

  Suppose $X^i$ and $Y^i$ do not cross.  Then  
  $\mP(\{X^i\ge Y^i\}\cup\{X^i\le Y^i\})=1$. 
  We show that this implies 
  $X^i=Y^i$ a.s.   This gives us the contradiction needed, since  $X^i=Y^i$  for all $i\in[n]$ implies that $\eta\deq\xi$.  
  
To show  $X^i=Y^i$ a.s., write  $\{X^i=Y^i\}^c=A^+\cup A^-$ a.s.\  for 
\begin{align*} 
A^+&=\{X^i\geq Y^i \text{ and } X^i_k>Y^i_k \text{ for some }k\in \Z\} \\
\text{ and}\quad  A^-&=\{X^i\leq Y^i \text{ and } X^i_k<Y^i_k \text{ for some }k\in \Z\}.
\end{align*} 
$A^+$ is a shift-invariant event.   By the joint ergodicity of $(X,Y)$  and $\mE(X^i_0-Y^i_0)=0$, 
\begin{align*}
0&= \ind_{A^+}\cdot  \lim_{n\to\infty}\frac{1}{2n+1}\,\sum_{-n \leq k\leq n}(X^i_k-Y^i_k)\\
&= \lim_{n\to\infty}\frac{1}{2n+1}\,\sum_{-n \leq k\leq n}(X^i_k-Y^i_k)\cdot  \ind_{\theta_{-k}A^+}  =   \mE[(X^i_0-Y^i_0)\cdot  \ind_{A^+}] \quad\text{a.s. }   
\end{align*} 
Thus $X^i_0=Y^i_0$ a.s.\ on $A^+$.   By the shift-invariance of $A^+$,  $X^i_k=Y^i_k$ a.s.\ on $A^+$ for all $k\in\Z$.  But then it must be that $\mP(A^+)=0$.  Similarly $\mP(A^-)=0$.  

To summarize, we have shown that some  $X^i$ and $Y^i$ must cross. Following the proof on p.~1131-1132 of \cite{MR1303943} gives the strict inequality \eqref{strict}.   The connection between the notation of \cite{MR1303943} and ours is $S_k=\w_k$, $(T^1_{1,k-1},\,T^1_{2,k-1})=(X^i_k, \wt X^i_k)$ and  $(T^2_{1,k-1},\,T^2_{2,k-1})=(Y^i_k, \wt Y^i_k)$.  
\end{proof}


\section{Proofs of the results for Busemann functions} 
\label{s:Bpf}

 We prove the theorems of Section \ref{s:Bus} in the order in which they were stated. 

\subsection{Continuity of $\mu^\rho$ and distribution of the Busemann process} 
  
 \begin{proof}[Proof of the continuity claim of Theorem \ref{murho-th1}]    Fix $\rho=(\rho_1,\dotsc,\rho_n)$ such that  $0<\rho_1<\dotsc<\rho_n$.   Let  $\{\rho^h\}_{h\in\Z_{>0}}$  be a sequence of parameter vectors such that  $\rho^h=(\rho^h_1,\dotsc,\rho^h_n)\to(\rho_1,\dotsc,\rho_n)$ as $h\to\infty$.   We construct variables $\eta^h\sim\mu^{\rho^h}$ and $\eta\sim\mu^{\rho}$ such that $\eta^h\to\eta$ coordinatewise almost surely.  
 
Let $I=(I^1,\dotsc,I^n)\sim\nu^\rho$ and define  $I^{h,i}_k=(\rho^h_i/\rho_i)I^i_k$.  Then $I^h=(I^{h,1},\dotsc,I^{h,n})\sim\nu^{\rho^h}$ and we have the pointwise limits $I^{h,i}_k\to I^i_k$ for all $i\in[n]$ and $k\in\Z$ as $h\to\infty$.   Furthermore,  the assumption in \eqref{D88} holds: 
\be\label{D88.88} 
\varlimsup_{\substack{m\to-\infty\\h\to\infty}}    \,\biggl\lvert \frac1{\abs m}\sum_{j=m}^0 I^{h,i}_j-\rho_i\,\biggr\rvert=0 \quad\text{almost surely}  \quad \forall i\in[n]. 
\ee

Let $\eta^h=\Daop^{(n)}(I^h)$ and  $\eta=\Daop^{(n)}(I)$. Apply Lemma \ref{lm-D33} repeatedly to show that $\eta^h\to\eta$ coordinatewise almost surely:     

(1)  $\eta^{h,1}=I^{h,1}\to I^1=\eta^1$ needs no proof.   
 
(2) Lemma \ref{lm-D33} gives the limit  $\eta^{h,2}=\Dop(I^{h,2}, I^{h,1}) \to  \Dop(I^{2}, I^1)=\eta^2$ and that $\Dop(I^{h,2}, I^{h,1})$ satisfies the hypotheses of the lemma. 
  
(3) For  $\eta^{h,3}=\Dop^{(3)}(I^{h,3},I^{h,2} , I^{h,1})=\Dop( \Dop(I^{h,3},I^{h,2}) , I^{h,1})$,  by case (2),  $\Dop(I^{h,3},I^{h,2})$ satisfies the hypotheses of Lemma \ref{lm-D33}. Then Lemma \ref{lm-D33}   gives   $\Dop( \Dop(I^{h,3},I^{h,2}) , I^{h,1})\to \Dop( \Dop(I^{3},I^{2}) , I^{1})$ and that  $\Dop^{(3)}(I^{h,3},I^{h,2} , I^{h,1})$ satisfies the hypotheses of Lemma \ref{lm-D33}. 
  
(4) Proceed by induction.  From the case of $i-1$ sequences,  $\Dop^{(i-1)}(I^{h,i},I^{h,i-1} , \dotsc , I^{h,2})$ satisfies the hypotheses of Lemma \ref{lm-D33}.  Apply the Lemma to conclude that the mapping for $i$ sequences obeys the limit 
\[  \eta^{h,i}=   \Dop^{(i)}(I^{h,i} , \dotsc , I^{h,2}, I^{h,1}) =\Dop( \Dop^{(i-1)}(I^{h,i},I^{h,i-1} , \dotsc , I^{h,2}), I^{h,1}) \to  \eta^i \]
  and  also satisfies the assumptions of Lemma \ref{lm-D33}. This is then passed on to be used for the case of $i+1$ sequences.   

   This completes the proof of $\eta^h\to\eta$.
\end{proof}

 \begin{proof}[Proof of Theorem \ref{B-th1}]
 Introduce an $(n+1)$st parameter value $\rho_0\in(1,\rho_1)$.  
 By Lemma \ref{B-lm-erg},   the $\R_{\ge0}^{n+1}$-valued  $\Z$-indexed process  
\be\label{B-666}    \Busv{\rho_0,\dotsc,\rho_n}{\evec_1}{t}=
 \{(\Bus^{\rho_0}_{(k-1,t),(k,t)}, \Bus^{\rho_1}_{(k-1,t),(k,t)},\dotsc, \Bus^{\rho_n}_{(k-1,t),(k,t)})\}_{k\in\Z}
\ee  
  is stationary and  ergodic under translation  of the $k$-index and furthermore  $\Busv{\rho_0,\dotsc,\rho_n}{\evec_1}{t}$ has the same distribution as the sequence $\Busv{\rho_0,\dotsc,\rho_n}{\evec_1}{t-1}$ on the previous level $t-1$.   Lemma \ref{B-lm3} gives $\Busv{\rho_0,\dotsc,\rho_n}{\evec_1}{t}=\Dop(  \Busv{\rho_0,\dotsc,\rho_n}{\evec_1}{t-1}, \wb\Yw_t)$.    By the uniqueness given in  Theorem \ref{m:thm-eta},  the distribution of $\Busv{\rho_0,\dotsc,\rho_n}{\evec_1}{t}$ must be the invariant distribution $\mu^{(\rho_0,\dotsc,\rho_n)}$. 
  
 Let $\rho_0\searrow 1$.  By Lemma \ref{B-reco-lm},  almost surely, 
 \[   \lim_{\rho_0\searrow 1}  \Busv{\rho_0, \rho_1, \dotsc,\rho_n}{\evec_1}{t} \;=\; 
 \bigl\{\bigl(\Yw_{(k,t)}, \Bus^{\rho_1}_{(k-1,t),(k,t)},\dotsc, \Bus^{\rho_n}_{(k-1,t),(k,t)}\bigr)\bigr\}_{k\in\Z}, 
\] 
while Theorem \ref{murho-th1} gives the weak convergence $\mu^{(\rho_0,\rho_1,\dotsc,\rho_n)}\to\mu^{(1,\rho_1,\dotsc,\rho_n)}$ as $\rho_0\searrow 1$.
%
 \end{proof}

  The proof of Lemma \ref{B-lm3} below relies  on the iterative equations \eqref{m:IJ5}.  Since these equations can have solutions other than the one coming  from the queuing mapping,  additional conditions are needed as specified in Lemma \ref{lind-lm} in Appendix \ref{a:queue}. 

 \begin{proof}[Proof of Lemma \ref{B-lm3}]   We show that there is an event $\Omega_0$ of full probability on which the assumptions of Lemma \ref{lind-lm} hold for   the sequences  $(\wt I, J, I, \w)=(\Busv{\rho}{\evec_1}{t}, \Busv{\rho}{\evec_2}{t}, \Busv{\rho}{\evec_1}{t-1}, \wb\Yw_t)$ for {\it all} $\rho\in(1,\infty)$ and  $t\in\Z$. 
 
Assumption \eqref{lind-a1} requires 
 \[ \lim_{m\to-\infty}  \sum_{k=m}^0 (\Yw_{(k,t)}-\Bus^\rho_{(k,t-1),(k+1,t-1)})=-\infty  
 \quad \forall t\in\Z.  \]
This  holds almost surely simultaneously for all   $\rho$ in a dense countable subset of $(1,\infty)$.  By the monotonicity \eqref{B-mono} this extends  to all $\rho\in(1,\infty)$ on a single event of full probability.

 Utilizing  the recovery property \eqref{B-reco} and additivity \eqref{B-add}, 
 \begin{align*}
 &\Yw_{(k,t)} + \bigl(  \Bus^\rho_{(k-1,t-1),(k,t-1)} -   \Bus^\rho_{(k-1,t-1),(k-1,t)} \bigr)^+ \\
 &\qquad 
 =  \Bus^\rho_{(k-1,t),(k,t)}\wedge\Bus^\rho_{(k,t-1),(k,t)}
 + \bigl(  \Bus^\rho_{(k-1,t),(k,t)} -   \Bus^\rho_{(k,t-1),(k,t)} \bigr)^+ \\
 &\qquad = \Bus^\rho_{(k-1,t),(k,t)}  
 \end{align*}
and 
  \begin{align*}
 &\Yw_{(k,t)} + \bigl(  \Bus^\rho_{(k-1,t-1),(k-1,t)}- \Bus^\rho_{(k-1,t-1),(k,t-1)}  \bigr)^+ \\
 &\qquad =  \Bus^\rho_{(k-1,t),(k,t)}\wedge\Bus^\rho_{(k,t-1),(k,t)}
 + \bigl(  \Bus^\rho_{(k,t-1),(k,t)}- \Bus^\rho_{(k-1,t),(k,t)}  \bigr)^+ \\
 &\qquad = \Bus^\rho_{(k,t-1),(k,t)}  . 
 \end{align*}
These equations are valid for all $\rho$ and all $(k,t)$  on a single event of full probability because this is true of properties \eqref{B-reco} and   \eqref{B-add}.   Assumption \eqref{lind-a2} has been verified. 

 Lemma \ref{lind-lm3} implies that with probability one,   for all   $\rho$ in a dense countable subset of $(1,\infty)$,   $\Yw_{(k,t)}= \Bus^\rho_{(k,t-1),(k,t)}$ for infinitely many $k<0$.  Monotonicity \eqref{B-mono} and recovery \eqref{B-reco}    extend this property  to all $\rho\in(1,\infty)$  on the same event.
 \end{proof}

 \subsection{Triangular arrays and independent increments} \label{s:array}
  To extract further properties of the distribution $\mu^\rho$, we  develop an alternative  representation for $\eta=\Daop^{(n)}(I)$ of \eqref{m:eta5}.  
 Assume given $I=(I^1,\dotsc,I^n)\in \caY_{n}$.  
 Define arrays  $\{\eta^{i,j}: 1\le j\le i\le n\}$ and $\{\xi^{i,j}: 1\le j\le i\le n\}$ of elements of $\R_{\ge0}^\Z$  as follows. 
 The $\xi$ variables are passed from one $i$ level to the next. 
 \begin{enumerate}[(i)] \itemsep=4pt
\item For $i=1$ set  $\eta^{1,1}=I^1=\xi^{1,1}$. 
\item For $i=2,3,\dotsc,n$, 
\be\label{m:ar5} \begin{aligned}
 &\ \eta^{i,1}=I^i\\
 &\begin{cases}  \eta^{i,j} =\Dop(\eta^{i,j-1}, \xi^{i-1,j-1})\\[3pt]
  \xi^{i,j-1}= \Rop(\eta^{i,j-1}, \xi^{i-1,j-1})  \end{cases} 
 \quad\text{for } j=2, 3, \dotsc, i \\
 &\; \,\xi^{i,i}=\eta^{i,i}.  
 \end{aligned}\ee
 Step $i$ takes inputs from two sources:  from the outside it takes $I^i$, and from step $i-1$ it takes the    configuration  $\xi^{i-1,\, \rcbullet}=(\xi^{i-1, 1},  \xi^{i-1, 2}, \dotsc,  \xi^{i-1,i-2},  \xi^{i-1,i-1}=\eta^{i-1,i-1})$. 
 \end{enumerate}  
 
 Lemma \ref{lm-D14} ensures that the arrays are well-defined for $I\in \caY_{n}$.  
 The inputs $I^1,\dotsc,I^n$ enter the algorithm one by one in order.  If the process is stopped after the step  $i=m$ is completed for some $m<n$, it produces  the arrays for $(I^1,\dotsc,I^m)\in \caY_{m}$. 
 
 \begin{figure}
 \[
 \begin{matrix}
 \eta^{1,1} \\[2pt] 
 \eta^{2,1} &\eta^{2,2} \\[2pt]
  \eta^{3,1} &\eta^{3,2} & \eta^{3,3} \\
 \vdots& \vdots& \vdots&  \ddots \\[1pt]  
   \eta^{n,1} &\eta^{n,2} & \eta^{n,3}  & \dotsm&  \eta^{n,n}
  \end{matrix}  
  \qquad \qquad 
  \begin{matrix}
 \xi^{1,1} \\[2pt]
 \xi^{2,1} &\xi^{2,2} \\[2pt]
  \xi^{3,1} &\xi^{3,2} & \xi^{3,3} \\
 \vdots& \vdots& \vdots&  \ddots \\[1pt] 
   \xi^{n,1} &\xi^{n,2} & \xi^{n,3}  & \dotsm&  \xi^{n,n}
  \end{matrix}  
\] 
\caption{\small Arrays $\{\eta^{i,j}: 1\le j\le i\le n\}$ and $\{\xi^{i,j}: 1\le j\le i\le n\}$.   } 
\label{fig-ar} 
\end{figure}

  The arrays are illustrated in Figure \ref{fig-ar}.  The following properties of the arrays come from Lemmas  \ref{lm:eta5} and \ref{lm:ar4} and their proofs.  
 \begin{enumerate}[{\rm(i)}] \itemsep=4pt
 \item  The input of the $\Daop^{(n)}$-mapping  lies on the left edge of the $\eta$-array: $(\eta^{1,1},\dotsc, \eta^{n,1})=(I^1,\dotsc,I^n)$.   The output of the $\Daop^{(n)}$-mapping  lies on   the right-hand diagonal edges of both arrays: \[ (\eta^{1,1}, \eta^{2,2}, \dotsc, \eta^{n,n})=(\xi^{1,1}, \xi^{2,2}, \dotsc, \xi^{n,n})=\Daop^{(n)}(I^1,\dotsc,I^n)\;\sim\; \mu^{(\rho_1, \,\rho_{2},\dotsc, \,\rho_n)}. \]
 
 \item The $j$th column $(\eta^{j,j}, \eta^{j+1,j},  \dotsc, \eta^{n,j})$ of the $\eta$-array  has the product distribution  $\nu^{(\rho_j, \,\rho_{j+1},\dotsc, \,\rho_n)}$.   It is obtained from the $(j-1)$st column  $(\eta^{j,j-1}, \eta^{j+1,j-1},  \dotsc, \eta^{n,j-1})$ by  the mapping   \eqref{m:I4}   with  $\eta^{j-1,j-1}=\xi^{j-1,j-1}$ as the external driving weights. 
 
 \item  Row $(\xi^{i,1}, \xi^{i,2},  \dotsc, \xi^{i,i})$  of the $\xi$-array   has the product   distribution 
$\nu^{(\rho_1, \,\rho_{2},\dotsc, \,\rho_i)}$. 
 
 
 \end{enumerate}

 \begin{lemma}\label{lm:eta5}  Let   $I=(I^1,\dotsc,I^n)\in \caY_{n}$.  
Let  
$(\wt\eta^1,\dotsc,\wt\eta^n)=\Daop^{(n)}(I^1,\dotsc,I^n)$
be given by the   mapping   \eqref{m:eta5}.  Let $\{\eta^{i,j}\}$ be the array defined above.   Then $\wt\eta^i=\eta^{i,i}$ for $i=1,\dotsc,n$. 
 \end{lemma}  
  
 \begin{proof}  
It suffices to prove $\wt\eta^n=\eta^{n,n}$  because the same proof applies to all $i$.  The construction of the array can be reimagined as follows.  Start with 
$(\eta^{1,1}, \eta^{2,1},\dotsc, \eta^{n,1})=(I^1, I^2,\dotsc, I^n)$.  
Then for  $\ell= 2, 3, \dotsc ,n-1$    
iterate the following step that maps   the $(n-\ell+2)$-vector 
\[  (\eta^{n, \ell-1},  \eta^{n-1, \ell-1}, \dotsc, 
\eta^{\ell, \ell-1}, \eta^{\ell-1,\ell-1}) \]
to  the $(n-\ell+1)$-vector 
 \begin{align*}
   &(\eta^{n, \ell},  \eta^{n-1, \ell}, \dotsc, \eta^{\ell+1,\ell}, \eta^{\ell, \ell} ) 
   \\[4pt]   &\qquad 
=   \bigl( \Dop(\eta^{n, \ell-1}, \xi^{n-1,\ell-1}),  \Dop(\eta^{n-1, \ell-1},\xi^{n-2,\ell-1}), 
   \dotsc, 
   \Dop(\eta^{\ell+1,\ell-1}, \xi^{\ell,\ell-1}),  \Dop(\eta^{\ell, \ell-1}, \eta^{\ell-1,\ell-1}) \bigr). 
\end{align*}
The $\xi$-variables above satisfy 
\begin{align*}
\xi^{\ell,\ell-1}&=\Rop(\eta^{\ell, \ell-1}, \xi^{\ell-1,\ell-1})=\Rop(\eta^{\ell, \ell-1}, \eta^{\ell-1,\ell-1})\\
\xi^{\ell+1,\ell-1}&=\Rop(\eta^{\ell+1, \ell-1}, \xi^{\ell,\ell-1})\\
&\ \,\vdots\\
\xi^{n-1,\ell-1}&=  \Rop(\eta^{n-1, \ell-1},\xi^{n-2,\ell-1}). 
\end{align*}
Thus \eqref{m:D126}  
implies that 
\be\label{m:ar8} \begin{aligned}
&\Dop^{(n-\ell+2)}\bigl(\eta^{n, \ell-1},  \eta^{n-1, \ell-1}, \dotsc, 
\eta^{\ell, \ell-1}, \eta^{\ell-1,\ell-1}\bigr) \\
&\qquad=  \Dop^{(n-\ell+1)}\bigl(\eta^{n, \ell},  \eta^{n-1, \ell}, \dotsc, \eta^{\ell+1,\ell}, \eta^{\ell, \ell}\bigr). 
\end{aligned}\ee
In the derivation below, use the first line of \eqref{m:ar5} to replace each $I^i$ with $\eta^{i,1}$.  Then iterate \eqref{m:ar8}  from $\ell=2$   to $\ell=n-1$ to obtain  
\begin{align*}
\wt\eta^{n}&=\Dop^{(n)}(I^{n}, I^{n-1}, \dotsc,  I^3,  I^2, I^1) \\
&=  \Dop^{(n)}\bigl(\eta^{n, 1},  \eta^{n-1, 1}, \dotsc, \eta^{3,1}, \eta^{2,1}, \eta^{1, 1} \bigr)\\
&=  \Dop^{(n-1)}\bigl(\eta^{n, 2},  \eta^{n-1, 2}, \dotsc,  \eta^{3,2},\eta^{2,2} \bigr)\\
&=\dotsm= \Dop^{(3)}(\eta^{n,n-2}, \eta^{n-1,n-2}, \eta^{n-2,n-2}) = \Dop(\eta^{n,n-1}, \eta^{n-1,n-1})=\eta^{n,n}. 
\qedhere \end{align*} 
  \end{proof}  
  
 The next two lemmas describe the distributions of the arrays.

\begin{lemma} \label{lm:ar4}    Fix  $0<\rho_1<\dotsm<\rho_n$ 
 and let the    multiline configuration $I=(I^1, \dotsc, I^n)$ have distribution 
$\nu^{(\rho_1,\dotsc, \rho_n)}$.   Let  $\{\eta^{i,j}\}_{1\le j\le i\le n}$ and $\{\xi^{i,j}\}_{1\le j\le i\le n}$  be the arrays defined above.  
   Then for each   $1\le i, j\le n$,  configuration  $(\eta^{j,j}, \eta^{j+1,j},  \dotsc, \eta^{n,j})$  has   distribution 
$\nu^{(\rho_j, \,\rho_{j+1},\dotsc, \,\rho_n)}$ and configuration  $(\xi^{i,1}, \xi^{i,2},  \dotsc, \xi^{i,i})$  has   distribution 
$\nu^{(\rho_1, \,\rho_{2},\dotsc, \,\rho_i)}$.   In particular,    each $\eta^{i,j}$ has distribution $\nu^{\rho_i}$  and each    $\xi^{i,j}$ has distribution $\nu^{\rho_j}$.
\end{lemma}

\begin{proof}   First we prove the claim for $(\eta^{j,j}, \eta^{j+1,j},  \dotsc, \eta^{n,j})$.   Recall that by definition $\xi^{j,j}=\eta^{j,j}$. 

  For $j=1$, the definitions give   $(\xi^{1,1}=\eta^{1,1}, \eta^{2,1},  \dotsc, \eta^{n,1})=(I^1, I^2,\dotsc, I^n)\sim\nu^{(\rho_1, \,\rho_{2},\dotsc, \,\rho_n)}$.   
  
  Let $j\in\lzb2,n\rzb$.  Assume inductively that  
\[  \bigl(\xi^{j-1,j-1}=\eta^{j-1,j-1}, \eta^{j,j-1},  \dotsc, \eta^{n,j-1}\bigr)\;\sim\; \nu^{(\rho_{j-1}, \,\rho_{j},\dotsc, \,\rho_n)}. \] 
  The mapping from  $(\eta^{j,j-1},  \dotsc, \eta^{n,j-1})$ to  $(\eta^{j,j},   \dotsc, \eta^{n,j})$   is   the mapping   \eqref{m:I4} of the multiline process,   with  $\xi^{j-1,j-1}$ as the external driving weights $\w$.   Namely, this mapping is carried out by iterating 
\[  \begin{cases}  
\eta^{j+k,j}=\Dop(\eta^{j+k,j-1}, \xi^{j+k-1,j-1}) \\
\xi^{j+k,j-1} =\Rop(\eta^{j+k,j-1}, \xi^{j+k-1,j-1})
\end{cases}
\quad \text{ for }  k=0,1,\dotsc, n-j.  \]
Then  $(\eta^{j,j}, \eta^{j+1,j},  \dotsc, \eta^{n,j})\sim \nu^{(\rho_j, \,\rho_{j+1},\dotsc, \,\rho_n)}$ follows  from the invariance in Theorem \ref{m:thm-I}.  

\medskip 

Next the proof for $(\xi^{i,1}, \xi^{i,2},  \dotsc, \xi^{i,i})$. 
 The claim is  immediate for $i=1$ because there is just one sequence $\eta^{1,1}=I^1=\xi^{1,1}\sim\nu^{\rho_1}$.  Let $i\in\lzb 2,n\rzb$ and assume inductively that 
 $(\xi^{i-1,1}, \xi^{i-1,2},  \dotsc, \xi^{i-1,i-1})\sim\nu^{(\rho_1, \,\rho_{2},\dotsc, \,\rho_{i-1})}$. 
 By construction,  $\eta^{i,1}=I^i\sim\nu^{\rho_i}$ is independent of $\xi^{i-1,\rcbullet}$, and hence 
 \[   (\eta^{i,1}, \xi^{i-1,1}, \xi^{i-1,2},  \dotsc, \xi^{i-1,i-1})\sim\nu^{(\rho_i, \,\rho_1, \,\rho_{2},\dotsc, \,\rho_{i-1})}. \]
 Now we transform the sequence above  by repeated application of  the mapping $(\eta^{i,\ell}, \xi^{i-1,\ell})\mapsto (\xi^{i,\ell}, \eta^{i,\ell+1})$ defined by \eqref{m:ar5}: 
 \[  \begin{cases}  
\eta^{i,\ell+1}=\Dop(\eta^{i,\ell}, \xi^{i-1,\ell}) \\
\xi^{i,\ell} =\Rop(\eta^{i,\ell}, \xi^{i-1,\ell})  \end{cases} 
\quad \text{ for }  \ell=1,\dotsc, i-1.   
 \]
 The  pair to be transformed next   slides successively  to the right.  The succession of sequences produced by this process is displayed below, beginning with the first one from above.    The pair to which the mapping is applied next is enclosed in the box.  The distribution follows from Lemma \ref{m:Lem-I}.
 \begin{align*}
&\bigl( \; \boxed{ \eta^{i,1}, \xi^{i-1,1} } \, , \xi^{i-1,2},  \xi^{i-1,3}, \dotsc, \,\xi^{i-1,i-1}\bigr)  \;  \sim\; \nu^{(\rho_i, \,\rho_1, \,\rho_{2}, \, \rho_3,\dotsc, \,\rho_{i-1})} \\
&\bigl( \,\xi^{i,1},\, \boxed{ \eta^{i,2}, \xi^{i-1,2} } \, , \, \xi^{i-1,3}, \dotsc, \,\xi^{i-1,i-1} \bigr)   \;  \sim\; \nu^{(\rho_1, \,\rho_i, \,\rho_{2},\, \rho_3, \dotsc, \,\rho_{i-1})}   \\
&   \qquad \qquad \qquad \  \  \ddots \\  
&\bigl(\, \xi^{i,1}, \dotsc, 
 \xi^{i,\ell-1}, \,\boxed{\eta^{i,\ell} ,\xi^{i-1, \ell}}\, , \, \xi^{i-1, \ell+1}, \dotsc, \,\xi^{i-1,i-1} \bigr)   \;  \sim\; \nu^{(\rho_1, \dotsc, \rho_{\ell-1}, \,\rho_i, \,\rho_{\ell},\,\rho_{\ell+1}\dotsc, \,\rho_{i-1})} \\
 &   \qquad \qquad   \ddots \\
&\bigl( \, \xi^{i,1}, \dotsc,  
 \, \xi^{i,i-1},\, \eta^{i,i} \,\bigr)  \;  \sim\; \nu^{(\rho_1, \dotsc,   \,\rho_{i-1} , \, \rho_i)}. 
 \end{align*}
 To complete the induction from $i-1$ to  $i$, set $\xi^{i,i}= \eta^{i,i}$. 
\end{proof}

\begin{remark}[Notation]  \label{rm:notn}   To keep track of the inputs when   processes are constructed by queueing mappings \eqref{m:DSR},    superscripts indicate the arrival and service processes used in the construction.    This works  as follows when the arrival process is $I$ and the service process is $\w$. 
\begin{itemize}
\item  $G^I$ denotes a function that satisfies $I_k=G^I_k-G^I_{k-1}$. 
\item $\wt G^{I,\,\w}$ is the process defined by \eqref{m:800} whose increments are the output $\wt I^{I,\,\w}_k=\wt G^{I,\,\w}_k-\wt G^{I,\,\w}_{k-1}$, and so   $\wt I^{I,\,\w}=\Dop(I,\w)$. 
\item  $J^{I,\,\w}=\Sop(I,\w)$ is the process defined by \eqref{m:J} as $J^{I,\,\w}_k=\wt G^{I,\,\w}_k-G^I_k$. 
\item $\wt\w^{I,\,\w}=\Rop(I,\w)$.  \qedrm
\end{itemize} 
\end{remark}

\smallskip 

\begin{lemma}\label{lm:eta12}    Fix  $0<\rho_1<\dotsm<\rho_n$ 
 and let the    multiline configuration $I=(I^1, \dotsc, I^n)$ have distribution 
$\nu^{(\rho_1,\dotsc, \rho_n)}$.    Let $\eta=(\eta^1,\dotsc,\eta^n)=\Daop^{(n)}(I)$ and let $\{\eta^{i,j}\}$ and $\{\xi^{i,j}\}$ be the arrays constructed above.    Then for each $m\in\lzb 2,n\rzb$ and $k\in\Z$, the following random variables are independent: 
\[ 
\{\xi^{m,1}_i\}_{i\le k}, \{\xi^{m,2}_i\}_{i\le k}, \dotsc, \{\xi^{m,m-1}_i\}_{i\le k}, \{\eta^m_i\}_{i\le k-1}, \,\eta^m_k-\eta^{m-1}_k, \eta^{m-1}_k-\eta^{m-2}_k, \dotsc,   \eta^2_k-\eta^1_k, \eta^1_k.
\]
\end{lemma} 

\begin{proof}  Index $k$ is fixed throughout the proof.   We begin with the case $m=2$. 

\medskip

By the definitions,   $\eta^1=I^1$, 
\[  \xi^{2,1}=\Rop(\eta^{2,1}, \xi^{1,1})=\Rop(I^2, I^1)=\wt\w^{I^2,I^1} \ \text{ and }\ \eta^2=\Dop(I^2, I^1)=\wt I^{I^2,I^1}. \]   
Hence 
$\xi^{2,1}_i=I^2_i\wedge J^{I^2,I^1}_{i-1}$ and 
$\eta^2_k-\eta^1_k=(I^2_k-J^{I^2,I^1}_{k-1})^+$.   
By Lemma \ref{m:Lem-I}(a), 
$\{\wt I^{I^2,I^1}_i\}_{i\le k-1}$,  $J^{I^2,I^1}_{k-1}$,  $\{\wt\w^{I^2,I^1}_i\}_{i\le k-1}$, $I^2_k$, $I^1_k$ are independent.    To be precise,  Lemma \ref{m:Lem-I}(a) gives the independence of $\{\wt\w^{I^2, I^1}_i\}_{i\le k-1}$,  $\{\wt I^{I^2, I^1}_i\}_{i\le k-1}$, and  $J^{I^2, I^1}_{k-1}$.   These are functions of $\{I^2_i, I^1_i\}_{i\le k-1}$, and thereby  independent of $I^2_k, I^1_k$.   Properties of independent exponentials (Lemma \ref{v:lm}(i)) imply that 
\be\label{m:ar26} 
\begin{aligned} 
&\text{$\xi^{2,1}_k=I^2_k\wedge J^{I^2, I^1}_{k-1}$ and $ \eta^2_k-\eta^1_k=   \bigl(I^2_k-J^{I^2, I^1}_{k-1}\bigr)^+$ are mutually independent. } 
\end{aligned} \ee
Altogether we have  that  $\{\xi^{2,1}_i\}_{i\le k}$, $\{\eta^2_i\}_{i\le k-1}$,   $\eta^2_k-\eta^1_k$, $\eta^1_k$ are independent. 

\medskip

%
%

Let $m\ge 3$ and make an induction assumption: 
\be\label{m:ar44} 
\begin{aligned}
&\text{$\{\xi^{m-1,1}_i\}_{i\le k}, \dotsc,\{\xi^{m-1,m-2}_i\}_{i\le k}$, $\{\eta^{m-1}_i\}_{i\le k-1}$, } \\
 &\text{$\eta^{m-1}_k-\eta^{m-2}_k,\dotsc,\eta^2_k-\eta^1_k, \eta^1_k$ are independent. }
\end{aligned}\ee
 The previous paragraph verified   this assumption for $m=3$. 

Since $\eta^{m,1}=I^m$ is independent of all the variables in \eqref{m:ar44},  apply Lemma \ref{m:Lem-I}(a) to the pair  $\xi^{m,1}=\Rop(\eta^{m,1}, \xi^{m-1,1})$,  $\eta^{m,2}=\Dop(\eta^{m,1}, \xi^{m-1,1})$ to conclude the independence of 
\be\label{m:ar46} 
\begin{aligned}
& \{\xi^{m,1}_i\}_{i\le k},   \{\eta^{m,2}_i\}_{i\le k},  \{\xi^{m-1,2}_i\}_{i\le k}, \dotsc,\{\xi^{m-1,m-2}_i\}_{i\le k},   \\
 &\{\eta^{m-1}_i\}_{i\le k-1}, \eta^{m-1}_k-\eta^{m-2}_k,\dotsc,\eta^2_k-\eta^1_k, \eta^1_k. 
\end{aligned}\ee
This starts an  induction on $j=2, 3, \dotsc, m-1$, whose induction assumption is the independence of 
\be\label{m:ar48} 
\begin{aligned}
& \{\xi^{m,1}_i\}_{i\le k}, \dotsc, \{\xi^{m,j-1}_i\}_{i\le k},    \{\eta^{m,j}_i\}_{i\le k},  \{\xi^{m-1,j}_i\}_{i\le k}, \dotsc,\{\xi^{m-1,m-2}_i\}_{i\le k},  \\
 &\{\eta^{m-1}_i\}_{i\le k-1},  \eta^{m-1}_k-\eta^{m-2}_k,\dotsc,\eta^2_k-\eta^1_k, \eta^1_k. 
\end{aligned}\ee
The induction step is the application of Lemma \ref{m:Lem-I}(a) to the pair  $\xi^{m,j}=\Rop(\eta^{m,j}, \xi^{m-1,j})$,  $\eta^{m,j+1}=\Dop(\eta^{m,j}, \xi^{m-1,j})$ to conclude the independence of 
\be\label{m:ar50} 
\begin{aligned}
& \{\xi^{m,1}_i\}_{i\le k}, \dotsc, \{\xi^{m,j-1}_i\}_{i\le k},   \{\xi^{m,j}_i\}_{i\le k},   \{\eta^{m,j+1}_i\}_{i\le k},  \{\xi^{m-1,j+1}_i\}_{i\le k},  \\
 &\dotsc,\{\xi^{m-1,m-1}_i\}_{i\le k}, \{\eta^{m-1}_i\}_{i\le k-1},  \eta^{m-1}_k-\eta^{m-2}_k,\dotsc,\eta^2_k-\eta^1_k, \eta^1_k. 
\end{aligned}\ee
Thus the  induction assumption \eqref{m:ar48} for $j$ has been advanced to $j+1$ in \eqref{m:ar50}. 

At the end of the $j$-induction we have the independence of 
\be\label{m:ar52} 
\begin{aligned}
& \{\xi^{m,1}_i\}_{i\le k}, \dotsc, 
  \{\xi^{m,m-2}_i\}_{i\le k},   \{\eta^{m,m-1}_i\}_{i\le k},   \\
 &  \{\eta^{m-1}_i\}_{i\le k-1}, \, \eta^{m-1}_k-\eta^{m-2}_k,\dotsc,\eta^2_k-\eta^1_k, \eta^1_k. 
\end{aligned}\ee
Split $ \{\eta^{m,m-1}_i\}_{i\le k}$ into the independent pieces $ \{\eta^{m,m-1}_i\}_{i\le k-1}$ and $\eta^{m,m-1}_k$.     Combine the former with $\{\eta^{m-1}_i\}_{i\le k-1}$,  Lemma \ref{m:Lem-I}(a), and the transformations 
 $\xi^{m,m-1}=\Rop(\eta^{m,m-1}, \eta^{m-1})$,  $\eta^{m}=\Dop(\eta^{m,m-1}, \eta^{m-1})$   to form the independent variables $\{\xi^{m,m-1}_i\}_{i\le k-1}$, $\{\eta^{m}_i\}_{i\le k-1}$,  $J^{\eta^{m,m-1}, \eta^{m-1}}_{k-1}$.    Transform the independent pair $(\eta^{m,m-1}_k,  J^{\eta^{m,m-1}, \eta^{m-1}}_{k-1})$  into 
the  independent pair  of $\xi^{m,m-1}_k= \eta^{m,m-1}_k\wedge J^{\eta^{m,m-1}, \eta^{m-1}}_{k-1}$ and $\eta^m_k-\eta^{m-1}_k=(\eta^{m,m-1}_k-J^{\eta^{m,m-1}, \eta^{m-1}}_{k-1})^+$.    Attach $\xi^{m,m-1}_k$ to the sequence $\{\xi^{m,m-1}_i\}_{i\le k-1}$.  After these steps, we have  the independence of 
\be\label{m:ar54} 
\begin{aligned}
& \{\xi^{m,1}_i\}_{i\le k}, \dotsc,    \{\xi^{m,m-2}_i\}_{i\le k},   \{\xi^{m,m-1}_i\}_{i\le k},   \\
 &  \{\eta^{m}_i\}_{i\le k-1}, \,  \eta^m_k-\eta^{m-1}_k, \eta^{m-1}_k-\eta^{m-2}_k,\dotsc,\eta^2_k-\eta^1_k, \eta^1_k. 
\end{aligned}\ee
Thus the induction assumption \eqref{m:ar44} has been advanced from $m-1$ to $m$. 
\end{proof}


\begin{proof}[Proof of Theorem \ref{B-th5}]
Fix  $1<\rho_1<\dotsm<\rho_n$ 
 and let  the   multiline configuration $I=(I^0, \dotsc, I^n)$ have distribution 
$\nu^{(1, \rho_1,\dotsc, \rho_n)}$.    Let $\eta=(\eta^0,\dotsc,\eta^n)=\Daop^{(n+1)}(I)$. By Theorem \ref{B-th1} proved above,  
$(\wb\Yw_t, \Busv{\rho_1}{\evec_1}{t}, \dotsc, \Busv{\rho_n}{\evec_1}{t})\deq\eta\sim\mu^{(1, \rho_1,\dotsc,\rho_n)}$.   Lemma \ref{lm:eta12} gives the independence of the components of the vector 
\begin{align*}
 &( \eta^1_k,    \eta^2_k-\eta^1_k, \dotsc, \eta^{n}_k-\eta^{n-1}_k)  \\
 &\qquad \deq \, \bigl(  \Yw_x,  \Bus^{\rho_1}_{x-\evec_1,x}-\Yw_x, \Bus^{\rho_2}_{x-\evec_1,x}-\Bus^{\rho_1}_{x-\evec_1,x}\,, \dotsc, \Bus^{\rho_{n}}_{x-\evec_1,x}-\Bus^{\rho_{n-1}}_{x-\evec_1,x} \bigr). 
 \end{align*} 
 (Above $k\in\Z$ and $x\in\Z^2$ are arbitrary.) 
 
 The distribution  of an increment $\eta^m_k-\eta^{m-1}_k$  can be  computed from the 2-component mapping $(\eta^{m-1}, \eta^m)=\Daop^{(2)}(I^{m-1}, I^m)=(I^{m-1}, \Dop(I^m, I^{m-1}))$ where  $(I^{m-1}, I^{m})\sim\nu^{\rho_{m-1},\rho_m}$.   The  first equation of \eqref{m:IJ5} gives
 \[   \eta^m_k-\eta^{m-1}_k =  \eta^m_k-I^{m-1}_k = (I^m_k- J^{I^m, I^{m-1}}_{k-1})^+ . \]  
The right-hand side has the   distribution in \eqref{B-incr13} with $(\lambda,\rho)=(\rho_{m-1},\rho_m)$ because,  by the structure of the queueing mapping,  $I^m_k$ and  $J^{I^m, I^{m-1}}_{k-1}$  are independent exponentials with parameters $\rho_m^{-1}$ and $\rho_{m-1}^{-1}-\rho_m^{-1}$.   

A computation  of the Laplace transform of the increment $X(\rho)-X(\lambda)$  of the process   defined by \eqref{X1} gives,  for  $\rho>\lambda\ge 1$ and  $\ttvar>0$, 
\begin{align*}
E[e^{-\ttvar (X(\rho)-X(\lambda))} ] =\frac{1+\lambda\ttvar}{1+\rho\ttvar} . 
\end{align*}
This is the Laplace transform of the  distribution in \eqref{B-incr13}.    
 Thus $\eta^m_k-\eta^{m-1}_k$ has the same distribution as   $X(\rho_{m})-X(\rho_{m-1})$.  

To summarize,  the  nondecreasing cadlag processes $\Bus^{\,\!\rcbullet}_{x-\evec_1,x}$ and $X(\abullet)$  have identically distributed   initial values (both $\Bus^1_{x-\evec_1,x}=\Yw_x$ and $X(1)$ are   Exp$(1)$-distributed) and   identically distributed  independent   increments.  Hence the processes are equal in distribution.    
\end{proof}

\subsection{Bivariate Busemann process on a line} 
The remainder of this section proves statements  for the sequence   $\{(\Bus^\lambda_{(k-1,t),(k,t)}, \Bus^\rho_{(k-1,t),(k,t)})\}_{k\in\Z}$ that has distribution   $\mu^{(\lambda,\rho)}$.   We use this notation. 
\be\label{nota-88}\begin{aligned}
&\text{Let  $\rho>\lambda>0$, $(I^1,I^2)\sim\nu^{(\lambda,\rho)}$ and $(\eta^1,\eta^2)=\Daop^{(2)}(I^1,I^2)=(I^1, \Dop(I^2,I^1))$.} \\
&\text{Then $(\eta^1,\eta^2)\sim\mu^{(\lambda,\rho)}$.    Let $J=J^{I^2,I^1}=\Sop(I^2,I^1)$. }
\end{aligned}\ee

\begin{proof}[Proof of Theorem \ref{B-th8}]
The next auxiliary lemma identifies a reversible Markov chain.

\begin{lemma}\label{X-lm1}  Let $X_i=J_{i-1}-I^2_i$.   Then   $\{X_i\}_{i\in\Z}$ and  $\{X_i^+\}_{i\in\Z}$ are stationary reversible Markov chains.   $\{X_i^-\}_{i\in\Z}$ is not a  Markov chain.

\end{lemma} 

\begin{proof}
From the second equation of \eqref{m:IJ5}, 
\begin{align*}
X_{i+1} = J_i-I^2_{i+1} = I^1_i+(J_{i-1}-I^2_i)^+-I^2_{i+1} = X_i^+ +I^1_i -I^2_{i+1}. 
\end{align*}
Since $J_{i-1}$ is a function of $(I^1_k, I^2_k)_{k\le i-1}$,  $X_i$ is independent of $(I^1_i, I^2_{i+1})$.   Schematically, we can express the transition probability as 
$X_{i+1}=X_i^++ {\rm Exp}(\lambda^{-1})-{\rm Exp}(\rho^{-1})$, where the three terms on the right-hand side are independent. 

Similarly, using conservation \eqref{m:cons} and  the dual equations \eqref{IJ101}, 
\begin{align*}
X_{i} &=J_{i-1}-I^2_i=   J_i-\eta^2_{i} = \wt\w_{i+1}+(J_{i+1}-\eta^2_{i+1})^+-\eta^2_{i} 
= X_{i+1}^+ +\wt\w_{i+1} -\eta^2_{i}. 
\end{align*}
$J_i$ and  $\eta^2_i$ are independent by Lemma \ref{m:Lem-I}(a), and hence the triple 
$(J_i, \eta^2_i, I^2_{i+1})$ is independent.  Consequently so is the triple 
$(X_{i+1}, \wt\w_{i+1} , \eta^2_{i})=(J_i-I^2_{i+1}, J_i\wedge I^2_{i+1},  \eta^2_{i})$  and we can express the equation above as 
  $X_{i}=X_{i+1}^++ {\rm Exp}(\lambda^{-1})-{\rm Exp}(\rho^{-1})$ where again  the three terms on the right-hand side are independent. 
    The transitions from $X_i$ to $X_{i+1}$ and back are the same.
    
 From the equations above we obtain equations that show  $X_i^+$ as  a reversible Markov chain.
 
 Writing temporarily $U_i=I^1_{i-1} -I^2_{i}$, we get these equations for $X_{i+1}^-$: 
 \[  X_{i+1}^-   =( X_i^+ +U_{i+1})^- =  \bigl( ( X_{i-1}^+ +U_{i})^+ +U_{i+1}\bigr)^- .  
 \]
 
 Conditioned on $X_i\ge 0$,   $X_i\sim$ Exp$(\lambda^{-1}-\rho^{-1})$.  Thus 
 \be\label{X67} 
 P(X_{i+1}^-=0\,\vert\,X_i^-=0)= P(X_i^++U_{i+1}\ge 0\,\vert\,X_i\ge 0)=
 P\bigl\{{\rm Exp} (\lambda^{-1}-\rho^{-1}) +U_{i+1}\ge 0\bigr\}.  
 \ee 
For the next calculation, note that $X_{i-1}<0$ implies  $X_{i}  =X_{i-1}^+ +U_{i}=U_i$ and then $X_{i+1}   = U_i^+ +U_{i+1}$. 
 \be\label{X69} \begin{aligned}
& P(X_{i+1}^-=0\,\vert\,X_i^-=0, X_{i-1}^->0)
  = \frac{P(X_{i+1}^-=0, X_i^-=0, X_{i-1}^->0)}{P(X_i^-=0, X_{i-1}^->0)}\\
  &\quad
 =   \frac{P(U_i+U_{i+1}\ge0, U_i\ge 0, X_{i-1}<0)}{P(U_i\ge 0, X_{i-1}<0)}
 = P(U_i+U_{i+1}\ge0\,\vert\, U_i\ge 0) \\
 &\quad  =P\bigl\{{\rm Exp} (\lambda^{-1}) +U_{i+1}\ge 0\bigr\}.  
 \end{aligned}\ee
 We used  above the independence of $X_{i-1}$ from $(U_i, U_{i+1})$ and then the conditional distribution  $U_i\sim$ Exp$(\lambda^{-1})$, given that $U_i\ge0$.  
 The conditional distributions in \eqref{X67} and \eqref{X69} do not agree, and consequently 
 $X_i^-$ is not a Markov chain. 
\end{proof}

Since $\eta^2_k-\eta^1_k=\eta^2_k-I^1_k=(I^2_k-J_{k-1})^+=X_k^-$, we   conclude that $\eta^2_k-\eta^1_k$ is not a Markov chain, but it is a function of a reversible Markov chain.  
Part (a) of Theorem \ref{B-th8} has been proved. 

\medskip 

Two more auxiliary lemmas.

\begin{lemma}   The process  $(\eta^1_k,\eta^2_k)_{k\in\Z}$ is not a Markov chain. 
\end{lemma}

\begin{proof}     The construction gives 
$  \eta^2_{k+1}=I^1_{k+1}+(I^2_{k+1}-J_k)^+ . $  
On the right, the variables $I^1_{k+1},I^2_{k+1}$ are independent and independent of $J_k, (\eta^1_j,\eta^2_j)_{j\le k}$. 
 The conclusion of the lemma follows from showing that conditioning on $\eta^1_k=\eta^2_k$ gives  $J_k$ an unbounded distribution,  while conditioning on $\eta^1_{k-1}<\eta^2_{k-1}$ and $\eta^1_k=\eta^2_k$ implies  $J_k\le \eta^1_{k-1}+\eta^1_{k}$.   Thus conditioning on $(\eta^1_k,\eta^2_k)$ does not completely decouple $\eta^2_{k+1}$ from the earlier past.  

From  the three independent variables  $(J_{k-1}, I^1_k, I^2_k)$ the queueing formulas  define 
\be\label{56:90}   \eta^1_k=I^1_k, \quad    \eta^2_{k}=I^1_{k}+(I^2_{k}-J_{k-1})^+ \quad\text{and}\quad 
J_k=I^1_{k}+(J_{k-1}-I^2_{k})^+  .  
\ee
The condition $\eta^1_k=\eta^2_k$ is equivalent to $J_{k-1}\ge I^2_k$,  and conditioning on this implies $J_{k-1}-I^2_{k}\sim$ Exp$(\lambda^{-1}-\rho^{-1})$. Thus  $J_k$ is unbounded. 

For the second scenario consider the five independent variables $(J_{k-2}, I^1_{k-1}, I^2_{k-1}, I^1_k, I^2_k)$ and augment \eqref{56:90} with the equations of the prior step: 
\be\label{56:93}   \eta^1_{k-1}=I^1_{k-1}, \quad    \eta^2_{k-1}=I^1_{k-1}+(I^2_{k-1}-J_{k-2})^+ \quad\text{and}\quad 
J_{k-1}=I^1_{k-1}+(J_{k-2}-I^2_{k-1})^+  . 
\ee
Now  $\eta^1_{k-1}<\eta^2_{k-1}$ implies  $J_{k-1}=I^1_{k-1}$  and then  $\eta^1_k=\eta^2_k$  implies  $J_k=I^1_{k}+J_{k-1}-I^2_{k} =I^1_{k}+I^1_{k-1}-I^2_{k}$.   Hence  $J_k\le I^1_{k}+I^1_{k-1}=\eta^1_{k}+\eta^1_{k-1}$.   The lemma is proved. 

With service process $I^1=\eta^1$, arrival process $I^2$ and departure process $\eta^2$, the queueing explanation of the proof is that    $\eta^1_k=\eta^2_k$ implies that  customer $k$ had to wait before entering service, and hence delays from the past can influence the  next interdeparture time $\eta^2_{k+1}$.  
\end{proof}

\begin{lemma}   The pair $((\eta^1_k,\eta^2_k),(\eta^1_{k+1},\eta^2_{k+1}))$ and its transpose $((\eta^1_{k+1},\eta^2_{k+1}),(\eta^1_k,\eta^2_k))$ are not equal in distribution. 
\end{lemma}

\begin{proof}   By the queueing construction, $\eta^1_{k+1}=I^1_{k+1}$ is independent of $(\eta^1_k,\eta^2_k)$ because the latter pair is a function of $(I^1_i,I^2_i)_{i\le k}$.   To see  that   $\eta^1_k=I^1_k$ is not independent of $(\eta^1_{k+1},\eta^2_{k+1})$, write 
\begin{align*}
\eta^2_{k+1}-\eta^1_{k+1} &= (I^2_{k+1}-J_k)^+
= \bigl(I^2_{k+1}- I^1_k-[J_{k-1} - I^2_k]^+\,\bigr)^+
\end{align*}
where all four variables in the last expression are independent. 
\end{proof}

Part (b) of Theorem \ref{B-th8} follows from the two lemmas above. 
\end{proof}

\begin{proof}[Proof of Theorem \ref{B-th10} and Remark \ref{B-rm10}]  
Part (a) comes from translating the condition $\Bus^\rho_{(k-1)\evec_1, k\evec_1}=\Yw_{k\evec_1}$ into a statement about the queueing mapping  $\wt I=\Dop(I,\w)$.

By \eqref{bg6},
\begin{align*}
\P\{\counta_x=0\}=\P\{ \Bus^\rho_{x-\evec_2, x} < \Bus^\rho_{x-\evec_1, x}\} 
=1-\rho^{-1}
\end{align*}
from the independence and exponential distributions in Theorem \ref{B-th-zero}(ii).  From \eqref{B-incr13}, 
\begin{align*}
\P\{\counta^{\lambda,\rho}_x=0\}=\P\{ \Bus^\rho_{x-\evec_1, x} > \Bus^\lambda_{x-\evec_1, x}\} 
=\frac{\rho-\lambda}\rho.  
\end{align*}

 
 To calculate  $\P\{\counta^{\lambda,\rho}_x=n\}$ for $n\ge 1$ we put $x$ on the $x$-axis and 
 use  the distribution  $(\Busv{\lambda}{\evec_1}{0},  \Busv{\rho}{\evec_1}{0})\deq (\eta^1,\eta^2)\sim\mu^{(\lambda,\rho)}$ given by Theorem \ref{B-th1}, with the convention   from \eqref{nota-88}.   By setting $\lambda=1$ the same calculation gives $\P\{\counta_x=n\}$ because  $\wb\Yw_0=  \Busv{1}{\evec_1}{0}$. 
\be\label{bg40} \begin{aligned}
\P\{\counta^{\lambda,\rho}_{n\evec_1}=n\}  
&=\P\bigl\{   \Bus^\rho_{-\evec_1, 0}>\Bus^\lambda_{-\evec_1, 0}, \, \Bus^\rho_{0, \evec_1}=\Bus^\lambda_{0, \evec_1},   \, \Bus^\rho_{\evec_1, 2\evec_1}=\Bus^\lambda_{\evec_1, 2\evec_1}, \\
&\qquad\qquad\qquad 
 \dotsc,  \Bus^\rho_{(n-1)\evec_1, n\evec_1}=\Bus^\lambda_{(n-1)\evec_1, n\evec_1}\bigr\}  \\
&=P\bigl\{   \eta^2_0>I^1_0, \,   \eta^2_1=I^1_1,  \,  \eta^2_2=I^1_2, \dotsc,   \eta^2_n=I^1_n
\bigr\} \\
&=P\bigl\{   I^2_0>J_{-1}, \,   I^2_1\le J_0,  \,  I^2_2\le J_1, \dotsc,   I^2_n\le J_{n-1} \bigr\} \\
\end{aligned}\ee
The last equality used  $\eta^2_i=I^1_i+(I^2_i-J_{i-1})^+$ repeatedly:   $\eta^2_i>I^1_i$ is equivalent to $I^2_i>J_{i-1}$.

Next apply repeatedly the equation  $J_i=I^1_i+(J_{i-1}-I^2_i)^+$ inside the last probability in \eqref{bg40}.     $ I^2_0>J_{-1}$ implies $J_0=I^1_0$.   Then  $I^2_1\le J_0$ implies  $J_1=I^1_1+J_0-I^2_1=I^1_1+I^1_0-I^2_1$. Assume inductively that 
\be\label{bg41}  J_i=I^1_i+\dotsm+I^1_0-I^2_1-\dotsm-I^2_i .  \ee
    Then $I^2_{i+1}\le J_i$  implies 
\begin{align*}
J_{i+1}= I^1_{i+1}+ J_i-I^2_{i+1} =  I^1_{i+1}+ (I^1_i+\dotsm+I^1_0-I^2_1-\dotsm-I^2_i)-I^2_{i+1}  
\end{align*}
 and the induction goes from $i$ to $i+1$.     Substitute \eqref{bg41} for $J_0,\dotsc, J_{n-1}$ in the last probability in \eqref{bg40}. Use the independence of the variables $J_{-1}, \{ I^1_i, I^2_i\}_{i\ge0}$.   Let $S^\alpha_m$ denote the sum of $m$ i.i.d.\ Exp$(\alpha)$ random variables, with $S$ and $\wt S$  denoting independent sums. 
  \be\label{bg42} \begin{aligned}
\P\{\counta^{\lambda,\rho}_{n\evec_1}=n\}
&=P\bigl\{   I^2_0>J_{-1}, \,   I^2_1\le I^1_0,  \,  I^2_2 +I^2_1\le I^1_1+I^1_0, \\
 &\qquad\qquad \qquad  
  \dotsc,   I^2_n+\dotsm+I^2_1\le I^1_{n-1}+\dotsm+I^1_0 \bigr\} \\
&=P\bigl\{   I^2_0>J_{-1}\bigr\} \,P\bigl\{  S^{\rho^{-1}}_m\le \wt S^{\lambda^{-1}}_m\ \forall m\in[n]\bigr\}  \\
&=\frac{\rho-\lambda}{\rho} \sum_{k=0}^{n-1} C(n-1,k) \frac{\rho^k\lambda^n}{(\lambda+\rho)^{n+k}}. 
\end{aligned}\ee
The last line  comes from the independence of $I^2_0$ and $J_{-1}$, their distributions $I^2_0\sim$ Exp$(\rho^{-1})$ and $J_{-1}\sim$ Exp$(\lambda^{-1}-\rho^{-1})$, and Lemma \ref{lm:poi8}. 
\end{proof}


\appendix

\section{Queues}
\label{a:queue}

We prove elementary   lemmas about the  queueing  mappings.   Unless otherwise stated,   the weights   are real numbers without any probability distributions.  

\begin{lemma} \label{lm-D13}    
 Fix $0\le a<b$. Let  $I=(I_k)_{k\in\Z}$ and $\w=(\w_j)_{j\in\Z}$ in $\R_{\ge0}^\Z$ satisfy 
\be\label{D88.2} 
\varliminf_{m\to-\infty} \frac1{\abs m} \sum_{i=m}^0 I_i \ge b  
\quad\text{and}\quad   
\lim_{m\to-\infty} \frac1{\abs m} \sum_{i=m}^0 \w_i =a.   
\ee
  Then  $\wt I=\Dop(I,\w)$ is well-defined and satisfies   $\varliminf_{m\to-\infty} {\abs m}^{-1}  \sum_{i=m}^0\wt  I_i \ge b$. 
\end{lemma} 

\begin{proof}
Assumption \eqref{Iw} is obviously satisfied. 
Without loss of generality   assume $G_0=0$.    Let  $0<\e<(b-a)/3$.   Then for large enough $n$, 
\begin{align*}
\wt G_{-n} &= \sup_{k:\,k\le -n}  \Bigl\{  G_k+\sum_{i=k}^{-n}  \w_i\Bigr\}
=\sup_{k:\,k\le -n}  \Bigl\{  -\sum_{i=k+1}^0 I_i +\sum_{i=k}^{0} \w_i\Bigr\}  
- \sum_{i=-n+1}^{0} \w_i  \\
&\le  \sup_{k:\,k\le -n}  \bigl\{ -\abs k(b-\e) +\abs k(a+\e)\bigr\} -n(a-\e) = n(-b+3\e). 
  \end{align*}
Since $\sum_{i=m+1}^0\wt  I_i =  \wt G_0- \wt G_m$ this proves   $\ddd\varliminf_{m\to-\infty} {\abs m}^{-1}  \sum_{i=m}^0\wt  I_i \ge b$.  
 \end{proof} 

%
%
%
%
 
 \begin{lemma} \label{lm-D33}    
 Fix $0\le a<b$.   Assume given nonnegative real  sequences   $I=(I_i)_{i\in\Z}$, $\w=(\w_i)_{i\in\Z}$,  $I^{(h)}=(I^{(h)}_i)_{i\in\Z}$ and $\w^{(h)}=(\w^{(h)}_i)_{i\in\Z}$ where    $h\in\Z_{>0}$ is an index.  Assume the following.  
 $I^{(h)}_i\to I_i$ and $\w^{(h)}_i\to\w_i$ as $h\to\infty$ for all $i\in\Z$, and furthermore 
\be\label{D88} 
\varlimsup_{\substack{m\to-\infty\\h\to\infty}}    \,\biggl\lvert \frac1{\abs m}\sum_{i=m}^0 I^{(h)}_i-b\,\biggr\rvert=0 
\quad\text{and}\quad   
\varlimsup_{\substack{m\to-\infty\\h\to\infty}}    \,\biggl\lvert\frac1{\abs m}\sum_{i=m}^0 \w^{(h)}_i -a\,\biggr\rvert=0.   
\ee
Then   $\wt I=\Dop(I,\w)$ and $\wt\w=\Rop(I,\w)$ are well-defined, as are 
   $\wt I^{(h)}=\Dop(I^{(h)},\w^{(h)})$ and  $\wt\w^{(h)}=\Rop(I^{(h)},\w^{(h)})$ for large enough $h$.   We have the limits   
   \be\label{D89} \lim_{h\to\infty}\wt I^{(h)}_i= \wt I_i 
\quad\text{and}\quad  \lim_{h\to\infty} \wt\w^{(h)}_i=\wt\w_i  
   \qquad  \forall i\in\Z \ee
   and   
\be\label{D90} 
\varlimsup_{\substack{m\to-\infty\\h\to\infty}}    \,\biggl\lvert \frac1{\abs m}\sum_{i=m}^0  \wt I^{(h)}_i-b\,\biggr\rvert=0 
\quad\text{and}\quad   
\varlimsup_{\substack{m\to-\infty\\h\to\infty}}    \,\biggl\lvert\frac1{\abs m}\sum_{i=m}^0 \wt\w^{(h)}_i -a\,\biggr\rvert=0.   
\ee
\end{lemma} 

 \begin{proof}
Assumption \eqref{Iw} is  satisfied to make   $\wt I^{(h)}=\Dop(I^{(h)},\w^{(h)})$  well-defined for large enough $h$. 

We can assume  $G^{(h)}_0=0$.  
 Compute $\wt I^{(h)}=\Dop(I^{(h)},\w^{(h)})$ as the increments of the function 
\be\label{D96}  
  \wt G^{(h)}_\ell =\sup_{k\le\ell}  \Bigl\{   G^{(h)}_{k}+ \sum_{i=k}^\ell \w^{(h)}_i\Bigr\} .   
\ee 

Let $k_0$ be a maximizer in \eqref{m:800} for $\wt G_\ell$.  Then 
\be\label{D98} 
\varliminf_{h\to\infty} \wt G^{(h)}_\ell \ge \varliminf_{h\to\infty}  \Bigl\{   G^{(h)}_{k_0}+ \sum_{i=k_0}^\ell \w^{(h)}_i\Bigr\}  =   G_{k_0}+ \sum_{i=k_0}^\ell \w_i= \wt G_\ell. 
\ee 

Let $k(h)$ be a maximizer in \eqref{D96}.  If  $\varlimsup_{h\to\infty} \wt G^{(h)}_\ell \le \wt G_\ell$ fails then it must be that $k(h)\to-\infty$ along a subsequence.    But we can write 
\be\label{D98.8}\begin{aligned}
 \wt G^{(h)}_\ell  &= G^{(h)}_{k(h)}+ \sum_{i=k(h)}^\ell \w^{(h)}_i \\ 
 &=   -  \sum_{i=k(h)+1}^0  I^{(h)}_{i}  +  \sum_{i=k(h)}^0 \w^{(h)}_i    + \biggl( \ind_{\ell>0} \sum_{i=1}^\ell \w^{(h)}_i  -  \ind_{\ell<0} \sum_{i=\ell+1}^0 \w^{(h)}_i \biggr) 
\end{aligned}\ee
which converges to $-\infty$ as $k(h)\to-\infty$  by the assumptions and thereby contradicts \eqref{D98}.    We have now proved that  
\be\label{D99}   \lim_{h\to\infty}  \wt G^{(h)}_\ell =  \wt G_\ell \qquad\forall \ell\in\Z \ee
and thereby verified \eqref{D89} for $\wt I^{(h)}$.

  Let  $0<\e<(b-a)/3$.   By assumption \eqref{D88}  there exist  finite $n_1(\e)$ and   $h_1(\e)$  such that, when  $n\ge n_1(\e)$ and  $h\ge h_1(\e)$, 
\begin{align*}
\wt G^{(h)}_{-n} &= \sup_{k:\,k\le -n}  \Bigl\{  G^{(h)}_k+\sum_{i=k}^{-n}  \w^{(h)}_i\Bigr\}
=\sup_{k:\,k\le -n}  \Bigl\{  -\sum_{i=k+1}^0 I^{(h)}_i +\sum_{i=k}^{0} \w^{(h)}_i\Bigr\}  
- \sum_{i=-n+1}^{0} \w^{(h)}_i  \\
&\le  \sup_{k:\,k\le -n}  \bigl\{ -\abs k(b-\e) +\abs k(a+\e)\bigr\} -n(a-\e) = n(-b+3\e). 
  \end{align*}
  From this
  \be\label{D98.2}   \varlimsup_{m\to-\infty}  \sup_{h\ge h_1(\e)} \frac{\wt G^{(h)}_m}{\abs m}  \le -b+3\e. \ee
Since $\sum_{i=m}^0\wt  I^{(h)}_i =  \wt G^{(h)}_0- \wt G^{(h)}_{m-1}$ and $\wt G^{(h)}_0\ge \w^{(h)}_0\ge 0$,  this proves   
  \be\label{D98.4} \varliminf_{m\to-\infty} \;  \inf_{h\ge h_1(\e)} \; \frac1{\abs m}\sum_{i=m}^0\wt  I^{(h)}_i \ge b-3\e.\ee 

For the complementary upper bound,  get a lower bound for $\wt G^{(h)}_{m-1}$ by taking $k=\ell$ in \eqref{D96}. 
\begin{align*}
\sum_{i=m}^0\wt  I^{(h)}_i &=  \wt G^{(h)}_0- \wt G^{(h)}_{m-1} \le  \wt G^{(h)}_0- G^{(h)}_{m-1}  
= \wt G^{(h)}_0+  \sum_{i=m}^0  I^{(h)}_i . 
\end{align*}
Apply  limit \eqref{D99} and  assumption \eqref{D88}.  Limit \eqref{D90}  has been proved for $\wt I^{(h)}$. 

The limits for $\wt\w^{(h)}$ follow from the other limits and the  generally valid identity 
\be\label{m:IJ8} \w_k+I_k=\wt\w_k+\wt I_k   \ee
that   comes from equations \eqref{m:R} and  \eqref{m:IJ5}. 
\end{proof} 

For reference elsewhere in the paper we state the simple consequence of Lemma \ref{lm-D33}  where the  sequences are constant functions of $h$.

\begin{lemma} \label{lm-D14}    
 Fix $0\le a<b$. Let  $I=(I_k)_{k\in\Z}$ and $\w=(\w_j)_{j\in\Z}$ in $\R_{\ge0}^\Z$ satisfy 
\be\label{D88.4} 
\lim_{m\to-\infty} \frac1{\abs m} \sum_{i=m}^0 I_i = b  
\quad\text{and}\quad   
\lim_{m\to-\infty} \frac1{\abs m} \sum_{i=m}^0 \w_i =a.   
\ee
  Then  $\wt I=\Dop(I,\w)$ and $\wt\w=\Rop(I,\w)$ are  well-defined and satisfy 
\be\label{D88.6} 
\lim_{m\to-\infty} \frac1{\abs m} \sum_{i=m}^0 \wt I_i = b  
\quad\text{and}\quad   
\lim_{m\to-\infty} \frac1{\abs m} \sum_{i=m}^0 \wt\w_i =a.   
\ee
\end{lemma}

For the purpose of verifying that Busemann functions obey the queueing operation $\wt I=\Dop(I,\w)$, it is convenient to have a lemma that deduces this from assuming the iterative equations \eqref{m:IJ5}.   The first lemma below makes a statement without randomness. 

\begin{lemma}\label{lind-lm}   Let $\{\wt I_k, J_k, I_k, \w_k\}_{k\in\Z}$ be  nonnegative real numbers  that satisfy the three assumptions   below: 
\begin{align}
\label{lind-a1} &\lim_{m\to-\infty}  \sum_{i=m}^0 (\w_i-I_{i+1})=-\infty. \\
 \label{lind-a2}  
 &\wt I_k=\w_k+(I_k-J_{k-1})^+
 \quad\text{and}\quad 
 J_k=\w_k+(J_{k-1}-I_k)^+  
\quad \forall k\in\Z.  \\[3pt] 
 \label{lind-a3}  &J_k=\w_k \quad\text{ for infinitely many $k<0$.}  
 \end{align}  
 
Then $\wt I=\Dop(I,\w)$ and $J=\Sop(I,\w)$.  
\end{lemma} 

\begin{proof}  Rewrite the second equation of \eqref{lind-a2} as follows.  Let $W_k=J_k-\w_k$ and  $U_k=\w_k-I_{k+1}$.  Then 
\be\label{m:IJ5.78}  W_k=(W_{k-1}+U_{k-1})^+. \ee
This is  Lindley's recursion from queueing theory and  $W_k$ is the {\it waiting time} of customer $k$.  
Equation \eqref{m:IJ5.78}  iterates inductively to give 
\be\label{m:IJ5.80} 
W_k=\biggl\{ \biggl( W_\ell+ \sum_{i=\ell}^{k-1} U_i\biggr)\bigvee  \biggl( \max_{m:\, \ell+1\le m\le k-1} \sum_{i=m}^{k-1} U_i\biggr)\biggr\}^+ \qquad \forall \,\ell< k.  
\ee 
%
We claim that 
\be\label{m:IJ5.85} 
W_k=   \biggl( \,\sup_{m: \, m\le k-1} \sum_{i=m}^{k-1} U_i\biggr)^+ \qquad \forall \,  k\in\Z.    
\ee 
Dropping the first term on the right in \eqref{m:IJ5.80} and letting $\ell\to-\infty$  gives $\ge$ in \eqref{m:IJ5.85}.  By assumption  \eqref{lind-a3} $W_\ell=0$ for some $\ell<k$.  Then \eqref{m:IJ5.80} gives also $\le$ in \eqref{m:IJ5.85}.   

The proof is completed by making explicit the content of \eqref{m:IJ5.85}.  Let $G$ and $\wt G$ be as defined in the definition of the mappings $\Dop$ and $\Sop$.  Then  from \eqref{m:IJ5.85} and \eqref{m:806} deduce 
\begin{align*}
J_k&=\w_k + \biggl( \,\sup_{m: \, m\le k-1} \sum_{i=m}^{k-1} (\w_i-I_{i+1})\biggr)^+ 
=  \w_k + \biggl( \,\sup_{m: \, m\le k-1}  \Bigl\{  G_m-G_k+  \sum_{i=m}^{k-1} \w_i \Bigr\}  \biggr)^+\\
&=\w_k + (\wt G_{k-1}-G_k)^+ = \w_k + \wt G_{k-1}\vee G_k - G_k =  \wt G_{k}-G_k. 
\end{align*} 
Thus $J=\Sop(I,\w)$.    Finally, from the first  equation of \eqref{lind-a2} and $I_k=G_k-G_{k-1}$, 
\begin{align*}
\wt I_k&=\w_k+(I_k-J_{k-1})^+ =   I_k-J_{k-1} + \w_k+ (J_{k-1}-I_k)^+ =  I_k +J_k-J_{k-1}\\
&= I_k +(\wt G_{k}-G_k) -(\wt G_{k-1}-G_{k-1}) =  \wt G_{k}- \wt G_{k-1}. 
\qedhere
\end{align*}
\end{proof} 

Here is a version for a random sequence.  

\begin{lemma}\label{lind-lm3}   Let $\{\wt I_k, J_k, I_k, \w_k\}_{k\in\Z}$ be finite nonnegative random variables that satisfy assumptions {\rm(i)--(iii)} below: 
\begin{enumerate}[{\rm(i)}]  \itemsep=3pt 
\item   
$\ddd\lim_{m\to-\infty}  \sum_{i=m}^0 (\w_i-I_{i+1})=-\infty$  almost surely. 

\item  $\{J_k, \w_k\}_{k\in\Z}$ is a stationary process. 
 
\item   Equations
  \be\label{m:IJ5.55}  
 \wt I_k=\w_k+(I_k-J_{k-1})^+
 \quad\text{and}\quad 
 J_k=\w_k+(J_{k-1}-I_k)^+   
  \ee
are valid for all $k\in\Z$, almost surely.

\end{enumerate}

Then  $J_k=\w_k$  for infinitely many $k<0$ with probability 1, and  $\wt I=\Dop(I,\w)$ and $J=\Sop(I,\w)$  almost surely.
\end{lemma} 

\begin{proof}      Lemma \ref{lind-lm}  gives the conclusion once  we  verify  that assumption \eqref{lind-a3} holds almost surely. 
Using the waiting time notation $W$ from the previous proof,   it suffices to show that 
\be\label{m:IJ5.87}
\P\{\text{$W_\ell=0$ for infinitely many $\ell<0$}\}=1 
\ee
The complementary event is $B=\{\exists m<0\text{ such that } W_k>0\,\forall k\le m\}$. $B$ is a shift-invariant event.   On the event $B$, the right-hand side of \eqref{m:IJ5.78} is strictly positive for all $k\le m$ (for a random $m$).  This   implies, for all $k<m$, 
\begin{align*}
  0 &< W_m= W_{m-1}+U_{m-1}=W_{m-2}+U_{m-2}+U_{m-1}=\dotsm= W_k + \sum_{i=k}^{m-1} U_i \\
  &= W_k + \sum_{i=k}^{m-1} (\w_i  - I_{i+1}) . 
\end{align*} 
By assumption (i) of the lemma,  $W_k\to\infty$ a.s.\ on the event $B$ as $k\to-\infty$.   Let  $c<\infty$.    By the shift-invariance of $B$ and the stationarity of the process $\{W_k=J_k-\w_k\}_{k\in\Z}$,  
\begin{align*}
\P( W_0\ge c, B)= \P( W_k\ge c, B)\to  \P(B)   \quad\text{as }  k\to-\infty.  
\end{align*}
We conclude that $W_0=\infty$ a.s.\ on the event $B$, and hence $\P(B)=0$.  Claim \eqref{m:IJ5.87}  has been verified. 
\end{proof} 

\begin{remark}[Non-stationary solution to Lindley's recursion]   Some result such as Lemma \ref{lind-lm3} is needed, for there can be  another solution to Lindley's recursion that blows up as $n\to-\infty$.  Suppose $\{U_k\}$ is ergodic and   $\E U_k<0$.  Pick any random $N$ such that 
$\sum_{k=m}^N U_k<0$ for all $m\le N$.    Set 
\begin{align*}
W_n&=-\sum_{k=n}^N U_k\qquad\text{for }n\le N\\
W_{N+1}&=0\\
W_n&=(W_{n-1}+U_{n-1})^+ \qquad\text{for } n\ge N+2.
\end{align*}
One can check that $W_n=(W_{n-1}+U_{n-1})^+$ holds for all $n\in\Z$. 
\qedrm\end{remark}

\section{Exponential distributions}
\label{a:exp} 

The next lemma is elementary. 
  The mapping $( I,J,W)\mapsto(I', J', W')$ in the lemma   is an involution, that is, its own inverse. 

 \begin{lemma}\label{v:lm}  Let $\alpha, \beta>0$.  
 Assume given independent variables $W\sim$ {\rm Exp$(\alpha+\beta)$},  $I\sim$ {\rm Exp$(\alpha)$}, and 
 $J\sim$ {\rm Exp$(\beta)$}.   Define 
 \be\label{IJw19}\begin{aligned}
I'&=W+(I-J)^+ \\
J'&=W+(I-J)^- \\
W'&=I\wedge J   .  
\end{aligned}\ee 
\begin{enumerate}[{\rm(i)}] \itemsep=3pt 
\item $I-J$ and $I\wedge J$ are independent. 
\item $(I-J)^+\sim$ {\rm Ber}$(\frac\beta{\alpha+\beta})\cdot {\rm Exp}(\alpha)$, that is, the product of a Bernoulli with success probability $\frac\beta{\alpha+\beta}$ and an independent  rate $\alpha$ exponential. 
\item The  triple  $( I', J', W')$  has the same distribution as  $(I,J,W)$. 

\end{enumerate}
 \end{lemma} 
 
 We use the previous lemma to establish some facts about the queueing operators.  To be consistent with the queueing discussion we   parametrize exponentials with their means  $\paramb$ and $\rho$.   
  
 \begin{lemma}\label{m:Lem-I}  Let $0<\paramb<\rho$.  Let $(I_k)_{k\in\Z}$ and $(\w_j)_{j\in\Z}$  be mutually independent random variables such that $I_k\sim$ {\rm Exp}$(\rho^{-1})$ and $\w_j\sim$ {\rm Exp}$(\paramb^{-1})$.  Let $\wt I=\Dop(I,\w)$ as defined by \eqref{m:800} and \eqref{m:801},    $\wt\w=R(I,\w)$ as defined by \eqref{m:R}, and $J_k=\wt G_k-G_k$ as in  \eqref{m:J}.      Let $\Lambda_k= (\{\wt I_j\}_{j\leq k}, \,J_k, \,\{\wt \w_j\}_{j\leq k})$.  
  	\begin{enumerate}[{\rm(a)}] \itemsep=3pt
 		 		\item   $\{\Lambda_k\}_{k\in\Z}$ is a stationary, ergodic process.  
 		    For each $k\in\Z$, the random variables $\{\wt I_j\}_{j\leq k}, \,J_k, \,\{\wt \w_j\}_{j\leq k} $ are mutually independent with marginal distributions  
	\[  \text{ $\wt I_j\sim$ {\rm Exp}$(\rho^{-1})$, $\wt\w_j\sim$ {\rm Exp}$(\paramb^{-1})$ and   $J_k\sim$ {\rm Exp}$(\paramb^{-1}-\rho^{-1})$.   } \] 

\item  $\wt I$ and $\wt\w$  are    independent sequences of i.i.d.\ variables.  
 	\end{enumerate}
 \end{lemma}
 
\begin{proof}   Part (b) follows from part (a) by dropping the $J_k$ coordinate and letting $k\to\infty$.    Stationarity and ergodicity of $\{\Lambda_k\}$ follow from its construction as a mapping applied to the independent i.i.d.\ sequences $I$ and $\w$.  

The distributional claims in part (a) are proved   by coupling $(\wt I_k, J_{k-1}, \wt\w_k)_{k\in\Z}$ with another sequence whose distribution we know.  
Construct a process $( \wh I_k, \wh J_{k-1}, \wh\w_k)_{k\ge 1}$ as follows.    First let $\wh J_0$ be an Exp$(\paramb^{-1}-\rho^{-1})$ variable that is independent of $(I,\w)$.    Then for $k=1,2,3,\dotsc$  iterate the steps 
\be\label{98.3} \begin{aligned}
\wh I_k&=\w_k+(I_k-\wh J_{k-1})^+  \\
\wh J_k&=\w_k+(\wh J_{k-1}-I_k)^+\\
\wh\w_k&=I_k\wedge\wh J_{k-1}.  
\end{aligned}\ee

We prove the following claim  by induction  for each $m\ge 1$: 
\be\label{98.7} \begin{aligned} 
&\text{variables $\wh I_1,\dotsc, \wh I_m, \wh J_m, \wh\w_1,\dotsc, \wh\w_m$ are mutually independent with}\\
 &\text{marginal distributions  $\wh I_k\sim$ {\rm Exp}$(\rho^{-1})$, $\wh J_m\sim$ {\rm Exp}$(\paramb^{-1}-\rho^{-1})$ and $\wh\w_j\sim$ {\rm Exp}$(\paramb^{-1})$. } 
\end{aligned} \ee

By construction the variables  $(I_1, \wh J_{0}, \w_1)$ are independent with distributions  $({\rm Exp}(\rho^{-1}),  {\rm Exp}(\paramb^{-1}-\rho^{-1}), {\rm Exp}(\paramb^{-1}))$.   The base case $m=1$ of \eqref{98.7}  comes by applying Lemma \ref{v:lm} to  the mapping \eqref{98.3} with $k=1$.   Now assume  \eqref{98.7}   holds for $m$.  Then $(I_{m+1}, \wh J_{m}, \w_{m+1})$ are independent with distributions  $({\rm Exp}(\rho^{-1}),  {\rm Exp}(\paramb^{-1}-\rho^{-1}), {\rm Exp}(\paramb^{-1}))$ because, by construction, $\wh J_{m}$ is a function of  $( I_1,\dotsc,  I_m, \wh J_0, \w_1,\dotsc, \w_m)$ and thereby independent of $(I_{m+1}, \w_{m+1})$.     By Lemma \ref{v:lm}, mapping \eqref{98.3} turns triple $(I_{m+1}, \wh J_{m}, \w_{m+1})$ into the  triple  $(\wh I_{m+1}, \wh J_{m+1}, \wh\w_{m+1})$  of  independent variables, which is also independent of $\wh I_1,\dotsc, \wh I_m,   \wh\w_1,\dotsc, \wh\w_m$.  Statement \eqref{98.7} has been extended  to $m+1$.  



Our next claim is that 
\be\label{98.8} \begin{aligned} 
&\text{there exists (almost surely a random index) $m_0\ge 0$ such that $J_{m_0}=\wh J_{m_0}$.  }  
\end{aligned}\ee  

Suppose first that $J_0\ge \wh J_0$.   Then   \eqref{m:IJ5}  and \eqref{98.3} imply that  $J_k\ge\wh J_k$ for all $k\ge 0$.   If \eqref{98.8} fails then  $J_k>\wh J_k$ for all $k\ge 0$.  But then  for all $k> 0$, 
\[  J_k=J_{k-1}+ \w_k-I_k=\dotsm= J_0+\sum_{j=1}^k (\w_j-I_j) \to -\infty 
\quad \text{almost surely, as }\; k\to\infty  
\]
which contradicts the fact that $J_k\ge 0$ for all $k$.   Thus in this case \eqref{98.8} happens.  
The case  $J_0\le \wh J_0$ is symmetric. 

Through equations   \eqref{m:IJ5}  and \eqref{98.3},  \eqref{98.8} implies that $\wt I_k=\wh I_k$,  $J_k=\wh J_k$,  and $\wt\w_k=\wh\w_k$ for all $k>m_0$.   Part (a) follows from \eqref{98.7}, because  for any $\ell$,  $(\wt I_{\ell-n},\dotsc, \wt I_\ell, J_\ell, \wt\w_{\ell-n},\dotsc, \wt\w_\ell)$ has the same distribution as $(\wt I_{k-n},\dotsc, \wt I_k, J_k, \wt\w_{k-n},\dotsc, \wt\w_k)$ which agrees with $(\wh I_{k-n},\dotsc, \wh I_k, \wh J_k, \wh\w_{k-n},\dotsc, \wh\w_k)$ with probability tending to one as $k\to\infty$. 
\end{proof}

Next we compute a competition   probability  
for  two independent homogeneous Poisson processes on $[0,\infty)$ with rates $\alpha$ and $\beta$. Let $\{\sigma_i\}_{i\geq 1}$ be the jump times of  the rate $\alpha$ Poisson process  and $\{\tau_i\}_{i\geq 1}$  the jump times of the  rate $\beta$ Poisson process.    For $n\ge 1$ define the events
\begin{align*}
A_n&=\{\sigma_i< \tau_i\,\text{ for all } i\in[n]\}\\
\text{ and }\quad B_n&=\{\sigma_i< \tau_i\,\text{ for all }  i\in[n-1], \; \sigma_n>\tau_n\}.
\end{align*}
The {\it Catalan numbers}  $\{C_n:\,n\geq 0\}$ are defined by 
\begin{equation}\label{Cat-nr}
C_n=\frac{1}{n+1}{2n\choose n}.
\end{equation}
The following properties of the Catalan triangle $\{C(n,k):\,0\leq k\leq n\}$ given in  \eqref{Cat-tr}  can be deduced with elementary arguments.    $C(n,0)=1$,  $C(n,k)=\binom{n+k}{k}-\binom{n+k}{k-1}$ for $k>0$,   $C(n,k)=C(n,k-1)+C(n-1,k)$, 
\be\label{Cat-tr3}   \sum_{k=0}^i C(n,k) = C(n+1,i) \quad \text{ for } 0\le i\le n, 
\ee
and 
\be\label{Cat-tr4}   \sum_{k=0}^n C(n,k) = C(n+1,n) = C(n+1,n+1)=C_{n+1}.   
\ee

\begin{lemma}  \label{lm:poi8}
For $n\geq 1$,  
\begin{align} \label{eq:A_n}   P(A_n)  
&= \sum_{k=0}^{n-1}C(n-1,k)\,\frac{\alpha^n\beta^k}{(\alpha+\beta)^{n+k}}  \\
 \label{eq:B_n}
\text{and} \qquad  P(B_n)
&= C_{n-1}\,  \frac{\alpha^{n-1}\beta^n}{(\alpha+\beta)^{2n-1}}. 
\end{align}
\end{lemma}


\begin{remark}
The generating function of the Catalan numbers  is  
$$f(x)=\sum_{n\geq 0}C_nx^n=\frac{1-\sqrt{1-4x}}{2x},\quad  \text{ for }  |x|\leq \tfrac{1}{4}.$$
Hence from \eqref{eq:B_n},
$$\sum_{n=1}^{\infty}P(B_n)=\tfrac{\beta}{\alpha+\beta}\,f\bigl(\tfrac{\alpha\beta}{(\alpha+\beta)^2}\bigr)  
=
\begin{cases}
1, & \text{ if }\beta\geq \alpha\\
\frac{\beta}{\alpha}, &\text{ if }\beta< \alpha.
\end{cases}
$$
In other words, the rate $\alpha$ process stays forever ahead of the rate $\beta$ process with probability $(1-\frac\beta\alpha)^+$.  
\qedrm\end{remark}

\begin{proof}   We compute $P(B_n)$ first and then obtain $P(A_n)$ by inclusion-exclusion.  

  Since $C_0=1$,    \eqref{eq:B_n} holds for $n=1$.  For   $n\geq 2$ condition on $(\sigma_n,\tau_n)$:  
\begin{align}\label{Cond_last}
P(B_n)= \int_{a>b>0}P_{(a,b)}\bigl\{U_i \leq V_i \text{ for }i\in[n-1]\bigr\}\,P((\sigma_n,\tau_n)\in d(a,b)),
\end{align}
where   under $P_{(a,b)}$, $0<U_1<\cdots<U_{n-1}$ are  the order statistics of $n-1$ i.i.d.\ uniform random variables on $[0,a]$ and $0<V_1<\cdots<V_{n-1}$ are the same on $[0,b]$, independent of the $\{U_i\}$.
We calculate the probability inside the integral. 

Below,  first use the equal probability of the permutations of $\{x_i\}$ among themselves and $\{y_j\}$ among themselves.  Note that  $a>b$ and the conditions  $x_i<y_i$ force all  $\{x_i, y_j\}$ to lie in $[0,b]$.   Then use the equal probability of all permutations of $\{x_i, y_j\}$ together.    The Catalan number $C_{k}$ is the number of permutations of $\{x_1,\dotsc, x_k, y_1,\dotsc, y_k\}$    such that 
$x_1<\cdots<x_{k}$, $ y_1<\cdots<y_{k}$  and  $x_i<y_i$ for all $i$  (see Cor.~6.2.3 and item {\bf dd} on Page 223 of \cite{stan-II}).
\begin{align*}
&P_{(a,b)}\bigl(U_i \leq V_i \text{ for }i\in[n-1]\bigr)
= \frac{((n-1)!)^2}{(ab)^{n-1}}  \int\limits_{\substack{x_1<\dotsm<x_{n-1}<a\\y_1<\dotsm<y_{n-1}<b}}  \ind_{x_i<y_i \,\forall i\in[n-1]} \,d\xvec \,d\yvec \\
&\qquad 
=   C_{n-1}  \frac{((n-1)!)^2}{(ab)^{n-1}}    \int\limits_{ x_1<\dotsm<x_{n-1}< y_1<\dotsm<y_{n-1}<b}    \,d\xvec \,d\yvec =  C_{n-1}  \frac{((n-1)!)^2}{(ab)^{n-1}}  \cdot \frac{b^{2(n-1)}}{(2n-2)!} . 
\end{align*}
Substitute this back into \eqref{Cond_last}.  Use the gamma densities of $\sigma_n$ and $\tau_n$.  
\begin{align*} 
P(B_n)
&= C_{n-1}  \frac{((n-1)!)^2}{(2n-2)!} \int_{0<b<a<\infty}\frac{b^{n-1}}{a^{n-1}}\cdot \frac{(\alpha a)^{n-1} }{\Gamma(n)}\,\alpha e^{-\alpha a}\cdot\frac{(\beta b)^{n-1} }{\Gamma(n)} \,\beta e^{-\beta b}\,da\,db\\[3pt]
&=  C_{n-1}  \frac{ \alpha^n\beta^{n}}{(2n-2)!}  \int_{0<b<a<\infty}b^{2n-2}\,e^{-\alpha a -\beta b}\,da\,db   
\;= \; C_{n-1}\,\frac{\alpha^{n-1}\,\beta^{n}}{(\beta+\alpha)^{2n-1}}. 
\end{align*}

 \medskip 
 
We  prove \eqref{eq:A_n}.   The case $n=1$ is elementary.   Let $n\ge 2$ and assume \eqref{eq:A_n} for $n-1$. 
Abbreviate  $p=\frac{\beta}{\alpha+\beta}$ and $q=\frac{\alpha}{\alpha+\beta}$.    Use below \eqref{Cat-tr3} and \eqref{Cat-tr4}.
\begin{align*}
P(A_n) &= P(A_{n-1})-P(B_n)  
= q^{n-1}\sum_{k=0}^{n-2}C(n-2,k)\,p^k-C_{n-1}q^{n-1}\,p^n\\
&= q^{n-1} \sum_{k=0}^{n-2}C(n-2,k)\,(p^k -p^n)  
= q^{n} \sum_{k=0}^{n-2}\,\sum_{j=k}^{n-1}C(n-2,k)\,p^{j}\\
&= q^{n} \sum_{j=0}^{n-1}\,\sum_{k=0}^{j\wedge(n-2)}C(n-2,k)\,p^{i}  
= q^n \sum_{j=0}^{n-1}C(n-1, j\wedge(n-2))p^j =  q^n \sum_{j=0}^{n-1}C(n-1, j)p^j. 
\qedhere 
\end{align*}
\end{proof}

\bigskip

\small  

\bibliographystyle{abbrv}

\bibliography{PaperCGM}

\end{document}